\documentclass{amsart}
\usepackage{amssymb}

\newcommand{\w}{\omega}
\newcommand{\e}{\varepsilon}
\newcommand{\supp}{\mathrm{supp}}
\newcommand{\A}{\mathcal A}
\newcommand{\C}{\mathcal C}
\newcommand{\M}{\mathcal M}
\newcommand{\cov}{\mathrm{cov}}

\newcommand{\IN}{\mathbb N}
\newcommand{\E}{\mathcal E}
\newcommand{\F}{\mathcal F}
\newcommand{\IT}{\mathbb T}
\newcommand{\IC}{\mathbb C}
\newcommand{\IQ}{\mathbb Q}
\newcommand{\add}{\mathrm{add}}
\newcommand{\cof}{\mathrm{cof}}
\newcommand{\non}{\mathrm{non}}
\newcommand{\I}{\mathcal{I}}
\newcommand{\N}{\mathcal{N}}
\newcommand{\J}{\mathcal{J}}
\newcommand{\pr}{\mathrm{pr}}
\newcommand{\Ker}{\mathrm{Ker}}
\newcommand{\IZ}{\mathbb Z}
\newcommand{\pack}{\mathrm{pack}}
\newcommand{\Ipack}{{\mathcal I\mbox{-}\mathrm{pack}}}
\newcommand{\Ra}{\Rightarrow}

\newcommand{\HH}{\mathcal H}
\newcommand{\Hom}{\mathrm{Hom}}
\newcommand{\K}{\mathcal K}
\newcommand{\Inf}{\mathrm{Inf}}
\newcommand{\Sup}{\mathrm{Sup}}

\newtheorem{theorem}{Theorem}[section]
\newtheorem{proposition}[theorem]{Proposition}
\newtheorem{lemma}[theorem]{Lemma}
\newtheorem{claim}[theorem]{Claim}
\newtheorem{corollary}[theorem]{Corollary}
\newtheorem{question}[theorem]{Question}
\newtheorem{example}[theorem]{Example}
\newtheorem{problem}[theorem]{Problem}
\theoremstyle{definition}

\newtheorem{remark}[theorem]{Remark}

\title{The Solecki submeasures and densities on groups}
\author{Taras Banakh}
\address{Ivan Franko National University of Lviv (Ukraine) and Jan Kochanowski University of Kielce (Poland)}
\email{t.o.banakh@gmail.com}
\subjclass{28C10; 22C05; 05D10; 28A12; 05E15; 20F24; 43A05; 43A07}
\keywords{Invariant submeasure, group, Solecki submeasure, Haar measure, compact topological group}
\thanks{The first author has been partially financed by NCN grant  DEC-2011/01/B/ST1/01439.}

\begin{document}
\begin{abstract}
By definition, the {\em right Solecki density} $\sigma^R$ (resp. the {\em Solecki submeasure} $\sigma$) on a group $G$ is the invariant monotone (subadditive) function assigning to each subset $A\subset G$ the real number
$\sigma^R(A)=\inf_{F\in[G]^{<\w}}\sup_{y\in G}\frac{|F\cap Ay|}{|F|}$ (resp. $\sigma(A)=\inf_{F\in[G]^{<\w}}\sup_{x,y\in G}\frac{|F\cap xAy|}{|F|}$). In this paper we study the properties of the Solecki submeasures and Solecki densities on (topological) groups and
establish an interplay between the Solecki submeasure $\sigma$ and the Haar measure $\lambda$ on a compact topological group $G$.  In particular, we prove that that every subset $A\subset G$ has $\max\{\lambda_*(A),\lambda(A^\bullet)\}\le\sigma(A)\le\lambda(\bar A)$ where $A^\bullet$ is the largest open set in $G$ such that $A^\bullet\setminus A$ is meager in $G$. So, $\lambda$ and $\sigma$ coincide on the family of all closed subsets of $G$ and hence the Haar measure $\lambda$ is completely determined by the Solecki submeasure $\sigma$. On the other hand, for any amenable group $G$ the right Solecki density $\sigma^R$ coincides with the upper Banach density $d^*$ well-known in Combinatorics of Groups. The right Solecki density yields a convenient tool for studying the difference sets $AA^{-1}$ and sumsets $AB$ of subsets $A,B$ in groups. Generalizing results of Jin, Beiglb\"ock, Bergelson and Fish, for any subsets $A,B\subset G$ of positive right Solecki density $\sigma^R(A)$ and $\sigma^R(B)$ in an amenable group $G$ we prove that (1) $G=FAA^{-1}$ for some set $F\subset G$ of cardinality $|F|\le 1/\sigma^R(A)$,
(2) the sets $AA^{-1}BB^{-1}$ and $ABB^{-1}A^{-1}$ contain some Bohr open subset $U\ni 1_G$ of $G$,
(3) $B^{-1}AA^{-1}$ contains some non-empty Bohr open set $U$ in $G$, (4) $AA^{-1}\supset U\setminus N$ for some Bohr open set $U\ni 1_G$ in $G$ and some set $N\subset G$ with $\sigma^R(N)=0$, (5) $AB\supset U\cap T$ for some non-empty Bohr open set $U$ in $G$ and some set $T\subset G$ with $\sigma^R(T)=1$.
\end{abstract}
\maketitle
\tableofcontents

\newpage
\section*{Introduction}

In this paper we consider invariant densities and submeasures on groups and define a canonical invariant submeasure $\sigma$ (called the Solecki submeasure) on each group $G$, and four canonical invariant densities $\sigma_L$, $\sigma^L$, $\sigma^R$, $\sigma_R$ (called the Solecki densities) on $G$. Then we shall study the properties of the Solecki submeasure and densities on (topological) groups, and establish the interplay between the Solecki submeasure $\sigma$ and the Haar measure $\lambda$ on a compact topological group and also the interplay between the right Solecki density and the upper Banach density on an amenable group. The obtained results allow us to generalize some fundamental results of
Bogoliuboff, F\o lner \cite{Folner}, Cotlar and Ricabarra \cite{CR}, Ellis and Keynes \cite{EK} concerning the difference sets $AA^{-1}$ and Jin \cite{Jin}, Beiglb\"ock, Bergelson and Fish \cite{BBF} about sumsets $AB$ to the class of all amenable groups.

\section{Submeasures and densities on sets and groups}

A function $\mu:\mathcal P(X)\to[0,1]$ defined on the algebra of all subsets of a set $X$ is called
\begin{itemize}
\item {\em monotone} if $\mu(A)\le \mu(B)$ for any subsets $A\subset B\subset X$;
\item {\em subadditive} if $\mu(A\cup B)\le\mu(A)+\mu(B)$ for any subsets $A,B\subset X$;
\item {\em additive} if $\mu(A\cup B)=\mu(A)+\mu(B)$ for any disjoint subsets $A,B\subset X$;
\item a {\em density} if $\mu$ is monotone, $\mu(\emptyset)=0$ and $\mu(X)=1$;
\item a {\em submeasure} if $\mu$ is a subadditive density;
\item a {\em measure} if $\mu$ is an additive density.
\end{itemize}
So, all measures considered in the paper are in fact finitely additive probalility measures.

Each point $x\in X$ supports the {\em Dirac measure} $\delta_x$ defined by
$$\delta_x(A)=\begin{cases}
1,&\mbox{$x\in A$},\\
0,&\mbox{$x\notin A$}.
\end{cases}
$$

A submeasure $\mu$ on a set $X$ is {\em finitely supported\/} if $\mu(X\setminus F)=0$ for a suitable finite set $F\subset X$. It is well-known that each finitely supported probability measure $\mu$ on $X$ can be written as a convex combination $\mu=\sum_{i=1}^n\alpha_i\delta_{x_i}$ of Dirac measures. A finitely supported measure $\mu$ is called {\em uniformly distributed} if $\mu=\frac1{|F|}\sum_{x\in F}\delta_x$ for some non-empty finite subset $F\subset X$.

For a set $X$ we denote by  $[X]^{<\w}$ the family of all non-empty finite subsets of $X$, by $P(X)$ the set of all measures on $X$, by $P_\w(X)$ the subset of $P(X)$ consisting of all finitely supported  measures on $X$, and by $P_u(X)$ the set of all uniformly distributed finitely supported measures on $X$. The letter $P$ in those notations comes from the fact that all measures we consider are probability measures, i.e., assign measure 1 to $X$.

For each function $f:X\to Y$ and a density $\mu$ on $X$ we can define its image $f(\mu)$ as the density on $Y$ assigning to each subset $A\subset Y$ the real number $\mu(f^{-1}(A))$.

For a set $X$ by $|X|$ we denote its cardinality and for two sets $A,B$ by $A\triangle B$ their symmetric difference $(A\setminus B)\cup(B\setminus A)$.
For a group $G$ by $1_G$ we shall denote its unit.

For two finitely supported measures $\mu=\sum_i\alpha_i\delta_{a_i}$ and $\nu=\sum_j\beta_j\delta_{b_j}$ on a group $G$ their convolution $\mu*\nu$ is defined as $\mu*\nu=\sum_{i,j}\alpha_i\beta_j\delta_{a_ib_j}$. The convolution of two measures can be also defined if one of the measures is finitely supported while the other is not. Namely, for a measure $\mu\in P(G)$ and a finitely supported measure $\nu=\sum_{i}\alpha_i\delta_{a_i}\in P_\w(G)$ on a group $G$ their  convolutions $\mu*\nu$ and $\nu*\mu$ are the measures on $G$ defined by the formulas
$$\mu*\nu(A)=\sum_{i}\alpha_i\mu(Aa_i^{-1})\mbox{ and }\nu*\mu(A)=\sum_{i}\alpha_i\mu(a_i^{-1}A)\mbox{ for $A\subset G$}.$$
The operation of convolution is associative in the sense that $(\mu*\nu)*\eta=\mu*(\nu*\eta)$ for any measures $\mu,\nu,\eta\in P(G)$ among which at least two are finitely supported.

A density $\mu:\mathcal P(G)\to[0,1]$ on a group $G$ is called
\begin{itemize}
\item {\em left} (resp. {\em right}) {\em invariant} if $\mu(xA)=\mu(A)$ \textup{(}resp. $\mu(Ax)=\mu(A)$\textup{)} for all $A\subset G$ and $x\in G$;
\item {\em invariant} if $\mu(xAy)=\mu(A)$ for all $A\subset G$ and $x,y\in G$;
\item {\em inversion invariant} if $\mu$ is invariant and $\mu(A^{-1})=\mu(A)$ for all $A\subset G$.
\end{itemize}

A group $G$ is called {\em amenable} if it admits a left-invariant measure $\mu:\mathcal P(G)\to[0,1]$.
By \cite{Namioka}, a group $G$ is amenable if and only if it admits an inversely invariant measure.
The class of amenable groups contains all abelian groups and is closed under many operations over groups (see \cite{Pat}). On the other hand, the free group with two generators is not amenable. By the F\o lner criterion \cite[4.10]{Pat}, a group $G$ is amenable if and only if for every finite set $F\subset G$ and every $\e>0$ there is a finite set $K\subset G$ such that $|FK\setminus K|<\e|K|$.

It is well-known that the class of amenable group includes all FC-groups. A group $G$ is called an {\em $FC$-group} if each point $x\in G$ has finite conjugacy class $x^G=\{gxg^{-1}:g\in G\}$.
FC-groups were introduced by Baer \cite{Baer}. It is clear that each abelian group is an FC-group. By \cite{Neumann}, a finitely generated group $G$ is an FC-group if and only if $G$ is finite-by-abelian, i.e., $G$ contains a finite normal subgroup $H$ with abelian quotient $G/H$.

\section{The Solecki submeasure on a group}

Each group $G$ carries a canonical inversion invariant submeasure $\sigma:\mathcal P(G)\to[0,1]$ called the {\em Solecki submeasure}. It assigns to each subset $A\subset G$ the real number
$$\sigma(A)=\inf_{\mu\in P_\w(G)}\sup_{x,y\in G}\mu(xAy).$$
The Solecki submeasure was (implicitly) introduced by Solecki in \cite{Sol}.

\begin{proposition}\label{subad} For every group $G$ the Solecki submeasure $\sigma$ on a group $G$ is an inversion invariant submeasure on $G$.
\end{proposition}

\begin{proof} The definition of the Solecki submeasure implies that $\sigma$ is inversion invariant, monotone, and takes the values $\sigma(\emptyset)=0$ and $\sigma(G)=1$. It remains to prove that $\sigma$ is subadditive, i.e., $\sigma(A\cup B)\le\sigma(A)+\sigma(B)$ for any subsets $A,B\subset G$.

This inequality will follow as soon as we check that $\sigma(A\cup B)\le\sigma(A)+\sigma(B)+2\e$ for any $\e>0$.
By the definition of $\sigma(A)$ and $\sigma(B)$, there are finitely supported measures $\mu_A=\sum_i\alpha_i\delta_{a_i}$ and $\mu_B=\sum_{j}\beta_j\delta_{b_j}$ on $G$ such that $\sup_{x,y\in G}\mu_A(xAy)<\sigma(A)+\e$ and $\sup_{x,y\in G}\mu_B(xBy)<\sigma(B)+\e$. Consider the convolution measure $\mu=\mu_A*\mu_B=\sum_{i,j}\alpha_i\beta_j\delta_{a_ib_j}$
and observe that for every $x,y\in G$ the set $xAy$ has measure
$$
\mu(xAy)=\sum_{i,j}\alpha_i\beta_j\delta_{a_ib_j}(xAy)=
\sum_{j}\beta_j\sum_{i}\alpha_i\delta_{a_i}(xAyb_j^{-1})=
\sum_{j}\beta_j\mu_A(xAyb_j^{-1})<\sum_{j}\beta_j(\sigma(A)+\e)=
\sigma(A)+\e.
$$
By analogy we can prove that $\mu(xBy)<\sigma(B)+\e$. Then
$$\sigma(A\cup B)\le \sup_{x,y\in G}\mu\big(x(A\cup B)y\big)\le\sup_{x,y\in G}\big(\mu(xAy)+\mu(xBy)\big)< \sigma(A)+\sigma(B)+2\e.$$
\end{proof}

In Theorem 1.2 of \cite{Sol} Solecki proved that the Solecki submeasure can be equivalently defined using finite sets (instead of finitely supported probability measures).

\begin{theorem}[Solecki]\label{sol} Every subset $A$ of a group $G$ has Solecki submeasure
$$\sigma(A)=\inf_{\mu\in P_u(G)}\sup_{x,y\in G}\mu(xAy)=\inf_{F\in[G]^{<\w}}\sup_{x,y\in G}\frac{|F\cap xAy|}{|F|}.$$
\end{theorem}

Let us observe that the Solecki submeasure $\sigma$ can be also equivalently defined using convolutions of measures.

\begin{proposition}
Every subset $A$ of a group $G$ has Solecki submeasure
 $$\sigma(A)=\inf_{\mu\in P_\w(G)}\sup_{x,y\in G}\delta_x*\mu*\delta_y(A)=\inf_{\mu_1\in P_\w(G)}\sup_{\mu_2\in P_\w(G)}\sup_{\mu_3\in P_\w(G)}\mu_2*\mu_1*\mu_3(A).$$
 \end{proposition}

The Solecki submeasure is preserved by homomorphisms. The following proposition can be easily derived from the definition of the Solecki submeasure.

\begin{proposition}\label{p2.3} For any surjective homomorphism $h:G\to H$ between groups and any set $A\subset H$ we get $\sigma(h^{-1}(A))=\sigma(A)$.
\end{proposition}

\section{Left and right Solecki densities on a group}

In this section we introduce and study four left and right modifications of the Solecki submeasure, called the {\em Solecki densities}.

For a subset $A$ of a group $G$ the Solecki densities are defined by the formulas:
$$
\begin{aligned}
&\sigma^L(A)=\inf_{F\in[G]^{<\w}}\sup_{x\in G}\frac{|F\cap xA|}{|F|},&
&\sigma^R(A)=\inf_{F\in[G]^{<\w}}\sup_{y\in G}\frac{|F\cap Ay|}{|F|},&\\
&\sigma_L(A)=\inf_{\mu\in P_\w(G)}\sup_{x\in G}\mu(xA),&
&\sigma_R(A)=\inf_{\mu\in P_\w(G)}\sup_{y\in G}\mu(Ay).&
\end{aligned}
$$
It is clear that $\sigma_L\le\sigma^L\le\sigma$ and $\sigma_R\le\sigma^R\le\sigma$. Like the Solecki submeasure $\sigma$, the densities $\sigma_L,\sigma^L,\sigma_R,\sigma^R$ are invariant. In general, they  are not inversely invariant, but
$$\sigma^R(A^{-1})=\sigma^L(A)\mbox{ \  and \ }\sigma_R(A^{-1})=\sigma_L(A)$$for every subset $A\subset G$. If a subset $A\subset G$ is {\em inner invariant} (i.e., $xAx^{-1}=A$ for all $x\in G$), then all its Solecki densities coincide:
$$\sigma_L(A)=\sigma^L(A)=\sigma(A)=\sigma^R(A)=\sigma_R(A).$$

Because of the equalities $\sigma_L(A)=\sigma_R(A^{-1})$ and $\sigma^L(A)=\sigma^R(A^{-1})$, the study of the left densities $\sigma_L$ and $\sigma^L$ can be reduced to their right counterparts $\sigma_R$ and $\sigma^R$. So in the sequel we shall restrict ourselves to the right Solecki densities $\sigma_R$ and $\sigma^R$.

The following theorem was proved by Solecki in Theorems~1.1, 1.3, 5.1 \cite{Sol}.

\begin{theorem}[Solecki]\label{t-FC} Let $G$ be a group.
\begin{enumerate}
\item If $G$ is amenable, then $\sigma_L=\sigma^L$ and $\sigma_R=\sigma^R$.
\item If $G$ is an FC-group, then $\sigma_L=\sigma^L=\sigma=\sigma^R=\sigma_R$.
\item If $G$ is not an FC-group, then $G$ contains a subset $A\subset G$ such that $\sigma^L(A)<\sigma^R(A)=\sigma(A)=1$;
\item If $G$ contains a non-abelian free subgroup, then for every $\e>0$ the group $G$ contains a subset $A\subset G$ such that $\sigma_R(A)<\e$ and $\sigma^R(A)>1-\e$;
\item If $G$ is countable and contains a non-abelian free subgroup, then for every $\e>0$ the group $G$ contains a subset $A\subset G$ such that $\sigma_R(A)=0$ and $\sigma^R(A)>1-\e$.
\end{enumerate}
\end{theorem}

Unlike the Solecki submeasure $\sigma$ its modifications $\sigma_L$, $\sigma^L$, $\sigma_R$, $\sigma^R$ are not subadditive in general.

\begin{example}\label{not-ideal} The free group $F_2$ with two generators can be written as the union $F_2=A\cup B$ of two sets with $\sigma^R(A)=\sigma^R(B)=0$.
\end{example}

\begin{proof} Let $a,b$ be the generators of the free group $G=F_2$. Elements of the group $G$ can be written as irreducible words in the alphabet $\{a,b,a^{-1},b^{-1}\}$. The empty word $e$ is the unit of the group $G$. Let $A$ be the set of all irreducible words that start with $a$ or $a^{-1}$.
We claim that $\sigma^R(A)=0$. To show this, for every $n\in\IN$ consider the finite subset $F=\{b,b^2,\dots,b^n\}$ and observe that $|Fy\cap A|\le 1$ for every $y\in G$, which implies that $\sigma^R(A)\le \sup_{y\in G}|Fy\cap A|/|F|\le 1/n$ and hence $\sigma^R(A)=0$. By analogy we can show that the set $B=G\setminus A$ of irreducible words which are empty or start with $b$ or $b^{-1}$ has right Solecki density $\sigma^R(B)=0$.
\end{proof}

Nonetheless, the function $\sigma_R$ has a property which is weaker than the subadditivity.

\begin{proposition}\label{semiadditive} Let $G$ be a group. Then
\begin{enumerate}
\item $\sigma_R(A\cup B)\le\sigma_R(A)+\sigma(B)$ for any subsets $A,B\subset G$;
\item $\sigma^R(AF)\le|F|\cdot\sigma^R(A)$ and $\sigma_R(AF)\le|F|\cdot\sigma_R(A)$ for any subset $A\subset G$ and finite subset $F\subset G$.
\end{enumerate}
\end{proposition}

\begin{proof} The first statement can be proved by analogy with Proposition~\ref{subad}.
\smallskip

To prove the second statement, fix any subset $A\subset G$ and a finite subset $F\subset G$.
Given arbitrary $\e>0$, by the definition of $\sigma^R(A)$, find a uniformly distributed finitely supported measure $\mu\in P_u(G)$ such that $\sup_{z\in G}\mu(Az)<\sigma_R(A)+\e$.
Then $$\sigma^R(AF)\le\sup_{z\in G}\mu(AFz)\le \sum_{y\in F}\mu(Ayz)<\sum_{y\in F}(\sigma_R(A)+\e)=|F|\cdot(\sigma^R(A)+\e).$$
Since $\e>0$ was arbitrary, this implies the required inequality  $\sigma^R(AF)\le|F|\cdot\sigma^R(A)$.

By analogy we can prove the inequality $\sigma_R(AF)\le|F|\cdot\sigma_R(A)$.
\end{proof}

The density $\sigma_R$ has a nice characterization in terms of Kelley's intersection number. Following Kelley \cite{Kelly} we define the {\em intersection number} $I(\mathcal B)$ of a family $\mathcal B$ of subsets of a set $X$ as $$I(\mathcal B)=\inf_{B_1,...,B_n\in\mathcal B}\sup_{x\in X}\frac1n\sum_{i=1}^n\chi_{B_i}(x).$$
Here by $\chi_B:X\to\{0,1\}$ denotes the characteristic function of a set $B\subset X$.

We recall that by $P(X)$ we denote the family of all measures on a set $X$ and by $P_\w(X)$ the  set of all finitely supported  measures on $X$. The following minimax theorem was inspired by a result of Zakrzewski \cite{Zak}.

\begin{theorem}\label{t3.3nn} For every subset $A$ of a group $G$ we get
$$\inf_{\mu_1\in P_\w(G)}\sup_{\mu_2\in P_\w(G)}\mu_1*\mu_2(A)=\sigma_R(A)=I(\{xA\}_{x\in G})=\sup_{\mu\in P(G)}\inf_{x\in G}\mu(xA)=\sup_{\mu_2\in P(G)}\inf_{\mu_1\in P_\w(G)}\mu_1*\mu_2(A).$$
\end{theorem}

\begin{proof} It follows from the definition of $\sigma_R$ that $$\sigma_R(A)=\inf_{\mu\in P_\w(G)}\sup_{y\in G}\mu*\delta_y(A)=\inf_{\mu_1\in P_\w(G)}\sup_{\mu_2\in P_\w(G)}\mu_1*\mu_2(A).$$ To see that $\sigma_R(A)\le I(\{xA\}_{x\in G})$, it suffices to check that $\sigma_R(A)\le I(\{xA\}_{x\in G})+\e$ for every $\e>0$. By the definition of the intersection number, there is a sequence of points $x_1,\dots,x_{n}\in G$ such that $\frac1n\sup_{y\in G}\sum_{i=1}^n\chi_{x_iA}(y)<I(\{xA\}_{x\in G})+\e$.
Consider the finitely supported measure $\mu=\frac1n\sum_{i=1}^n\delta_{x^{-1}_i}$ and observe that for every $y\in G$
$$\mu(Ay)=\frac1n\sum_{i=1}^n\delta_{x^{-1}_i}(Ay)=\frac1n\sum_{i=1}^n\chi_{Ay}(x_i^{-1})=
\frac1n\sum_{i=1}^n\chi_{x_iA}(y^{-1})<I(\{xA\}_{x\in G})+\e$$ and hence $\sigma_R(A)\le\sup_{y\in G}\mu(Ay)\le I(\{xA\}_{x\in G})+\e$.
\smallskip

Next, we prove that $\sigma_R(A)=I(\{xA\}_{x\in G})$. In the opposite case,
$\sigma_R(A)<I(\{xA\}_{x\in G})-\e$ for some $\e>0$. By the definition of $\sigma_R(A)$, there exists a finitely supported probability measure $\mu$ on $G$ such that $\sup_{y\in G}\mu(Ay)<I(\{xA\}_{x\in G})-\e$.
The measure $\mu$ can be written as a convex combination of Dirac measures $\sum_{i=1}^k\alpha_i\delta_{y_i}$. Replacing each $\alpha_i$ by a near rational number, we can additionally assume that each $\alpha_i$ is a positive rational number. Moreover, we can assume that the numbers $\alpha_1,\dots,\alpha_k$ have a common denominator $n$. In this case the measure $\mu=\sum_{i=1}^k\alpha_i\delta_{y_i}$ can be written as $\mu=\sum_{i=1}^n\frac1n\delta_{x_i}$ for some points $x_1,\dots,x_n\in\{y_1,\dots,y_k\}$. Then
$$
\begin{aligned}
I(\{xA\}_{x\in G})&\le \frac1n\sup_{y\in G}\sum_{i=1}^n\chi_{x_i^{-1}A}(y)=\frac1n\sup_{y\in G}\sum_{i=1}^n\chi_{Ay^{-1}}(x_i)=\\
&=\frac1n\sup_{y\in G}\sum_{i=1}^n\delta_{x_i}(Ay^{-1})=\sup_{y\in G}\mu(Ay^{-1})<I(\{xA\}_{x\in G})-\e
\end{aligned}
$$is a desired contradiction proving the equality $\sigma_R(A)=I(\{xA\}_{x\in G})$.
\smallskip

The equality $I(\{xA\}_{x\in G})=\sup_{\mu\in P(G)}\inf_{x\in G}\mu(Ay)$ follows from Proposition 1 and Theorem 2 of \cite{Kelly}. So, $$\inf_{\mu\in P_\w(G)}\sup_{y\in G}\mu(Ay)=\sigma_R(A)=I(\{xA\}_{x\in G})=\sup_{\mu\in P(G)}\inf_{x\in G}\mu(xA).$$
\end{proof}

For a group $G$ by $P_l(G)$ we denote the subset of $P(G)$ consisting of all left-invariant  probability measures on $G$. Observe that a group $G$ is amenable if and only if $P_l(G)\ne\emptyset$.

\begin{theorem}\label{t3.4nn} If a group $G$ is amenable, then $$\sigma_R(A)=\sigma^R(A)=\sup_{\mu\in P_l(G)}\mu(A)$$for every subset $A\subset G$.
\end{theorem}

\begin{proof} By Theorem~\ref{t-FC}, $\sigma^R(A)=\sigma_R(A)$.  Theorem~\ref{t3.3nn} implies that
$$\sup_{\mu\in P_l(G)}\mu(A)=\sup_{\mu\in P_l(G)}\inf_{x\in G}\mu(xA)\le\sup_{\mu\in P(G)}\inf_{x\in G}\mu(xA)\le \sigma_R(A).$$
To show that $\sigma_R(A)\le\sup_{\mu\in P_l(G)}\mu(A)$, take any $\e>0$ and using Theorem~\ref{t3.3nn}, find a measure $\nu\in P(G)$ such that $\sigma_R(A)-\e<\inf_{x\in G}\nu(xA)$. Now we shall modify the measure $\nu$ to a right-invariant measure $\tilde\nu$.

Let $l_\infty(G)$ be the Banach lattice of all bounded real-valued functions on the group $G$. Each real number $c\in\mathbb R$ will be identified with the constant function $G\to\{c\}\subset \mathbb R$. The set $l_\infty(G)$ is endowed with the right action $l_\infty\times G\to l_\infty$ of the group $G$. This action assigns to each pair $(f,z)\in l_\infty\times G$ the function $fz:G\to\mathbb R$ defined by $fz(x)=f(xz)$ for $x\in G$. By \cite{Pat}, the amenability of the group $G$ implies the existence of a $G$-invariant linear functional $a^*:l_\infty(G)\to\mathbb R$ with $\|a^*\|=1=a^*(1)$. This functional is monotone in the sense that $a^*(f)\le a^*(g)$ for any bounded functions $f\le g$ on $G$.

For each subset $B\subset G$ consider the function $\nu_B\in l_\infty$ defined by $\nu_B(x)=\nu(xB)$ for $x\in G$ and put $\tilde \nu(B)=a^*(\nu_B)$. It is standard to check that $\tilde\nu:\mathcal P(G)\to[0,1]$, $\tilde\nu:B\mapsto \tilde \nu(B)$, is a well-defined measure on $G$. To see that the measure $\tilde \nu$ is left-invariant, observe that for every $B\subset G$ and $y,x\in G$ we get $$\nu_{xB}(y)=\nu(yxB)=\nu_Bx(y),$$ which means that $\nu_{xB}=\nu_Bx$. The $G$-invariance of the functional $a^*$ guarantees that $a^*(\nu_Bx)=a^*(\nu_B)$ and hence $\tilde \nu(xB)=a^*(\nu_{xB})=a^*(\nu_Bx)=a^*(\nu_B)=\tilde\nu(B)$, which means that the measure $\tilde\nu$ is left-invariant. It follows from $\inf_{x\in G}\nu(xA)>\sigma_R(A)-\e$ that $\nu_A\ge \sigma_R(A)-\e$ and $\tilde \nu(A)=a^*(\nu_A)\ge \sigma_R(A)-\e$ by the monotonicity of the functional $a^*$. So, $\sigma_R(A)-\e\le \tilde\nu(A)\le\sup_{\mu\in P_l(G)}\mu(A)$. Since $\e>0$ was arbitrary, this implies $\sigma_R(A)\le\sup_{\mu\in P_l(G)}\mu(A)$. So, $\sigma^R(A)=\sigma_R(A)=\sup_{\mu\in P_l(G)}\mu(A)$.
\end{proof}

Theorems~\ref{t3.3nn} and \ref{t3.4nn} imply the following result due to Solecki \cite[\S7]{Sol}.

\begin{corollary}[Solecki]\label{solamenable} If $G$ is an amenable group, then the function is $\sigma^R=\sigma_R$ is subadditive.
\end{corollary}

\begin{proof} The equality $\sigma^R=\sigma_R$ follows from Theorem~\ref{t-FC}(1). To see that $\sigma_R$ is subadditive, take any subsets $A,B\subset G$ and apply Theorem~\ref{t3.4nn} to get:
$$\sigma_R(A\cup B)=\sup_{\mu\in P_l(G)}\mu(A\cup B)\le \sup_{\mu\in P_l(G)}(\mu(A)+\mu(B))\le\sup_{\mu\in P_l(G)}\mu(A)+\sup_{\mu\in P_l(G)}\mu(B)=\sigma_R(A)+\sigma_R(B).$$
\end{proof}

We define a group $G$ to be {\em Solecki amenable} if the function $\sigma_R$ is subadditive. By Corollary~\ref{solamenable}, each amenable group is Solecki amenable. It is not known if each Solecki amenable group is amenable (see \cite[\S7]{Sol}). Nonetheless the following characterization of amenability holds.

\begin{theorem} For a group $G$ the following conditions are equivalent:
\begin{enumerate}
\item $G$ is amenable;
\item the group $\IZ\times G$ is Solecki amenable;
\item for each $n\in\IN$ there is a finite group $F$ of cardinality $|F|\ge n$ such that the group $F\times G$ is Solecki amenable;
\item for each $n\in\IN$ there is a finite group $F$ of cardinality $|F|\ge n$ such that for any partition $F\times G=A\cup B$ of the group $F\times G$ we get $\sigma_R(A)+\sigma_R(B)\ge 1$.
\end{enumerate}
\end{theorem}

\begin{proof} The implication $(1)\Ra(2)$ follows from Corollary~\ref{solamenable} and the well-known fact that the product of two amenable groups is amenable. To see that $(2)\Ra(3)$ it suffices to observe that a quotient group of a Solecki amenable group is Solecki amenable. The implication $(3)\Ra(4)$ is trivial. So, it remains to prove that $(4)\Ra(1)$.

Assume that the group $G$ is not amenable. Consider the Banach space $l_1(G)$ of all real-valued functions $f$ on $G$ with $\sum_{x\in G}|f(x)|<\infty$. The Banach space $l_1(G)$ is endowed with the norm $\|f\|_1=\sum_{x\in G}|f(x)|$.
The dual Banach space $l_1(G)^*$ to $l_1(G)$ can be identified with the Banach space $l_\infty(G)$ of all bounded functions on $G$ endowed with the norm $\|f\|_\infty=\sup_{x\in G}|f(x)|$.

Consider the closed convex set $P=\{f\in l_1(G):f\ge 0,\;\|f\|_1=1\}$ in $l_1(G)$. Each function $f\in P$ can be identified with the probability measure $\sum_{x\in G}f(x)\delta_x$. Since $G$ is not amenable, Emerson's characterization of amenability \cite[1.7]{Emerson} yields two  measures $\mu,\eta\in P$ such that the convex sets $\mu*P=\{\mu*\nu:\nu\in P\}$ and $\eta*P=\{\eta*\nu:\nu\in P\}$ have disjoint closures in the Banach space $l_1(G)$. By the Hahn-Banach Theorem, the convex sets $\mu*P$ and $\eta*P$ can be separated by a linear functional $f\in l_1(G)^*=l_\infty(G)$ in the sense that
$$\sup_{\nu\in P}\mu*\nu(f)=c<C=\inf_{\nu\in P}\eta*\nu(f)$$for some real numbers $c<C$.
Multiplying $f$ by a suitable positive constant, we can assume that $\|f\|_\infty\le \frac12$. Let $n\in\IN$ be any number such that $n\ge \frac5{C-c}$ and let $F$ be a finite group of cardinality $m=|F|\ge n$. Choose two finitely supported measures $\tilde \mu,\tilde\eta\in P_\w(G)$ such that $\|\mu-\tilde\mu\|_1<\frac1{m}$ and $\|\eta-\tilde\eta\|<\frac1{m}$. Also choose a function $g:G\to[0,1]\cap\frac1{m}\IZ$ such that $\|g-(\frac12+f)\|<\frac1{m}$. Observe that
$$\sup_{\nu\in P}\tilde \mu*\nu(g)\le c+\frac2{m}<C-\frac2{m}\le\inf_{\nu\in P}\tilde\eta*\nu(g).$$
Take any subset $A\subset F\times G$ such that for each $y\in G$ the set $\{x\in F:(x,y)\in A\}$ has cardinality $m\cdot g(y)$. Put $B=(F\times G)\setminus A$. We claim that $\sigma_R(A)+\sigma_R(B)<1$.
Let $\lambda=\frac1m\sum_{x\in F}\delta_x$ be the Haar measure on the finite group $F$. Identifying $F$ and $G$ with the subgroups $F\times\{1_G\}$ and $\{1_F\}\times G$ of $F\times G$, we can consider the finitely supported measures $\lambda*\tilde\mu$ and $\lambda*\tilde\eta$ on the group $F\times G$. Write $\tilde\mu=\sum_{i}\alpha_i\delta_{y_i}$ and observe that
$$
\begin{aligned}
\sigma_R(A)&\le\sup_{(x,y)\in F\times G}\lambda*\tilde\mu(Axy)=\sup_{(x,y)\in F\times G}\sum_{i}\alpha_i\sum_{z\in F}\frac{1}m\delta_{zy_i}(Axy)=\\
&=\sup_{y\in G}\sup_{z\in F}\sum_i\alpha_i\frac{|\{z\in F:zy_i\in Axy\}|}m=\sup_{y\in G}\sup_{x\in F}\sum_i\alpha_i\frac{|\{z\in F:zx^{-1}y_iy^{-1}\in A\}|}m=\\
&=\sup_{y\in G}\sup_{z\in F}\sum_i\alpha_i g(y_iy^{-1})=\sup_{y\in G}\sum_i\alpha_i\delta_{y_i}*\delta_{y^{-1}}(g)=\sup_{y\in G}\tilde\mu*\delta_{y^{-1}}(g)\le \sup_{\nu\in P}\tilde\mu*\nu(g)\le c+\frac2m\, .
\end{aligned}
$$
By analogy we can prove that for the set $B=(F\times G)\setminus A$ we get
$$\sigma_R(B)\le \sup_{(x,y)\in F\times G}\lambda*\tilde\eta(Bxy)=\sup_{(x,y)\in F\times G}(1-\lambda*\tilde\eta(Axy))=1-\inf_{(x,y)\in F\times G}\lambda*\tilde\eta\le 1-(C-\frac2m).
$$
Then $$\sigma_R(A)+\sigma_R(B)\le c+\frac2m+1-C+\frac2m<1-(C-c)+\frac4m<1-\frac5m+\frac4m<1=\sigma_R(F\times G).$$witnessing that the condition $(4)$ does not hold.
\end{proof}

\section{The subadditivization of the Solecki densities}\label{s4}

Since the Solecki densities on a group $G$ are not subadditive in general, it is reasonable to consider their subadditivizations. By the {\em subadditivization} of a density $\mu:\mathcal P(X)\to[0,1]$ we understand the submeasure $\hat\mu:\mathcal P(X)\to[0,1]$ defined by the formula
$$\hat\mu(A)=\sup_{B\in\mathcal P(X)}\mu(A\cup B)-\mu(B).$$

The following proposition can be easily derived from the defintion of $\hat\mu$.

\begin{proposition}\label{p4.1n} If $\mu:\mathcal P(G)\to[0,1]$ is a density on a group $G$, then
\begin{enumerate}
\item its subadditivisation $\hat\mu$ is a submeasure on $G$;
\item $\mu\le\hat\mu$;
\item $\mu=\hat\mu$ if and only if $\mu$ is subadditive;
\item the submeasure $\hat\mu$ is (right-, left-) invariant provided so is the density $\mu$.
\end{enumerate}
\end{proposition}

Given a group $G$ by $\hat\sigma^R$ and $\hat\sigma_R$ we shall denote the subadditivizations of the right-Solecki densities  $\sigma^R$ and $\sigma_R$ on $G$, respectively. It follows that  $\hat\sigma^R$, $\hat\sigma_R$ are invariant submeasures on $G$ and $\hat\sigma_R$ is bounded from above by the Solecki submeasure $\sigma$ according to
Proposition~\ref{semiadditive}(1). In fact the upper bound $\hat\sigma_R\le\sigma$ can be improved to $\hat\sigma\le\varsigma_R$, where $\varsigma_R:\mathcal P(G)\to[0,1]$ is a density defined on each group $G$ by the formula:
$$
\varsigma_R(A)=
\sup_{\mu_1\in P(G)}\sup_{\mu_2\in P_\w(G)}\inf_{\mu_3\in P_\w(G)}\mu_2*\mu_3*\mu_1(A).
$$
Using the fact that finitely supported measures on $G$ can be approximated by measures of the form $\frac1n \sum_{i=1}^n\delta_{x_i}$ for some points $x_1,\dots,x_n\in G$, we can show that the density $\varsigma_R$ can be equivalently defined as
$$\varsigma_R(A)=\sup_{\mu\in P(G)}\sup_{y_1,\dots,y_n\in G}\inf_{x\in G}\frac1n\sum_{i=1}^n\mu(xy_iA).$$

\begin{theorem}\label{subadit-sol} For each group $G$ we get
$$\sigma_R\le\hat\sigma_R\le\varsigma_R\le\hat\varsigma_R\le\sigma.$$
\end{theorem}

\begin{proof} The inequalities $\sigma_R\le\hat\sigma_R$ and $\varsigma_R\le\hat\varsigma_R$ are trivial.

The inequality $\hat\sigma_R\le\varsigma_R$ will be proved as soon as for every  $\e>0$ we find a measure $\mu_1\in P(G)$ and a finitely supported measure $\mu_2\in P_\w(G)$ such that  $\inf_{\mu_3\in P_\w(G)}\mu_2*\mu_3*\mu_1(A)\ge \hat\sigma_R(A)-3\e$. By the definition of the submeasure $\hat\sigma_R$, there is a subset $B\subset G$ such that $\sigma_R(A\cup B)-\sigma_R(B)>\hat\sigma_R(A)-\e$. Replacing the set $B$ by $B\setminus A$ we can additionally assume that $A\cap B=\emptyset$.

Theorem~\ref{t3.3nn}  implies that $\sigma_R(A\cup B)=\sup_{\mu\in P(G)}\inf_{x\in G}\mu(x(A\cup B))$ and hence there is a measure $\mu_1\in P(G)$ such that
$\inf_{x\in G}\mu_1(x(A\cup B))>\sigma_R(A\cup B)-\e$.

By Theorem~\ref{t3.3nn}, $\sigma_R(B)=I(\{xB\}_{x\in G})$. Consequently, we can find $n\in\IN$ and a sequence $(x_i)_{i\in n}\in G^n$ such that $\big\|\frac1n\sum_{i\in n}\chi_{x_iB}\big\|<\sigma_R(B)+\e$. Consider the measure $\mu_2=\sum_{i\in n}\alpha_i\delta_{x_i^{-1}}$ where $\alpha_i=\frac1n$ for all $i\in n$ and observe that for every $y\in G$ we get
$$
\Big\|\sum_{i\in n}\alpha_i\chi_{yx_iB}\Big\|=\sup_{z\in G}\sum_{i\in n}\alpha_i\chi_{yx_iB}(z)=
\sup_{z\in G}\sum_{i\in n}\alpha_i\chi_{x_iB}(y^{-1}z)\le \Big\|\sum_{i\in n}\alpha_i\chi_{x_iB}\Big\|<\sigma_R(B)+\e.
$$Integrating the function $\sum_{i\in n}\alpha_i\chi_{yx_i B}$ by the measure $\mu_1$, we obtain the inequality
$$
\sum_{i\in n}\alpha_i\mu_1(yx_iB)\le\Big\|\sum_{i\in n}\alpha_i\chi_{yx_iB}\Big\|<\sigma_R(B)+\e
$$
holding for every $y\in G$.

Now observe that for every finitely supported measure $\mu_2=\sum_j\beta_j\delta_{y_j}\in P_\w(G)$ we get
$$
\begin{aligned}
\mu_2*\mu_3*\mu_1(A)&=\sum_{i,j}\alpha_i\beta_j\cdot\delta_{x_i^{-1}}*\delta_{y_j}*\mu_1(A)=
\sum_{i,j}\alpha_i\beta_j\mu_1(y_j^{-1}x_iA)=\\
&=\sum_{i,j}\alpha_i\beta_j\mu_1(y_j^{-1}x_i(A\cup B))-\sum_{i,j}\alpha_i\beta_j\mu_1(y_j^{-1}x_iB)>\\
&>\sum_{i,j}\alpha_i\beta_j(\sigma_R(A\cup B)-\e)-\sum_j\beta_j\sum_i\alpha_i\mu_1(y_j^{-1}x_iB)>\\
&>\sigma_R(A\cup B)-\e-\sum_j\beta_j(\sigma_R(B)+\e)=\sigma_R(A\cup B)-\sigma_R(B)-2\e>\hat\sigma_R(A)-3\e,
\end{aligned}
$$
which implies the desired inequality
$$\hat\sigma_R(A)-3\e\le\inf_{\mu_3\in P_\w(G)}\mu_2*\mu_3*\mu_1(A)\le\varsigma_R(A).$$

The final inequality $\hat\varsigma_R\le\sigma$ will follow as soon as we prove that $\varsigma_R(A\cup B)\le\varsigma_R(A)+\sigma(B)+3\e$ for every sets $A,B\subset G$ and $\e>0$. By the definition of $\varsigma_R(A\cup B)$, there are measures $\mu_1\in P(G)$ and $\mu_2\in P_\w(G)$ such that
$\inf_{\mu_3\in P_\w(G)}\mu_2*\mu_3*\mu_1(A\cup B)>\varsigma_R(A\cup B)-\e$.
By the definition of the Solecki submeasure $\sigma(B)$ there is a finitely supported measure $\mu\in P_\w(G)$ such that $\sup_{x,y\in G}\mu(xBy)<\sigma(B)+\e$.

Next, choose a measure $\nu\in P_\w$ such that $(\mu_2*\mu)*\nu*\mu_1(A)<\inf_{\eta\in P_\w(G)}(\mu_2*\mu)*\eta*\mu_1(A)+\e$. Put $\mu_3=\mu*\nu$ and observe that $\sup_{x,y\in G}\mu(xBy)<\sigma(B)+\e$ implies $\sup_{x,y\in B}\mu*\nu(xBy)\le \sup_{x,y\in G}\mu(xBy)<\sigma(B)+\e$.

Write the finitely supported measures $\mu_2$ and $\mu_3$ as convex combinations $\mu_2=\sum_{i\in n}\alpha_i\delta_{a_i}$ and $\mu_3=\sum_{j\in m}\beta_j\delta_{b_j}$ of Dirac measures.
For every $i\le n$, consider the function $f_i=\sum_{j\in m}\beta_j\chi_{b_j^{-1}a_i^{-1}B}$ and observe that it has norm
$$
\begin{aligned}
\|f_i\|&=\sup_{x\in G}\sum_{j\in m}\beta_j\chi_{b_j^{-1}a_i^{-1}B}(x)=
\sup_{x\in G}\sum_{j\in m}\beta_j\delta_x(b_j^{-1}a_i^{-1}B)=\sup_{x\in G}\sum_{j\in m}\beta_j\delta_{b_jx}(a_i^{-1}B)=\\
&=\sup_{x\in G}\sum_{j\in m}\beta_j\delta_{b_j}(a_i^{-1}Bx^{-1})=
\sup_{x\in G}\mu_3(a_i^{-1}Bx^{-1})<\sigma(B)+\e.
\end{aligned}
$$
Integrating the function $f_i$ by the measure $\mu_1$, we get the inequality
$$\sum_{j\in m}\beta_j\mu_1(b_i^{-1}a_i^{-1}B)=\int_G f_i d\mu_1\le\|f_i\|<\sigma(B)+\e,$$
which implies
$$
\begin{aligned}
\varsigma_R(A\cup B)&<\e+\inf_{\eta\in P_\w(G)}\mu_2*\eta*\mu_1(A\cup B)\le
\e+ \mu_2*\mu_3*\mu_1(A\cup B)\le\\
&\le \e+\mu_2*\mu_3*\mu_1(A)+\mu_2*\mu_3*\mu_1(B)=\e+\mu_2*\mu*\nu*\mu_1(A)+
\sum_{i,j}\alpha_i\beta_j\cdot\delta_{a_i}*\delta_{b_j}*\mu_1(B)<\\
&<\e+\e+\inf_{\eta\in P_\w(G)}(\mu_2*\mu)*\eta*\mu_1(A)+\sum_{i}\alpha_i\sum_{j}\beta_j\mu_1(b_j^{-1}a_i^{-1}B)\le \\
&\le2\e+\varsigma_R(A)+\sum_i\alpha_i(\sigma(B)+\e)=\varsigma_R(A)+\sigma(B)+3\e.
\end{aligned}
$$
\end{proof}

Theorem~\ref{subadit-sol} has the following combinatorial corollary.

\begin{theorem}\label{t4.2} For any subset $A\subset G$ with $\varsigma_R(A)>0$ there is a finite subset $F\subset G$ such that $G=FAA^{-1}F$.
\end{theorem}

\begin{proof} By the definition of the density $\varsigma_R$, there are measures $\mu_1\in P(G)$ and $\mu_2\in P_\w(G)$ such that $\inf_{\mu_3\in P_\w(G)}\mu_2*\mu_3*\mu_1(A)>\frac12\varsigma_R(A)>0$.
Write the finitely supported measure $\mu_2$ as a convex combination $\mu_2=\sum_{i\in n}\alpha_i\delta_{a_i}$ of Dirac measures and put $S=\{a_i\}_{i\in n}$.

By the Zorn's Lemma there is a maximal subset $M\subset G$ such that for every $a\in S$ the indexed family $(xa^{-1}A)_{x\in M}$ is disjoint. By the maximality of $M$, for every $x\in G$ there is $a^{-1}\in S$ such that $xa^{-1}A\cap Ma^{-1}A\ne\emptyset$, which implies that $x\in Ma^{-1}AA^{-1}a\subset MS^{-1}AA^{-1}S$ and hence $G=MS^{-1}AA^{-1}S$. It remains to prove that the set $M$ is finite. For this observe that for every finite subset $F\subset M$ and every $a\in S$ the indexed family $(xa^{-1}A)_{x\in F}$ is disjoint and hence $\sum_{x\in F}\mu_1(xa^{-1}A)\le 1$.  Then $$\sum_{x\in F}\mu_2*\delta_{x^{-1}}*\mu_1(A)=\sum_{x\in F}\sum_{i\in n}\alpha_i\delta_{a_i}*\delta_{x^{-1}}*\mu_1(A)=
\sum_{i\in n}\alpha_i\sum_{x\in F}\mu_1(xa_i^{-1}A)\le \sum_{i\in n}\alpha_i\cdot 1=1$$ and hence $$|F|\cdot\inf_{x\in F}\mu_2*\delta_{x^{-1}}*\mu_1(A)\le\sum_{x\in F}\mu_2*\delta_{x^{-1}}*\mu_1(A)\le 1.$$
This implies the inequalities $$\frac1{|F|}\ge \inf_{x\in F}\mu_2*\delta_{x^{-1}}*\mu_1(A)\ge \inf_{\mu_3\in P_\w(G)}\mu_2*\mu_3*\mu_1(A)>\frac12\varsigma_R(A)$$ and $|F|\le 2/{\varsigma_R(A)}$. So, $M$ is a finite set of cardinality $|M|\le 2/\varsigma_R(A)$ and $G=MS^{-1}AA^{-1}S$.
\end{proof}

\begin{problem} Is $\varsigma_R=\sigma_R$ for any amenable group $G$?
\end{problem}

\section{The right Solecki density versus the upper Banach density on amenable groups}

In this section we shall prove that for an amenable group $G$ the right Solecki density $\sigma_R=\sigma^R=\hat\sigma_R=\hat\sigma^R$ coincides with the upper Banach density $d^*$, widely exploited in Ramsey Theory of groups and semigroups, see \cite{HS} and references therein. For the group $\IZ$ of integers the upper Banach density was introduced by Polya \cite{Polya} in 1929. Later, with help of F\o lner sequences this notion was generalized to countable amenable groups; see \cite{BBF} and \cite{HS}.

A sequence $(F_n)_{n\in\w}$ of finite subsets of a  group $G$ is called a {\em F\o lner sequence} if for every $g\in G$ the sequence $(|F_n\triangle gF_n|/|F_n|)_{n\in\w}$ tends to zero. Here by $A\triangle B$ we denote the symmetric difference $(A\setminus B)\cup(B\setminus A)$ of two sets $A,B\subset G$. By the F\o lner criterion \cite[4.10]{Pat}, a group $G$ admits a F\o lner sequence $(F_n)_{n\in\w}$ if and only if $G$ is countable and amenable.

Let $G$ be a countable amenable group. The {\em upper density of a subset $A\subset G$ with respect to a F\o lner sequence $(F_n)_{n\in\w}$} is defined as $$\bar d_{(F_n)}(A)=\limsup_{n\to\infty}\frac{|A\cap F_n|}{|F_n|}$$and the number
$$d^*(A)=\sup\{\bar d_{(F_n)}(A):(F_n)_{n\in\w}\mbox{ is a F\o lner sequence}\}$$is called the {\em upper Banach density} of $A$.

In \cite{HS} and \cite{NL}  the upper Banach density was defined for subsets of any amenable group. According to \cite{NL}, the {\em upper Banach density} $d^*(A)$ of a subset $A$ of an amenable group $G$ is defined as
$$d^*(A)=\sup\Big\{\alpha\in[0,1]:\forall F\in[G]^{<\w}\;\forall\e>0\;\exists K\in[G]^{<\w}\mbox{ such that }\max_{x\in F}\frac{|xK\triangle K|}{|K|}<\e\mbox{ and }\frac{|K\cap A|}{|K|}\ge\alpha\Big\}.$$

It turns out that the right Solecki density $\sigma^R$ on an amenable group $G$ coincides with the upper Banach density $d^*$.

\begin{theorem}\label{Bandens} For any amenable group $G$ we get $d^*=\sigma_R=\sigma^R=\hat\sigma_R=\hat\sigma^R$.
\end{theorem}

\begin{proof} By Theorem~\ref{t-FC}(1), $\sigma_R=\sigma^R$. By Corollary~\ref{solamenable}, the right Solecki density $\sigma_R=\sigma^R$ is subadditive and hence coincides with its subadditivization. So, $\sigma_R=\sigma^R=\hat\sigma^R=\hat\sigma_R$.

To see that $d^*(A)\le\sigma^R(A)$, assume conversely that $\sigma_R(A)<d^*(A)$ and find a finite
subset $F\subset G$ such that $$\sigma^R(A)\le\sup_{y\in G}\frac{|Fy\cap A|}{|F|}<d^*(A)-\e$$ for some $\e>0$. Replacing $F$ by $Fz^{-1}$ for some $z\in F$ we can additionally assume that $F$ contains the unit $1_G$ of the group $G$. Choose a positive $\delta$ so small that
$$\frac{d^*(A)-\delta(|F|+1)}{1+\delta}>d^*(A)-\e.$$
By the definition of $d^*(A)$, for the finite set $F$ and the positive number $\delta$ there is a finite subset $K\subset G$ such that $\max_{x\in F}\frac{|xK\triangle K|}{|K|}<\frac{\delta}{|F|}$ and $|K\cap A|/|K|\ge d^*(A)-\delta$.
Then $|FK\setminus K|\le\sum_{x\in F}|xK\setminus K|<\delta$, $|FK|\le |FK\setminus K|+|K|<|K|(1+\delta)$ and hence $$|FK\cap A|\ge|K\cap A|\ge|K|(d^*(A)-\delta)\ge |FK|\frac{d^*(A)-\delta}{1+\delta}.$$

Consider the map $\pi:F\times K\to FK$, $\pi:(x,y)\mapsto xy$, and observe that $|\pi^{-1}(z)|\le |F|$ for all $z\in FK$. Let $S=\{z\in FK:|\pi^{-1}(z)|<|F|\}$ and $E=\{z\in FK:|\pi^{-1}(z)|=|F|\}$.
It follows that
$$|F|\cdot\frac{|FK|}{1+\delta}\le |F|\cdot|K|=|F\times K|=|\pi^{-1}(S)\cup \pi^{-1}(FK\setminus S)|\le (|F|-1)\cdot|S|+|F|\cdot(|FK|-|S|)=|F|\cdot|FK|-|S|$$which implies $|S|\le |F|\cdot |FK|\cdot\big(1-\frac1{1+\delta}\big)=|F|\cdot |FK|\frac\delta{1+\delta}$ and $$|E|=|FK|-|S|\ge |FK|\Big(1-\frac{\delta|F|}{1+\delta}\Big).$$
Observe that
$$|E\cap A|=|FK\cap A|-|S\cap A|\ge |FK|\frac{d^*(A)-\delta}{1+\delta}-|FK|\frac{\delta|F|}{1+\delta}\ge|K|\frac{d^*(A)-\delta(|F|+1)}{1+\delta}.
$$

The choice of $\delta$ guarantees that $$|\pi^{-1}(E\cap A)|=|E\cap A|\cdot|F|\ge |F|\cdot |K|\frac{d^*(A)-\delta(|F|+1)}{1+\delta}>|F|\cdot|K|(d^*(A)-\e).$$
On the other hand,
$$\pi^{-1}(E\cap A)\subset\{(x,y)\in F\times K:xy\in E\cap A\}\subset \{(x,y)\in F\times K:xy\in A\}$$ and hence
$$|\pi^{-1}(E\cap A)|\le|\{(x,y)\in F\times K:xy\in A\}|=\sum_{y\in K}|\{x\in F:xy\in A\}|=\sum_{y\in K}|Fy\cap A|<|K|\cdot |F| (d^*(A)-\e),$$which is a desired contradiction proving that $d^*(A)\le\sigma^R(A)$.
\smallskip

We claim that $d^*(A)=\sigma^R(A)$. In the opposite case $d^*(A)<\sigma^R(A)$ and by the definition of $d^*(A)$, there is a finite set $F\subset G$ and a positive number $\e$ such that for any finite set $K\subset G$ with $\max_{x\in F}\frac{|xK\triangle K|}{|K|}<\e$ we get $\frac{|K\cap A|}{|K|}<\sigma^R(A)$.
By the F\o lner criterion of the amenability, there is a finite set $E\subset G$ such that
$\max_{x\in F}\frac{|xE\triangle E|}{|E|}<\e$. By the definition of the right Solecki density $\sigma^R(A)$, there is a point $y\in G$ such that $\frac{|Ey\cap A|}{|E|}\ge\sigma^R(A)$. Then we get a contradiction letting $K=Ey$.
\end{proof}

\section{Solecki densities and combinatorial sizes of subsets in groups}

In this section we shall evaluate the Solecki densities or submeasures of subsets which are small or large in a suitable combinatorial sense. Combinatorial sizes of subsets in groups were studied in many papers (see, the survey \cite{Prot} and references therein).

Following \cite{Prot} we define a subset $A$ of a group $G$ to be
\begin{itemize}
\item {\em thick} if for every finite subset $F\subset G$ there is a point $y\in G$ such that $Fy\subset A$;
\item {\em large} if $FA=G$ for some finite subset $F\subset G$;
\item {\em small} if for any large set $B\subset G$ the complement $B\setminus A$ is large.
\end{itemize}
It follows that small sets form an invariant ideal on each group. By Theorem~12.4 \cite{PB}, a subset $A$ of a group $G$ is small if and only if for any finite subset $F\subset G$ the set $FA$ is not thick. This characterization of small sets and Proposition~\ref{semiadditive}(2) imply:

\begin{proposition}\label{p:rizne} Let $G$ be a group.
\begin{enumerate}
\item Each  large set $A\subset G$ has $\sigma_L(A)=\sigma_R(A^{-1})>0$;
\item A subset $A\subset G$ is thick iff $\sigma_R(A)=1$ iff $\sigma^R(A)=1$;
\item A subset $A\subset G$ is small iff $\sigma_R(FA)<1$ for each finite subset $F\subset G$.
\end{enumerate}
\end{proposition}

The following proposition combined with Proposition~\ref{p:rizne}(2) implies that each infinite group $G$ contains $|G|$ many subsets of right Solecki density 1.

\begin{proposition}\label{p2.6} Each infinite group $G$ contains $|G|$ many pairwise disjoint thick sets.
\end{proposition}

\begin{proof} We identify the cardinal $|G|$ with the smallest ordinal of cardinality $|G|$. Let $[G]^{<\w}$ be the family of all finite subsets of $G$. The set $[G]^{<\w}\times G$ has cardinality $|G|$ and hence can be enumerated as $[G]^{<\w}\times G=\{(F_\alpha,y_\alpha):\alpha\in|G|\}$. For each ordinal $\alpha\in|G|$ by transfinite induction choose a point $$x_\alpha\in G\setminus  \bigcup_{\beta<\alpha}F_\alpha^{-1}(x_\beta F_\beta \cup F_\beta x_\beta)F_\alpha^{-1}.$$ Such choice of the points $x_\alpha$ guarantees that the family $\{x_\alpha F_\alpha\cup F_\alpha x_\alpha\}_{\alpha\in|G|}$ is disjoint. Then the indexed family $\{X_y\}_{y\in G}$ consisting of the sets  $X_y=\bigcup\{x_\alpha F_\alpha\cup F_\alpha x_\alpha:y_\alpha=y\}$ is also disjoint. We claim that for each $y\in G$ the set $X_y$ and $X_y^{-1}$ are thick. Given any finite subset $F\subset G$, find an ordinal $\alpha<|G|$ such that $(F_\alpha,y_\alpha)=(F,y)$. Then $x_\alpha F\cup Fx_\alpha=x_\alpha F_\alpha\cup F_\alpha x_\alpha\subset X_y$, which implies that $X_y$ and $X_y^{-1}$ both are thick and hence have right Solecki density 1 according to Proposition~\ref{p:rizne}(2).
\end{proof}

Now we shall calculate the Solecki densities of subgroups of groups.

\begin{proposition}\label{zero-index} For a subgroup $H$ of a group $G$ the following conditions are equivalent:
\begin{enumerate}
\item $H$ has infinite index $G:H$ in $G$;
\item $H$ is not large in $G$;
\item $H$ is small in $G$;
\item $\sigma^R(H)=0$;
\item $\sigma_R(H)=0$;
\item $\hat\sigma_R(H)=0$;
\item $\varsigma_R(H)=0$.
\end{enumerate}
\end{proposition}

\begin{proof} It suffices to check that

$(1)\Leftrightarrow(2)\Leftrightarrow(3)\Ra(7)\Ra(6)\Ra(5)\Ra(2)\Ra(4)\Ra(5)$.
The equivalence $(1)\Leftrightarrow(2)$ follows from the definition of a large set and $(2)\Leftrightarrow(3)$ has been proved in Lemma~4.2 of \cite{LP}.

To prove that $(3)\Ra(7)$, assume that $\varsigma_R(H)>0$ and applying Theorem~\ref{t4.2}, conclude that $G=FHF$ for some finite subset $F\subset G$. Since  small subsets form a non-trivial invariant ideal of subsets of $G$, the equality $G=FHF$ implies that the group $H$ is not  small.

The implications $(7)\Ra(6)\Ra(5)$ follow from the inequalities $\sigma_R\le\hat\sigma_R\le\varsigma_R$ proved in Theorem~\ref{subadit-sol}. The implication $(5)\Ra(2)$ follows from Proposition~\ref{p:rizne}(1).

To prove that $(2)\Ra(4)$, assume that $\sigma^R(H)>0$. Applying Proposition~\ref{p7.1n} proved in Section~\ref{s:difference}, we conclude that the set $H=HH^{-1}$ is large. The final implication $(4)\Ra(5)$ is trivial.
\end{proof}

\begin{proposition}\label{p6.3n} Each subgroup $H$ of a group $G$ has
$\sigma_R(H)=\sigma_R(H)=\hat\sigma_R(H)=\varsigma_R(H)=\frac1{G:H}.$ If the group $H$ has finite index in $G$, then $\sigma(H)=\frac1{G:H}$.
\end{proposition}

\begin{proof} If one of the numbers $\sigma_R(H)$, $\sigma_R(H)$, $\hat \sigma_R(H)$ or $\varsigma_R(H)$  is equal to zero, then the subgroup $H$ has infinite index in $G$ and
$$\sigma_R(H)=\sigma_R(H)=\hat\sigma_R(H)=\varsigma_R(H)=\frac1{G:H}=0$$
according to Proposition~\ref{zero-index}.

So, we assume that the numbers $\sigma_R(H)$, $\sigma_R(H)$, $\hat \sigma_R(H)$, $\varsigma_R(H)$ are positive. In this case, the subgroup $H$ has finite index in $G$ according to Proposition~\ref{zero-index} and then the normal subgroup $N=\bigcap_{x\in G}xHx^{-1}$ has finite index in $G$ too.
Consider the quotient group $G/N$ and the quotient homomorphism $q:G\to G/N$. It follows that $H=q^{-1}(q(H))$. It can be easily deduced from the definitions of the Solecki submeasure $\sigma$ and the Solecki densities $\sigma_R$ and $\sigma^R$  on the finite group $G/N$ that any subset $A\subset G/N$ has submeasure $\sigma(A)=\sigma_R(A)=\sigma^R(A)=\frac{|A|}{|G/N|}$. Now Proposition~\ref{p2.3} and its counterpart for the right Solecki density $\sigma_R$ imply that $$\sigma_R(H)=\sigma_R(q(H))=\frac{|q(H)|}{|G/N|}=\sigma(q(H))=\sigma(H).$$ Taking into account the inequalities $\sigma_R(H)\le\sigma^R(H)\le\sigma(H)$ and $\sigma_R(H)\le \hat\sigma_R(H)\le\varsigma_R(H)\le\sigma(H)$, we get the desired equalities $$\sigma^R(H)=\sigma_R(H)=\hat\sigma_R(H)=\varsigma_R(H)=\sigma(H)=\frac{|q(H)|}{|G/N|}=\frac1{G:H}.$$
\end{proof}

It is possible to generalize Proposition~\ref{p6.3n} from subgroups to subgroup cosets.
A subset $A$ of a group $G$ will be called a {\em subgroup coset} if $A=xHy$ for some subgroup $H\subset G$ and some points $x,y\in G$. In this case $AA^{-1}=xHx^{-1}$ is a subgroup of $G$, conjugated to $H$. By the {\em index} $G:A$ of the subgroup coset $A$ we understand the index $G:AA^{-1}=|G/AA^{-1}|$ of the subgroup $AA^{-1}$.

The invariance of the densities $\sigma_R$, $\sigma^R$, $\hat\sigma_R$ and $\varsigma_R$ and Proposition~\ref{p6.3n} imply the following corollary.

\begin{corollary}\label{c6.5n} Each subgroup coset $A$ in a group $G$ has
$\sigma_R(A)=\sigma_R(A)=\hat\sigma_R(A)=\varsigma_R(A)=\frac1{|G:A|}.$ If the index of $A$ in $G$ is finite, then $\sigma(A)=\frac1{G:A}$.
\end{corollary}

It is interesting that even for a normal subgroup $H$ of an amenable group $G$ the equality $\sigma(A)=\frac1{G:A}$ need not hold. As a counterexample consider the group $S_X$ of all bijective transformations of an infinite set $X$ and the normal subgroup $FS_X$ of $S_X$ consisting of all bijective transformations $f:X\to X$ with finite support $\supp(f)=\{x\in X:f(x)\ne x\}$.

\begin{example}\label{permutation} For any infinite sets $E\subset X$ the subgroup $FS_E=\{f\in FS_X:\supp(f)\subset E\}$ has Solecki submeasure $\sigma(FS_E)=1$ in the group $FS_X$. If the complement $X\setminus E$ is infinite, then the subgroup $FS_E$ has infinite index in $FS_X$ and hence $\sigma(FS_E)=1\ne 0=\frac1{FS_X:FS_E}$.
\end{example}

\begin{proof} Given a finite subset $A\subset FS_X$ consider its (finite) support
$\supp(A)=\bigcup_{a\in A}\supp(a)$ and find a finitely supported permutation $f\in FS_X$ such that $f(\supp(A))\subset E$. It follows that $\supp(fAf^{-1})\subset E$ and hence $fAf^{-1}\subset FS_E$, witnessing that the set $FS_E$ has Solecki submeasure $\sigma(FS_E)=1$ (according to Proposition~\ref{thick}).

If the complement $X\setminus E$ is infinite, then the subgroup $FS_E$ has infinite index in the group $FS_X$.
\end{proof}

\section{Solecki null, Solecki positive and Solecki one sets in groups}

A subset $A$ of a group $G$ is called
\begin{itemize}
\item {\em Solecki null} if $\sigma(A)=0$;
\item {\em Solecki positive} if $\sigma(A)>0$;
\item {\em Solecki one} if $\sigma(A)=1$.
\end{itemize}

The subadditivity of the Solecki submeasure $\sigma$ implies that Solecki null sets form an invariant ideal of subsets of a group $G$.

Solecki one sets admit a simple combinatorial characterization, which follows immediately from the definition of the Solecki submeasure.

\begin{proposition}\label{thick} A subset $A$ of a group $G$ is Solecki one if and only if for each finite subset $F\subset G$ there are points $x,y\in G$ such that $xFy\subset A$.
\end{proposition}

The notions of Solecki null, one, and positive sets have right modifications.

A subset $A$ of a group $G$ is called
\begin{itemize}
\item {\em right-Solecki null} if $\sigma^R(A)=0$;
\item {\em right-Solecki positive} if $\sigma^R(A)>0$;
\item {\em right-Solecki one} if $\sigma^R(A)=1$.
\end{itemize}

Since $\sigma^R\le\sigma$, each right-Solecki one set in Solecki one and each Solecki null set is right-Solecki null. However the converse implications are not true as Example~\ref{permutation} shows.

Theorems~\ref{t3.3nn} and \ref{t3.4nn} imply the following Zakrzewski's  characterization \cite{Zak} of right-Solecki null sets in amenable groups.

\begin{theorem}[Zakrzewski]\label{t-zak} A subset $A$ of an amenable group $G$ is right-Solecki null if and only if $\mu(A)=0$ for each left-invariant measure on $G$.
\end{theorem}

Now we detect groups in which the classes of Solecki null and right Solecki null sets coincide.
For a group $G$ denote by $G_{FC}=\{x\in G:|x^G|<\infty\}$ the normal subgroup of $G$ consisting of elements $x\in G$ with finite conjugacy class $x^G=\{gxg^{-1}:g\in G\}$. Observe that a group $G$ is an FC-group if and only if $G=G_{FC}$. The following characterization was proved by Solecki in \cite[Theorem 1.3]{Sol}.

\begin{theorem}[Solecki]\label{t6.3n} For a group $G$ the following statements are equivalent:
\begin{enumerate}
\item The subgroup $G_{FC}$ has finite index in $G$;
\item A subset $A\subset G$ is Solecki null if and only if $A$ is right-Solecki null;
\item no Solecki one set $A\subset G$ is right-Solecki null.
\end{enumerate}
\end{theorem}

The Solecki submeasure can be helpful in generalizing some results of Ramsey Theory like the Gallai's  Theorem \cite[p.40]{GRS}. This theorem says that for any finite coloring of the group $G=\IZ^n$ and any finite set $F\subset G$ there are $b\in G$ and $n\in\IN$ such that the homothetic copy $b+nF$ of $F$ is monochrome.

The notion of a homothetic copy can be defined in each semigroup as follows. We say that a subset $B$ of a semigroup $S$ is a {\em homothetic image} of a set $A\subset S$ if $B=f(A)$ for some function $f:S\to S$ of the form $f(x)=a_0xa_1x\cdots xa_n$ for some $n\in\IN$ and some elements $a_0,\dots,a_n\in G$. If $n=1$, then $f(x)=a_0xa_1$ and we shall say that $B=a_0Aa_1$ is a {\em translation image} of $A$.

\begin{theorem}\label{t4.4n} If a subset $A$ of a group $G$ is:
\begin{enumerate}
\item Solecki one, then $A$ contains a translation image of each finite subset $F\subset G$.
\item Solecki positive, then $A$ contains a homothetic image of each finite subset $F\subset G$.
\end{enumerate}
\end{theorem}

\begin{proof} 1. The first statement is a trivial corollary of Proposition~\ref{thick}.
\smallskip

2. Assume that $\e=\sigma(A)>0$ and let $F$ be any finite subset of the group $G$. By the Density Version of the Hales-Jewett Theorem due to Furstenberg and Katznelson \cite{FK}, for the numbers $\e$ and $k=|F|$ there is a number $N$ such that every subset $S\subset F^N$ of cardinality $|S|\ge\e|F^N|$ contains the image $\xi(F)$ of $F$ under an injective function $\xi=(\xi_i)_{i=1}^N:F\to F^N$ whose components $\xi_i:F\to F$ are identity functions or constants.

On the ``cube'' $F^N$ consider the uniformly distributed measure $\mu=\frac1{|F^N|}\sum_{x\in F^N}\delta_x$. The multiplication function $\pi:F^N\to G$, $\pi:(x_1,\dots,x_N)\mapsto x_1\cdots x_N$, maps the measure $\mu$ to a finitely supported probability measure $\nu=\pi(\mu)$ on the group $G$.
By Theorem~\ref{sol}, $\e=\sigma(A)\le\sup_{u,v\in G}\nu(uAv)=\max_{u,v\in G}\nu(uAv)$. So, there are points $u,v\in G$ such that $\nu(uAv)\ge \e$. Then for the map $\pi_{u,v}:F^N\to G$, $\pi_{u,v}(\vec x)=u^{-1}\cdot\pi(\vec x)\cdot v^{-1}$, the preimage $S=\pi^{-1}_{u,v}(A)$ has measure $\mu(S)=\nu(uAv)\ge \e$ and hence $|S|=\mu(S)\cdot|F^N|\ge\e|F^N|$. By the choice of $N$, the set $S$ contains an image $\xi(F)$ of $F$ under some injective function $\xi=(\xi)_{i=1}^N:F\to F^N$ whose components $\xi_i:F\to F$ are identity functions or constants. It follows that $f=\pi_{u,v}\circ \xi:F\to G$ is a function of the form $f(x)=a_0xa_1\cdots xa_n$ for some $n\le N$ and some elements $a_0,\dots,a_n\in G$. Moreover, $f(F)=\pi_{u,v}\circ\xi(F)\subset \pi_{u,v}(S)\subset A$.
\end{proof}

Theorem~\ref{t4.4n} implies the following density version of the Van der Waerden Theorem proved by Szemer\'edi \cite{Szemeredi}.

\begin{corollary}[Szemer\'edi]\label{c:Wan} Each Solecki positive subset of integers contains arbitrarily long arithmetic progressions.
\end{corollary}

One of brightest recent results of Ramsey Theory is the Green-Tao Theorem \cite{GT} which says that the set of prime numbers $P$ contains arbitrarily long arithmetic progressions. It should be mentioned that this theorem cannot be derived from Corollary~\ref{c:Wan} as the set of primes is Solecki null, as shown in the following example.

\begin{example} The set  of prime numbers $P$ is Solecki null in the additive group of integers $\IZ$.
\end{example}

\begin{proof} Let $P=\{p_k\}_{k=1}^\infty$ be the increasing enumeration of prime numbers. For every $k\in\IN$ let $n_k=p_1\cdots p_k$ be the product of first $k$ prime numbers. Let us recall \cite[\S5.5]{HW} that the Euler function $\phi:\IN\to\IN$ assigns to each $n\in \IN$ the number of positive integers $k\le n$ which are relatively prime with $n$.
It is well-known that $\phi(p)=p-1$ for each prime number $p$ and by the multiplicativity of the Euler function, $\phi(n_k)=\phi(p_1\cdots p_k)=\prod_{i=1}^k(p_i-1)$ for every $k\in\IN$.
By Merten's Theorem \cite[\S22.8]{HW}, $$\lim_{k\to\infty}\frac{\phi(n_k)}{n_k}=\lim_{k\to\infty}\prod_{i=1}^k\Big(1-\frac1{p_i}\Big)=0.$$

Observe that for every $k\in\IN$ the set $A_k=\bigcup_{i=1}^kp_i\IZ$ coincides with the set of numbers which are not relatively prime with $n_k=p_1\cdots p_k$. Consequently, for the finite set $F_k=\{n\in\IZ:0<n\le n_k\}$ we get $|F_k\setminus A_k|=\phi(n_k)$. Observe that for every $x\in n_k\IZ$ the equality $x+A_k=A_k=-x+A_k$ implies $|(x+F_k)\setminus A_k|=|F_k\setminus (-x+A_k)|=\phi(n_k)$.
Since the set $P_k=P\setminus \{p_1,\dots,p_k\}$ is contained in $\IZ\setminus A_k$, we have an upper bound $|(x+F_k)\cap P_k|\le |(x+F_k)\setminus A_k|=\phi(n_k)$ for every $x\in n_k\IZ$. Given any integer number $y$, find an integer number $a\in\IZ$ such that $an_k<y\le (a+1)n_k$ and observe that $y+F_k\subset (an_k+F_k)\cup((a+1)n_k+F_k)$. Consequently,
$|(y+F_k)\cap P_k|\le |(an_k+F_k)\cap P_k|+|((a+1)n_k+F_k)\cap P_k)|\le 2\phi(n_k)$ and finally
$|(y+F_k)\cap P|\le|\{p_1,\dots,p_k\}|+|(y+F_k)\cap P_k|\le k+2\phi(n_k)$.

Applying Merten's Theorem \cite[\S22.8]{HW}, we  get the upper bound  $$\sigma(P)\le \inf_{k\in\IN}\sup_{y\in\IZ}\frac{|(y+F_k)\cap P|}{|F_k|}\le\lim_{k\in\IN}
\Big(\frac{k}{n_k}+2\frac{\phi(n_k)}{n_k}\Big)\le 0+2\lim_{k\to\infty}\prod_{i=1}^k\Big(1-\frac1{p_i}\Big)=0$$which implies the desired equality $\sigma(P)=0$.
\end{proof}

\section{The Solecki submeasure of subsets of small cardinality in groups}

In this section we shall evaluate the Solecki submeasure of sets of small cardinality in infinite groups. We start with two trivial propositions.

\begin{proposition}\label{finite} Each finite subset $A$ of an infinite group $G$ is Solecki null.
\end{proposition}

\begin{proof} Given any $\e>0$ take a finite subset $F\subset G$ of cardinality $|F|>|A|/\e$ and observe that $\sup_{x,y\in G}\frac{|F\cap xAy|}{|F|}\le\frac{|A|}{|F|}<\e$. So, $\sigma(A)=0$.
\end{proof}

\begin{proposition}\label{small-solnul} Any subset $A\subset G$ of cardinality $|A|<|G|$ in an infinite group $G$ is right-Solecki null in $G$.
\end{proposition}

\begin{proof} If the group $G$ is countable, then the conclusion follows from Proposition~\ref{finite}
and Theorem~\ref{subadit-sol}.
If $G$ is uncountable, then subgroup $H$ generated by $A$ has cardinality $|H|\le\max\{|A|,\aleph_0\}<|G|$ and hence has infinite index in $G$. By Proposition~\ref{zero-index}, the subgroup $H$ (and its subset $A$) is right-Solecki null.
\end{proof}

\begin{remark}\label{pathology} Example~\ref{permutation} implies that for each infinite cardinal $\kappa$ there is a locally finite (and hence amenable) group $G$ of cardinality $|G|=\kappa$ containing a countable subgroup $H\subset G$ which is Solecki one and right-Solecki null. This shows that Proposition~\ref{finite} cannot be generalized to uncountable group.
\end{remark}

However, Theorem~\ref{t-FC} and Proposition~\ref{small-solnul} imply:

\begin{corollary}\label{FC-zero} Any subset $A$ of cardinality $|A|<|G|$ in an infinite FC-group $G$ is  Solecki null.
\end{corollary}

A similar result holds also for compact Hausdorff topological groups. {\em All compact topological groups considered in this section are Hausdorff}. By $\cov(\M)$ (resp. $\cov(\E)$) we denote the smallest cardinality of a cover of an infinite compact metrizable group by meager subsets (resp. closed Haar null sets). It is known  that $\w_1\le\cov(\M)\le\cov(\E)\le\mathfrak c$ and the position of the cardinals $\cov(\M)$ and $\cov(\E)$ in the interval $[\w_1,\mathfrak c]$ depends on additional set-theoretic axioms (see \cite{BJ}, \cite{BS}). By \cite[7.13]{Blass}, the equality $\cov(\M)=\mathfrak c$ is equivalent to Martin's Axiom for countable posets.

\begin{theorem}\label{compact-zero} If a group $G$ admits a homomorphism $h:G\to H$ onto an infinite compact topological group $H$, then each subset $A\subset G$ of cardinality $|A|<\cov(\E)$ is Solecki null.
\end{theorem}

\begin{proof} We divide the proof of this theorem into a series of lemmas. In the proofs of these lemmas we shall use a well-known fact \cite{vN} that each compact topological group $G$ carries a Haar measure $\lambda$  (i.e., the unique invariant probability regular $\sigma$-additive measure $\lambda$ defined on the $\sigma$-algebra of Borel subsets of $G$). A subset $A\subset G$ will be called {\em Haar null} if $\lambda(A)=0$.

\begin{lemma} For any finite subset $T$ of a compact topological group $G$ and any $n\in\IN$ the set
$$G^n_T=\big\{(x_1,\dots,x_n)\in G^n:\exists x,y\in G\;\;xTy\subset\{x_1,\dots,x_n\}\big\}$$is closed in $G^n$.
\end{lemma}

\begin{proof} The set $G^n_T$ is closed being the continuous image of the closed subset
$$\big\{(x_1,\dots,x_n,x,y)\in G^n\times G^2:xTy\subset \{x_1,\dots,x_n\}\big\}$$
of the compact Hausdorff space $G^n\times G^2$.
\end{proof}

\begin{lemma}\label{Lie} For any $2$-element subset $T$ of an infinite connected compact Lie group $G$ and every $n\ge 2$ the closed set $G^n_T$ is Haar null in the compact topological group $G^n$.
\end{lemma}

\begin{proof} Replacing the set $T$ by a suitable shift, we can assume that $T$ contains the unit $1_G$ of the group $G$. In this case $T=\{1_G,t\}$ for some element $t\in G\setminus\{1_G\}$. Observe that a subset $\{x_1,\dots,x_n\}$ contains a shift $xTy$ for some $x,y\in G$ if and only if there are two distinct indices $1\le i,j\le n$ such that $x_i=xy$ and $x_j=xty$. In this case $x_jx^{-1}_i=xtyy^{-1}x^{-1}=xtx^{-1}\in t^G$. The conjugacy class $t^G$, being a closed submanifold of $G$ is Haar null. Then the set $G^n_T$ also is Haar null, being the finite union $G^n_T=\bigcup_{i\ne j}\big\{(x_1,\dots,x_n)\in G^n:x_jx^{-1}_i\in t^G\}$ of Haar null sets.
\end{proof}

\begin{remark} The connectedness of the Lie group $G$ in Lemma~\ref{Lie} is essential as shown by the example of the orthogonal group $G=O(2)$. It is easy to check that for any 2-element set $T=\{1_G,t\}\subset O(2)$ containing the unit $1_G$ and a reflection $t\in O(2)\setminus SO(2)$ (i.e., an orientation reversing isometry of $\mathbb R^2$) the set $G^2_T$ has Haar measure $\lambda(G^2_T)=\frac12$.
\end{remark}

A topological group $G$ is called {\em profinite} if it embeds into a Tychonoff product of finite groups.

\begin{lemma}\label{l3.9} For any $3$-element set $T$ in an infinite profinite compact topological group $G$ and any $n\ge 3$ the closed set $G^n_T$ is Haar null in $G^n$.
\end{lemma}

\begin{proof} It suffices to show that the set $G^n_T$ has Haar measure $\lambda(G^n_T)<\e$ for any $\e>0$. Since the group $G$ is infinite and profinite, there is a continuous surjective homomorphism $h:G\to H$ onto a finite group $H$ of cardinality $|H|>n(n-1)(n-2)/\e$ such that the restriction $h|T$ is injective. Then the subset $T'=h(T)$ of the group $H$ has cardinality $|T'|=3$. The homomorphism $h$ induces a homomorphism $h^n:G^n\to H^n$, $h^n:(x_1,\dots,x_n)\mapsto (h(x_1),\dots,h(x_n))$.

Observe that $h^n(G_T^n)\subset H^n_{T'}$, which implies that the Haar measure of $G^n_T$ does not exceed the Haar measure of $H^n_{T'}$. Taking into account that
$$
\begin{aligned}
H^n_{T'}&=\big\{(x_1,\dots,x_n)\in H^n:\exists x,y\in H\;xT'y\subset\{x_1,\dots,x_n\}\big\}=\\
&=\bigcup_{x,y\in H}\bigcup_{1\le i<i<k\le n}\{(x_1,\dots,x_n)\in H^n:xT'y=\{x_i,x_j,x_k\}\big\}
\end{aligned}
$$
and
$$\big|\{(x_1,\dots,x_n)\in H^n:xT'y=\{x_i,x_j,x_k\}\big\}\big|=6\cdot|H|^{n-3}$$for all $x,y\in H$ and $1\le i<j<k\le n$, we conclude that $$H^n_{T'}\le|H|^2\cdot\binom{n}{3}\cdot 6\cdot|H|^{n-3}=n(n-1)(n-2)\cdot|H|^{n-1}<\e\cdot|H|^n.$$ Consequently the sets $H^n_{T'}$ and $G^n_T$ have Haar measure $<\e$ in the groups $H^n$ and $G^n$, respectively.
\end{proof}

\begin{lemma}\label{l4.10n} If a group $G$ admits a homomorphism $h:G\to H$ onto an infinite compact topological group $H$, then for each subset $A\subset G$ of cardinality $|A|<\cov(\E)$  and every $n\ge3$ there is an $n$-element set $F\subset G$ such that $|F\cap xAy|\le 2$ for all $x,y\in G$. Consequently, $\sigma(A)=0$.
\end{lemma}

\begin{proof} Fix $n\ge 3$ and a subset $A\subset G$ of cardinality $|A|<\cov(\E)$. Depending on the properties of the compact group $H$ we shall separately consider two cases.
\smallskip

1. The infinite compact group $H$ is profinite. In this case $H$ admits a homomorphism onto a infinite metrizable profinite compact topological group. So, we lose no generality assuming that the group $H$ is metrizable. Given any subset $A\subset G$ of cardinality $|A|<\cov(\E)$, consider its image $B=h(A)\subset H$. Then the family $[B]^3$ of all 3-element subsets of $B$ has cardinality $|[B]^3|<\cov(\E)$.  By Lemma~\ref{l3.9}, for every $T\in[B]^3$ the set $H^n_T$ is closed and Haar null in the compact group $H^n$.
Since the diagonal of the square $H\times H$ is a subgroup of infinite index in $H\times H$, it has  Haar measure zero in $H\times H$. This fact can be used to show that the set $$\Delta H^{n}=\big\{(x_1,\dots,x_n)\in H^n:|\{x_1,\dots,x_n\}|<n\big\}$$ is closed and Haar null in the compact topological group $H^n$.
Since $|[B]^3|<\cov(\E)$, the union $\Delta H^{n}\cup \bigcup_{T\in[B]^3}H^n_T$ does not cover the compact metrizable group $H^n$. So, we can find a vector $(x_1,\dots,x_n)\in H^n$ which does not belong to this union. Since $(x_1,\dots,x_n)\notin \Delta H^n$, the set $F'=\{x_1,\dots,x_n\}$ has cardinality $|F'|=n$. We claim that $|F'\cap xBy|\le 2$ for any points $x,y\in H$. Assuming the converse, we can find a 3-element subset $T\subset B$ such that $xTy\subset F'$ for some $x,y\in H$. But this contradicts the choice of the vector $(x_1,\dots,x_n)\notin H^n_T$.

Choose any finite set $F\subset G$ such that the restriction $h|F:F\to F'$ is a bijective map. Then for any points $x,y\in G$ we get $|F\cap xAy|\le |F\cap xh^{-1}(B)y|=|F'\cap h(x)Bh(y)|\le 2$. It follows that $\sigma(A)\le \frac2{|F|}=\frac{2}n$ for all $n\ge 3$ and hence $\sigma(A)=0$.
\smallskip

2. The compact group $H$ is not profinite. In this case by \cite[9.1]{HM}, $H$ admits a continuous homomorphism onto an infinite Lie group and we lose no generality assuming that $H$ is an infinite Lie group. It follows that the connected component $L$ of the unit $1_H$ is an open normal subgroup of finite index in $H$ and hence $L$ is an infinite connected Lie group. Let $S\subset H$ be a finite subset such that $SL=H=LS$. Since the set $B=L\cap (S\cdot h(A)\cdot S)$ has cardinality $|B|\le|S|\cdot|A|\cdot|S|<\cov(\E)$, the family $[B]^2$ of all 2-element subsets of $B$ also has cardinality $|[B]^2|<\cov(\E)$. By Lemma~\ref{Lie}, for every  $T\in[B]^2$ the set $L^n_T$ is closed and Haar null in the connected Lie group $L^n$.
Since the set $\Delta L^{n}=\{(x_1,\dots,x_n)\in L^n:|\{x_1,\dots,x_n\}|<n\}$ is closed and Haar null in  $L^n$ and $|[B]^2|<\cov(\E)$, the union $\Delta L^{n}\cup \bigcup_{T\in[B]^2}L^n_T$ does not cover the compact metrizable group $L^n$. So, we can find a vector $(x_1,\dots,x_n)\in L^n$ which does not belong to this union. Since $(x_1,\dots,x_n)\notin \Delta L^n$, the set $F'=\{x_1,\dots,x_n\}$ has cardinality $|F'|=n$. We claim that $|F'\cap xh(A)y|\le 1$ for any points $x,y\in H$. Assuming the converse, we could find a 2-element set $T\subset h(A)$ such that $xTy\subset F'\subset L$ for some points $x,y\in H$. It follows from $H=SL=LS$ that $x=ua$ and $y=bv$ for some elements $a,b\in S$ and $u,v\in L$. On the other hand, $uaTbv=xTy\subset L$ implies $aTb\subset u^{-1}Lv^{-1}=L$ and  $aTb\subset L\cap Sh(A)S=B$.
Since $(x_1,\dots,x_n)\notin L^n_{aTb}$ we get $xTy=uaTbv\not\subset \{x_1,\dots,x_n\}=F'$, which is a desired contradiction showing that $|F'\cap xh(A)y|\le 1$ for all $x,y\in H$.

Choose any finite set $F\subset G$ such that the restriction $h|F:F\to F'$ is a bijective map. Then for any points $x,y\in G$ we get $|F\cap xAy|\le |F\cap xh^{-1}(h(A))y|=|F'\cap h(x)h(A)h(y)|\le 1$. It follows that $\sigma(A)\le \frac1{|F|}=\frac1n$ for all $n\ge 3$ and hence $\sigma(A)=0$.
\end{proof}
Lemma~\ref{l4.10n} completes the proof of Theorem~\ref{compact-zero}.
\end{proof}

Comparing Corollary~\ref{FC-zero} and Theorem~\ref{compact-zero} it is natural to ask:

\begin{question} Is $\sigma(A)=0$ for any subset $A$ of cardinality $|A|<|G|$ in an infinite (metrizable) compact topological group $G$?
\end{question}

Example~\ref{permutation} and Theorem~\ref{compact-zero} yield a measure-theoretic proof of the following known fact (for an alternative proof see \cite{BGP} and \cite{BG}).

\begin{corollary} The group $FS_X$ of finitely supported bijective transformations of an infinite set $X$ admits no homomorphism onto an infinite compact topological group.
\end{corollary}

\section{The Solecki submeasure on non-meager topological groups}

In this section we study the properties of the Solecki submeasure on non-meager topological groups.
The topological homogeneity of a topological group $G$ implies that $G$ is non-meager if and only if $G$ is {\em Baire} in the sense that the intersection $\bigcap_{n\in\w}U_n$ of a countable family of open dense subsets of $G$ is dense in $G$.

\begin{proposition}\label{p5.1} Each dense $G_\delta$-subset $A$ of a non-meager topological group $G$ is right-Solecki one.
\end{proposition}

\begin{proof} Given a finite set $F\subset G$ observe that for each $x\in F$ the shift $x^{-1}A$ is a dense $G_\delta$-set in $G$. Since the topological group $G$ is Baire, the intersection $\bigcap_{x\in F}x^{-1}A$ is not empty and hence contains some point $y\in G$. For this point $y$ we get $Fy\subset A$, which means that $A$ is right-Solecki one according to Proposition~\ref{p:rizne}.
\end{proof}

Let us recall that a subset $A$ of a topological space $X$ has the {\em Baire Property} if for some open set $U\subset X$ the symmetric difference $A\triangle U=(A\setminus U)\cup (U\setminus A)$ is meager in $X$. It is known \cite[8.22]{Ke} that the family of sets with the Baire Property is a $\sigma$-algebra containing all Borel subsets of $X$. A topological group $G$ is called {\em totally bounded} if each non-empty open subset is large in $G$.

\begin{proposition}\label{p5.2} Each right-Solecki null set with Baire Property in a totally bounded topological group $G$ is meager. In particular, each Borel Solecki null set in $G$ is meager.
\end{proposition}

\begin{proof} Given a Solecki null set $A$ with the Baire Property in $G$, we need to show that $A$ is meager in $G$. Assume conversely that $A$ is not meager. In this case the topological group $G$ is not meager and hence is Baire. Since $A$ has the Baire Property in $G$, there is an open set $U\subset G$ such that the symmetric difference $A\triangle U$ is meager in $G$ and hence can be enlarged to a meager $F_\sigma$-set $M\subset G$. Since $A$ is not meager, the open set $U$ is not empty and hence is a   Baire space. Then the complement $U\setminus M$ is a dense $G_\delta$-set in $U$. The total boundedness of the group $G$ implies that $UF=G$ for some finite subset $F\subset G$. By Proposition~\ref{p5.1}, the dense $G_\delta$-set $(U\setminus M)F$ in $G$ is right-Solecki one. Now Proposition~\ref{semiadditive}(2) implies that the set $U\setminus M$ is right-Solecki positive,  which is a desired contradiction.
\end{proof}

Proposition~\ref{p5.2} cannot be reversed as shown by the following theorem proved by Solecki in \cite{Sol2001}. This theorem can be considered as a topological counterpart of Proposition~\ref{p2.6}.

\begin{theorem}[Solecki] Let $G$ be a non-locally compact Polish group whose topology is generated by an invariant metric. Then there exists a closed subset $F\subset G$ and a continuous map $f:F\to \{0,1\}^\w$ such that for each $y\in\{0,1\}^\w$ the preimage $f^{-1}(y)$ is thick and hence  right-Solecki one in $G$.
\end{theorem}

\section{The Solecki submeasure versus the Haar submeasure on groups}\label{Bohr}

In this section we shall prove that the Solecki submeasure does not exceed the Haar submeasure. The Haar submeasure can be defined on each group with help of its Bohr compactification. The {\em Bohr compactification} of a group $G$ is a pair $(bG,\eta)$ consisting of a compact Hausdorff topological group $bG$ and a homomorphism $\eta:G\to bG$ such that for each homomorphism $f:G\to K$ to a compact topological group $K$ there is a unique continuous homomorphism $\bar f:bG\to K$ such that $f=\bar f\circ \eta$. The uniqueness of $\bar f$ implies that the subgroup $\eta(G)$ is dense in the compact topological group $bG$.

It is well-known that each group $G$ has a Bohr compactification, which is unique up to an isomorphism, see \cite[\S3.1]{BJM}. There are groups with trivial Bohr compactification. For example, so is the permutation group $S_X$ of an infinite set $X$ (this can be derived from \cite{Gau}, \cite{DS} or \cite{BGP}).

A subset $U\subset G$ of a group $G$ is called {\em Bohr open} if $U=\eta^{-1}(V)$ for some open subset $V\subset bG$. Bohr open subsets of a group $G$ form a topology called the {\em Bohr topology} on $G$. This is the largest totally bounded group topology on $G$. This topology needs not be Hausdorff. For example, the Bohr topology on the permutation group $S_X$ of an infinite set $X$ is anti-discrete.

The Bohr compactification $bG$, being a compact Hausdorff topological group, carries the Haar measure $\lambda$.
We recall that the {\em Haar measure} on a compact topological group $K$ is the unique invariant regular probability $\sigma$-additive measure $\lambda:\mathcal B(K)\to[0,1]$ defined on the $\sigma$-algebra $\mathcal B(K)$ of all Borel subsets of $K$. The {\em regularity} of $\lambda$ means that $$\lambda_*(B)=\lambda(B)=\lambda^*(B)$$ for each Borel subset $B$ of $K$. Here
$$
\lambda_*(B)=\sup\{\lambda(F):\mbox{$F\subset B$ is closed in $K$}\}\mbox{ \ and \ }
\lambda^*(B)=\inf\{\lambda(U):\mbox{$U\supset B$ is open in $K$}\}
$$
are the {\em lower} and {\em upper Haar measures} of a set $B\subset K$.

For each group $G$ the Haar measure $\lambda$ on its Bohr compactification $bG$ induces the {\em Haar submeasure} $$\bar \lambda:\mathcal P(G)\to[0,1],\;\;\bar\lambda:A\mapsto \lambda(\overline{\eta(A)}),$$ on $G$, assigning to each subset $A\subset G$ the Haar measure $\lambda(\overline{\eta(A)})$ of the closure of its image $\eta(A)$ in $bG$.

The Solecki and Haar submeasures relate as follows.

\begin{theorem}\label{bohr} Each subset $A$ of a group $G$ has Solecki submeasure $\sigma(A)\le\bar\lambda(A)$.
\end{theorem}

\begin{proof} Let $(bG,\eta)$ be a Bohr compactification of $G$ and $B$ be the closure of the set $\eta(A)$ in $bG$.

To prove the theorem, it suffices to check that $\sigma(A)\le\lambda(B)+\e$ for every $\e>0$. By the regularity of the Haar measure $\lambda$ and the normality of the compact Hausdorff space $bG$, the closed set $B$ has a closed neighborhood $\bar O(B)$ in $bG$ such that $\lambda(\bar O(B))<\lambda(B)+\e$. Let $1_{bG}$ denote the unit of the group $bG$. Since $1_{bG}\cdot B\cdot 1_{bG}=B\subset \bar O(B)$, the compactness of $B$ and the continuity of the group operation yield an open neighborhood $V\subset bG$ of $1_{bG}$ such that $VBV\subset \bar O(B)$. Then $\overline{VBV}\subset \bar O(B)$ and hence $\lambda(x\overline{VBV}y)=\lambda(\overline{VBV})\le\lambda(\bar O(B))<\lambda(B)+\e$ for any points $x,y\in bG$. The density of $\eta(G)$ in $bG$ implies that  $bG=\bigcup_{x\in\eta(G)}xV=\bigcup_{x\in\eta(G)}Vx$. By the compactness of $bG$ there is a finite set $F\subset \eta(G)$ such that $G=FV=VF$.

Let $P_\sigma(G)$ be the space of all probability regular Borel $\sigma$-additive measures on $G$ endowed with the topology generated by the subbase consisting of the sets $\{\mu\in P_\sigma(G):\mu(U)>a\}$ where $U$ is an open subset in $G$ and $a\in\mathbb R$. It follows that for each closed set $C\subset G$ the set $$\{\mu\in P_\sigma(G):\mu(C)<a\}=\{\mu\in P_\sigma(G):\mu(G\setminus C)>1-a\}$$is open in $P_\sigma(G)$.
Consequently, the set $$O_\lambda=\bigcap_{x,y\in F}\{\mu\in P_\sigma(G):\mu(x\overline{VBV}y)<\lambda(B)+\e\}$$is an open neighborhood of the Haar measure $\lambda$ in the space $P_\sigma(G)$.

Since $\eta(G)$ is a dense subset in $bG$, the subspace $P_\w(\eta(G))$ of finitely supported probability measures on $\eta(G)$ is dense in the space $P_\sigma(bG)$ (see e.g. \cite{Var} or \cite[1.9]{Fed}). Consequently, the open set $O_\lambda$ contains some probability measure $\mu\in P_\w(\eta(G))$ and we can find a finitely supported probability measure $\nu$ on $G$ such that $\eta(\nu)=\mu$. The latter equality means that $\mu(C)=\nu(\eta^{-1}(C))$ for all $C\subset bG$ and hence $\nu(D)\le\nu\big(\eta^{-1}(\eta(D))\big)=\mu(\eta(D))$ for each set $D\subset G$. We claim that $\sup_{x,y\in G}\nu(xAy)\le \sigma(A)+\e$. Indeed, since $bG=FV=VF$, for any points $x,y\in G$ we can find points $x',y'\in F$ such that $\eta(x)\in x'V$ and $\eta(y)=Vy'$. Then
$$\nu(xAy)\le \mu(\eta(x)\eta(A)\eta(y))\le\mu(\eta(x)B\eta(y))\le \mu(x'VBVy')\le \mu(x'\overline{VBV}y')<\lambda(B)+\e=\bar\lambda(A)+\e$$as $\mu\in O_\lambda$.
So, $\sigma(A)\le\sup_{x,y\in G}\nu(xAy)\le\bar\lambda(A)+\e$. Since the number $\e>0$ was arbitrary, we conclude that $\sigma(A)\le\bar\lambda(A)$.
\end{proof}

\section{The Solecki submeasure versus Haar measure on compact topological groups}

In this section we shall study the relation between the Solecki submeasure $\sigma$ and the Haar measure $\lambda$ on a compact Hausdorff topological group $G$. Some results remain true also for the right  Solecki submeasure  $\hat\sigma_R$, which does not exceed the Solecki submeasure $\sigma$ according to Theorem~\ref{subadit-sol}.

For a subset $A$ of $G$ by $\bar A$ and $A^\circ$ we shall denote the closure and the interior of $A$ in $G$, respectively. The difference $\partial A=\bar A\setminus A^\circ$ is the boundary of $A$ in $G$. Besides the interior $A^\circ$ we can assign to $A$ another canonical open set $A^\bullet$ called the {\em comeager interior} of $A$. By definition, $A^\bullet$ is the largest open set in $G$ such that $A^\bullet \setminus A$ is meager in $G$. It is easy to see that $A^\circ\subset A^\bullet\subset \bar A$. Observe that a set $A\subset X$ has the Baire Property if and only if the symmetric difference $A\triangle A^\bullet$ is meager.

It turns out that the Haar measure $\lambda$ on a compact topological group $G$ nicely agrees with the Solecki submeasure $\sigma$ (at least on the family of all closed subsets). We recall that $\lambda_*(A)=\sup\{\lambda(F):F=\bar F\subset A\}$ for $A\subset G$.

\begin{theorem}\label{t6.1} Each subset $A$ of a compact topological group $G$ has
$$\max\{\lambda_*(A),\lambda(A^\bullet))\le\sigma(A)\le\lambda(\bar A).$$
\end{theorem}

\begin{proof} We divide the proof of this theorem into five lemmas. In these lemmas we assume that $G$ is a compact topological group and $\lambda$ is the Haar measure on $G$.

\begin{lemma}\label{l6.3} $\lambda(A^\circ)\le\hat\sigma_R(A)\le\sigma(A)\le\lambda(\bar A)$ for each subset $A\subset G$.
\end{lemma}

\begin{proof} The group $G$, being compact, can be identified with its Bohr compactification $bG$. By Theorems~\ref{subadit-sol} and~\ref{bohr}, $\hat\sigma_R(A)\le\sigma(A)\le\sigma(\bar A)\le\lambda(\bar A)$.
 The subadditivity of the right-Solecki submeasure $\hat\sigma_R$ guarantees that $1=\hat\sigma_R(G)\le \hat\sigma_R(A^\circ)+\hat\sigma_R(G\setminus A^\circ)$. Since the set $G\setminus A^\circ$ is closed in $G$, Theorem~\ref{bohr} guarantees that $\hat\sigma_R(G\setminus A^\circ)\le\sigma(G\setminus A^\circ)\le\lambda(G\setminus A^\circ)$ and hence
$$\hat\sigma_R(A)\ge\hat\sigma_R(A^\circ)\ge 1-\hat\sigma_R(G\setminus A^\circ)\ge 1-\lambda(G\setminus A^\circ)=\lambda(A^\circ).$$
\end{proof}

\begin{lemma}\label{l6.4} $\hat\sigma_R(A)=\sigma(A)=\lambda(A)$ for any subset $A\subset G$ whose boundary $\partial A=\bar A\setminus A^\circ$ has Haar measure $\lambda(\partial A)=0$.
\end{lemma}

\begin{proof} The additivity of the Haar measure $\lambda$ guarantees that
$$\lambda(\bar A)=\lambda(A^\circ)+\lambda(\partial A)=\lambda(A^\circ)+0\le\lambda(A)\le\lambda(\bar A)$$and hence $\lambda(A^\circ)=\lambda(A)=\lambda(\bar A)$. Now the equality $\lambda(A)=\sigma(A)$ follows from Lemma~\ref{l6.3}.
\end{proof}

For the Solecki submeasure $\sigma$ we can prove more:

\begin{lemma}\label{l6.5} $\sigma(A)=\lambda(A)$ for each closed subset $A\subset G$.
\end{lemma}

\begin{proof} By Lemma~\ref{l6.3}, $\sigma(A)\le\lambda(A)$. So, it remains to show that $\sigma(A)\ge\lambda(A)$. Assuming conversely that $\sigma(A)<\lambda(A)$ we conclude that the number $\e=\frac12(\lambda(A)-\sigma(A))$ is positive. Then $\sigma(A)<\lambda(A)-\e$ and by the definition of the Solecki submeasure, there is a finitely supported probability measure $\mu$ on $G$ such that $\sup_{x,y\in G}\mu(xAy)<\lambda(A)-\e$. For each pair $(x,y)\in G\times G$, by the regularity of the measure $\mu$,  there is an open neighborhood $O_{x,y}(A)\subset G$ of $A$ such that $\mu(xO_{x,y}(A)y)<\lambda(A)-\e$. Using the compactness of $A$, we can find an open neighborhood $U_{x,y}\subset G$ of $1_G$ such that $U_{x,y}AU_{x,y}\subset O_{x,y}(A)$. The continuity of the group operation at $1_G$ yields an open neighborhood $V_{x,y}\subset G$ of $1_G$ such that $V_{x,y}\cdot V_{x,y}\subset U_{x,y}$. By the compactness of the space $G\times G$ the open cover $\{xV_{x,y}\times V_{x,y}y:(x,y)\in G\times G\}$ of $G\times G$ has a finite subcover $\{xV_{x,y}\times V_{x,y}y:(x,y)\in F\}$ where $F$ is a finite subset of $G\times G$. Consider the open neighborhood $V=\bigcap_{(x,y)\in F}V_{x,y}$ of $1_G$ and the open neighborhood $V\kern-1pt AV$ of the closed set $A$. By the Urysohn Lemma \cite[1.5.10]{En}, there is a continuous function $f:G\to[0,1]$ such that $f(A)\subset \{0\}$ and $f(G\setminus V\kern-1pt AV)\subset \{1\}$. By the $\sigma$-additivity of the Haar measure $\lambda$, there is a number $t\in(0,1)$ whose preimage $f^{-1}(t)$ has Haar measure $\lambda(f^{-1}(t))=0$. In this case the open neighborhood $W=f^{-1}\big([0,t)\big)\subset V\kern-1pt AV$ of $A$ has boundary $\partial W\subset f^{-1}(t)$ of Haar measure zero. By Lemma~\ref{l6.4}, $\sigma(W)=\lambda(W)$.

We claim that $\mu(aWb)<\lambda(A)-\e$ for any points $a,b\in G$. Since $\{xV_{x,y}\times V_{x,y}y:(x,y)\in F\}$ is a cover of $G\times G$, there is a pair $(x,y)\in F$ such that $a\in xV_{x,y}$ and $b\in V_{x,y}y$. Then $$aWb\subset aVAVb\subset xV_{x,y}V\kern-1pt AVV_{x,y}y\subset xV_{x,y}V_{x,y}AV_{x,y}V_{x,y}y\subset xU_{x,y}AU_{x,y}y\subset xO_{x,y}(A)y$$and hence
$$\mu(aWb)\le\mu(xO_{x,y}(A)y)<\lambda(A)-\e.$$
By Lemma~\ref{l6.4},
$$\sigma(W)\le\sup_{a,b\in G}\mu(aWb)\le\lambda(A)-\e<\lambda(W)=\sigma(W),$$
which is a desired contradiction. So, $\sigma(A)=\lambda(A)$.
\end{proof}

\begin{lemma}\label{l6.6} $\lambda_*(A)\le\sigma(A)$ for each subset $A\subset G$.
\end{lemma}

\begin{proof}
By Lemma~\ref{l6.5} and the monotonicity of the Solecki submeasure, we get
$$\lambda_*(A)=\sup\{\lambda(F):F=\bar F\subset A\}=\sup\{\sigma(F):F=\bar F\subset A\}\le\sigma(A).$$
\end{proof}

\begin{lemma}\label{l10.6} $\lambda(A^\bullet)\le\sigma(A)$ for each subset $A\subset G$.
\end{lemma}

\begin{proof} Assume conversely that $\sigma(A)<\lambda(A^\bullet)$ and put $\e=\frac12(\lambda(A^\bullet)-\sigma(A))$. Since $\sigma(A)<\lambda(A^\bullet)-\e$, by Theorem~\ref{sol}, there is a finite subset $F\subset G$ such that $\sup_{x,y\in G}|xFy\cap A|/|F|<(\lambda(A^\bullet)-\e)$.
By the regularity of the Haar measure $\lambda$, some compact set $K\subset A^\bullet$ has Haar measure $\lambda(K)>\lambda(A^\bullet)-\e$. By Lemma~\ref{l6.5}, $\lambda(K)=\sigma(K)\le\max_{x,y\in G}|xFy\cap K|/|F|$. So, there are points $u,v\in G$ such that $|uFv\cap A^\bullet|\ge |uFv\cap K|\ge \lambda(K)\cdot|F|$. Let $T=\{t\in F:utv\in A^\bullet\}$ and observe that $|T|=|uFv\cap A^\bullet|\ge\lambda(K)\cdot|F|$. For every $t\in T$ consider the homeomorphism $s_t:G\to G$, $s_t:x\mapsto xtv$, and observe that $s_t^{-1}(A^\bullet)$ is an open neighborhood of the point $u$. Since the set $A^\bullet\setminus A$ is meager in $G$ its preimage $s_t^{-1}(A^\bullet\setminus A)$ is a meager set in $G$. Since the space $G$ is compact and hence Baire, in the open neighborhood $V_u=\bigcap_{t\in T}s_t^{-1}(A^\bullet)$ of the point $u$ we can find a point $x\in V_u$ which does not belong to the meager set $\bigcup_{t\in T}s_t^{-1}(A^\bullet\setminus A)$. For this point $x$ we get $s_t(x)\in A$ for all $t\in T$, which implies that $xTv\subset A$ and then $|xFv\cap A|\ge|xTv\cap A|=|xTv|=|T|\ge\lambda(K)\cdot|F|>(\lambda(A^\bullet)-\e)\cdot|F|$, which contradicts the choice of $F$.
\end{proof}

Lemmas~\ref{l6.3}, \ref{l6.6} and \ref{l10.6} finish the proof of Theorem~\ref{t6.1}.
\end{proof}

\begin{remark} For a compact topological group $G$ the family
$$\A_0=\{A\subset G:\sigma(\partial A)=0\}=\{A\subset G:\lambda(\partial A)=0\}$$is an algebra of subsets of $G$. This algebra determines the Haar measure in the sense that a regular Borel $\sigma$-additive measure $\mu$ on $G$ coincides with the Haar measure $\lambda$ if $\mu|\A_0=\lambda|\A_0$. By Lemma~\ref{l6.4}, $\lambda|\A_0=\sigma|\A_0=\hat\sigma_R|\A_0$. This means that the Solecki submeasure $\sigma$ uniquely determines the Haar measure $\lambda$ on each compact topological group $G$. The same is true for the subadditivization $\hat\sigma_R$ of the right-Solecki density $\sigma_R$.
\end{remark}

\begin{problem} Let $A$ be a closed subset of a compact topological group $G$. Is $\hat\sigma_R(A)=\lambda(A)$?
\end{problem}

Looking at the lower bound $\max\{\lambda_*(A),\lambda(A^\bullet)\}\le\sigma(A)$ proved in Theorem~\ref{t6.1}, one can suggest that it can be improved to $\lambda_*(A\cup A^\bullet)\le\sigma(A)$. However this is not true.

\begin{example} The compact abelian group $\IT=\{z\in \IC:|z|=1\}$
contains a Borel subset $A$ such that
$$\frac14=\lambda(A)=\lambda(A^\bullet)=\sigma(A)<\lambda(A\cup A^\bullet)=\lambda(\bar A)=\frac12.$$
\end{example}

\begin{proof} Consider the open subset $U=\{e^{i\varphi}:0<\varphi<\pi/2\}\subset\IT$ of Haar measure $\lambda(U)=1/4$  and the countable dense subset  $Q=\{e^{i\varphi}:\varphi\in\pi\cdot\mathbb Q\}$ where $\IQ$ is the set of rational numbers. By the regularity of the Haar measure $\lambda$ on $\IT$, the set $U\setminus Q$ contains a $\sigma$-compact (meager) subset $K$ of Haar measure $\lambda(K)=\lambda(U\setminus Q)=\frac14$. Now consider the set $A=(U\setminus K)\cup (-K)$ where $-K=\{-z:z\in K\}$. The finite set $F=\{1,-1,i,-i\}$ witnesses that $\sigma(A)\le\sup_{x,y\in\IT}|xFy\cap A|/|F|=\frac14$.
It follows that, $A^\bullet=U$ and thus
$$
\frac14=\lambda(A)=\lambda(A^\bullet)\le\sigma(A)\le\frac14.$$
On the other hand,
$$\lambda(A\cup A^\bullet)=\lambda(U\cup(-K))=\frac14+\frac14=\frac12=\lambda(\bar U\cup(-\bar U))=\lambda(\bar A).$$
\end{proof}

Theorem~\ref{t6.1} implies:

\begin{corollary} In an infinite compact Hausdorff topological group $G$ each closed Haar null set is Solecki null and each Borel Solecki null set is meager and Haar null.
\end{corollary}

Finally we show that both inequalities $\max\{\lambda_*(A),\lambda(A^\bullet)\}\le\sigma(A)\le\lambda(\bar A)$ in Theorem~\ref{t6.1} can be strict.

\begin{proposition}\label{p6.10} Each infinite compact Hausdorff topological group $G$ contains
\begin{enumerate}
\item a dense $F_\sigma$-set with
$0=\lambda(A)=\lambda(A^\bullet)=\sigma(A)<\lambda(\bar A)=1$;
\item a dense $G_\delta$-set $B\subset G$ with
$0=\lambda(B)<\lambda(B^\bullet)=\sigma(B)=\lambda(\bar B)=1$;
\item a dense subset $C\subset G$ with
$0=\lambda_*(C)=\lambda(C^\bullet)<\sigma(C)=\lambda(\bar C)=1$.
\item If $G$ is topologically isomorphic to the product $G=\prod_{n\in\w}G_n$ of infinite compact topological groups, then $G$ contains a dense meager $F_\sigma$-set $D\subset G$ which is Haar null and Solecki one.
\end{enumerate}
\end{proposition}

\begin{proof} By \cite[9.1]{HM}, the group $G$ admits a continuous homomorphism $h:G\to \tilde G$ onto an infinite metrizable compact topological group $\tilde G$. By \cite[1.10]{HM} the homomorphism $h$ is an open map.
By $\lambda,\tilde\lambda$ we denote the Haar measures and by $\sigma,\tilde\sigma$ the Solecki submeasures on the groups $G,\tilde G$, respectively. The uniqueness of the Haar measure on the topological group $\tilde G$ implies that $\lambda(h^{-1}(B))=\tilde \lambda(B)$ for any Borel subset $B\subset \tilde G$.
\smallskip

1. The topological group $\tilde G$, being compact and metrizable, contains a countable dense subset $\tilde A$, which is Haar null (by the $\sigma$-additivity of the Haar measure $\tilde\lambda$).
By  Theorem~\ref{compact-zero}, $\tilde A$ is Solecki null in $\tilde G$. Since the homomorphism $h$ is continuous and open, the preimage $A=f^{-1}(\tilde A)$ is a dense meager $F_\sigma$-set in $G$. Taking into account that $A$ is meager in $G$, we get $A^\bullet=\emptyset$.
By Proposition~\ref{p2.3} the set $A=h^{-1}(\tilde A)$ has the Solecki submeasure $\sigma(A)=\tilde \sigma(\tilde A)=0$. The uniqueness of the Haar measure on the group $\tilde G$ implies that $\lambda(A)=\tilde \lambda(\tilde A)=0$. Now we see that
$0=\lambda(A)=\lambda(A^\bullet)=\sigma(A)<\lambda(\bar A)=1$.
\smallskip

2. By the regularity of the Haar measure $\lambda$, the dense $F_\sigma$-set $A$ can be enlarged to a dense $G_\delta$-set $B$ such that $\lambda(B)=\lambda(A)=0$. It follows that $B^\bullet=G$ and hence $\lambda(B^\bullet)=\lambda(\bar B)=1$. By Proposition~\ref{p5.1}, $\sigma(B)=1$.
\smallskip

3. By the Baire Theorem, the infinite compact Hausdorff group $G$ is uncountable and by Proposition~\ref{p2.6}, $G$ contains an uncountable disjoint family $\C$ of Solecki one sets.  By the $\sigma$-additivity of the Haar measure $\lambda$ on $G$, the subfamily $\C_1=\{C\in\C:\lambda_*(C)>0\}$ is at most countable. Since for any disjoint sets $A,B\subset G$ their comeager interiors $A^\bullet$ and $B^\bullet$ are disjoint, the family $\C_2=\{C\in\C:\lambda(C^\bullet)>0\}$ is at most countable. So, we can choose a set $C\in\C\setminus(\C_1\cup\C_2)$ and observe that
$$0=\lambda_*(C)=\lambda(C^\bullet)<\sigma(C)=\lambda(\bar C)=1.$$

4. Assume that $G=\prod_{n\in\w}G_n$ for suitable infinite compact topological groups $G_n$. For every $n\in\w$ consider the coordinate projection $\pr_n:G\to G_n$ and its kernel $\Ker(\pr_n)$, which is a compact subgroup of Haar measure zero in $G$. Then $D=\bigcup_{n\in\w}\Ker(\pr_n)$ is a dense Haar null $F_\sigma$-subset in $G$. Since $D$ is meager, its comeager interior $D^\bullet$ is empty. Consequently, $0=\lambda(D)=\lambda(D^\bullet)$ and $\lambda(\bar D)=\lambda(G)=1$. We claim that the set $D$ is Solecki one.

Given a finite set $F=\{x_1,\dots,x_n\}\subset G$, choose an element $g\in G$ such that $\pr_i(g)=\pr_i(x_i)$ for all $i\le n$. Then for every $i\le n$ we get $g^{-1}x_i\in\Ker(\pr_i)\subset D$, which implies $g^{-1}F\subset D$. So, the set $D$ is Solecki one according to Proposition~\ref{thick}.
\end{proof}

\begin{question} Does any infinite compact Hausdorff topological group $G$ contain an $F_\sigma$-set $D$ which is Haar null and Solecki one?
\end{question}

\section{The difference sets of right-Solecki positive sets in groups}\label{s:difference}

The right-Solecki density $\sigma^R$ and the right-Solecki submeasure $\hat\sigma_R$ are convenient instruments for generalization of many notions and results which were previously known in the context of Polish or amenable groups. A motivating example is the classical Steinhaus-Weil Theorem saying that for every measurable subset $A$ of positive Haar measure in a compact topological group $G$, the set $AA^{-1}$ is a neighborhood of the unit $1_G$ in $G$. We shall try to find a counterpart of this theorem replacing the Haar measure of $A$ by the (right) Solecki density $\sigma^R(A)$ of $A$.

We start with calculating the covering number of the difference set $AA^{-1}$.

For a non-empty subset $A$ of a group $G$ its {\em covering number} is defined as the cardinal
$$\cov(A)=\min\{|F|:F\subset G\mbox{  \ and \ }G=FA\}.$$
The  covering number $\cov(AA^{-1})$ of the difference set $AA^{-1}$ is bounded from above by
the {\em packing index}
$$
\pack(A)=\sup\big\{|E|:E\subset G,\,\;\forall x,y\in E\;\;(x\ne y\;\Ra\;xA\cap yA=\emptyset)\big\}
$$
of the set $A$. Packing indices of subsets in groups were studied in \cite{BL}, \cite{BL11}, \cite{BLR}, \cite{Lyas}, \cite{Prot}.

\begin{proposition}\label{p7.2n} For any non-empty subset $A$ of a group $G$ we get $\cov(AA^{-1})\le\pack(A).$
\end{proposition}

\begin{proof} By Zorn's Lemma, there is a maximal set $E\subset G$ such that for any distinct points $x,y\in E$ the sets $xA$ and $yA$ are disjoint. By the maximality of $E$, for each $g\in G$ there is an element
$e\in E$ such that $gA\cap eA\ne\emptyset$ and thus $g\in eAA^{-1}$. Then $G=EAA^{-1}$ and hence $\cov(AA^{-1})\le|E|\le\pack(A)$.
\end{proof}

\begin{proposition}\label{p7.1n} For any right-Solecki positive subset $A$ of a group $G$ we get
$$\cov(AA^{-1})\le\pack(A)\le\frac1{\sigma^R(A)}.$$
\end{proposition}

\begin{proof} By Proposition~\ref{p7.2n}, $\cov(AA^{-1})\le\pack(A)$. It remains to prove that
$\pack(A)>\frac1{\sigma^R(A)}$. Assume conversely that $\pack(A)>\frac1{\sigma^R(A)}$ and find a finite set $E\subset G$ of cardinality $|E|>\frac1{\sigma^R(A)}$ such that for any distinct points $x,y\in E$ the sets $xA$ and $yA$ are disjoint. Since $\frac1{|E|}<\sigma^R(A)\le\sup_{z\in G}|E^{-1}z\cap A|/|E^{-1}|$, there is a point $z\in G$ such that $|E^{-1}z\cap A|\ge 2$. Then we can choose two distinct points $x,y\in E$ such that $x^{-1}z,y^{-1}z\in A$ and hence $z\in xA\cap yA$, which contradicts the choice of the set $E$.
\end{proof}

Theorem~\ref{t-FC} and Propositions~\ref{p7.2n} and \ref{p7.1n} imply:

\begin{corollary}\label{c7.3n} For any Solecki positive set $A$ in an FC-group $G$ the difference set $AA^{-1}$ has covering number $\cov(AA^{-1})\le\pack(A)\le1/\sigma(A).$
\end{corollary}

\begin{remark} Corollary~\ref{c7.3n} cannot be generalized to amenable groups. A suitable counterexample can be constructed as follows. Take an infinite set $X$ and an infinite subset $Y\subset X$ with infinite complement $X\setminus Y$. Consider the group $FS_{X}$ of finitely supported bijections of $X$ and the subgroups $FS_{Y}=\{f\in FS_X:\supp(f)\subset Y\}$. Observe that the group $FS_X$ is locally finite and hence amenable, the subgroup $FS_{Y}$ has infinite packing index and infinite covering number but is Solecki one according to Example~\ref{permutation}.
\end{remark}

\begin{problem} Let $G$ be a non-trivial (amenable) group.
\begin{enumerate}
\item Is there a subset $A\subset G$ with $0<\sigma(A)<1$?
\item Is there a large subset $A\subset G$ with $\sigma(A)<1$?
\item Is there a finite partition $G=A_1\cup\dots\cup A_n$ of $G$ such that $\sigma(A_i)<1$ for all $i\le n$? What is the answer for $n=2$?
\end{enumerate}
\end{problem}

Corollary~\ref{c7.3n} implies that all these questions have affirmative answers for FC-groups $G$.

 Another question concerns a possible characterization of amenability.

\begin{problem}[Protasov] Is a group $G$ amenable if for each partition $G=A_1\cup \dots\cup A_n$ there is a cell $A_i$ of the partition satisfying one of the conditions: \textup{(a)}~$\sigma^R(A_i)\ge\frac1n$, \textup{(b)}~$\pack(A_i)\le n$, \textup{(c)}~$\cov(A_iA_i^{-1})\le n$, \textup{(d)}~$\sigma^R(A_i)>0$, \textup{(e)}~$\pack(A_i)<\w$?
\end{problem}

\section{The $\I$-difference sets of right-Solecki positive sets in groups}

In this section we generalize the upper bound $\cov(AA^{-1})\le1/\sigma^R(A)$ proved in Proposition~\ref{p7.1n} and give an upper bound on the covering number of the  $\I$-difference set
$$\Delta_\I(A)=\{x\in G:A\cap xA\notin\I\},$$where $\I$ a family of subsets of a group $G$ and $A$ is a subset of $G$. Usually we shall assume that $\I$ is a left-invariant ideal of subsets of $G$.

A non-empty family $\I$ of subsets of a set $X$ is an {\em ideal} if it is closed under unions and taking subsets. An ideal $\I$ of subsets of a group $G$ will be called {\em left-invariant} if for each set $A\in\I$ all its left shifts $xA$, $x\in G$, belong to $\I$.

Observe that the difference set $AA^{-1}$ of a set $A\subset G$ coincides with the $\I$-difference set $\Delta_\I(A)$ for the smallest ideal $\I=\{\emptyset\}$.

For a subset $A$ of a group $G$ and a left-invariant family $\I$ of subsets of $G$ the covering number $\cov(\Delta_\I(A))$ of the  $\I$-difference set is bounded from above by the {\em  $\I$-packing index}
$$
\begin{aligned}
\Ipack(A)&=\sup\big\{|E|:E\subset G,\,\;\forall x,y\in E\;(x\ne y\:\Rightarrow xA\cap yA\in \I) \big\}
\end{aligned}
$$
of the set $A$.
It is clear that $\pack(A)=\I_0\mbox{-}pack_L(A)$ for the smallest ideal $\I_0=\{\emptyset\}$.

\begin{proposition}\label{Ipack-cov} For any left-invariant family $\I$ of subsets of a group $G$ and any subset $A\notin\I$ of $G$ we get $\Ipack(A)\ge\cov(\Delta_\I(A))$.
\end{proposition}

\begin{proof} By Zorn's Lemma, there is a maximal set $E\subset G$ such that $xA\cap yA\in\I$ for any distinct points $x,y\in E$. By the maximality of $E$, for each $g\in G$ there is an element
$e\in E$ such that $eA\cap gA\notin\I$. Since $\I$ is left-invariant, this implies $A\cap e^{-1}gA\notin\I$ and hence $e^{-1}g\in \Delta_\I(A)$ according to the definition of $\Delta_\I(A)$.
Then $g\in e\cdot\Delta_\I(A)\subset E\cdot\Delta_\I(A)$, which implies  $G=E\cdot \Delta_\I(A)$ and $\cov(\Delta_\I(A))\le|E|\le\Ipack(A)$.
\end{proof}

Now we prove the ``idealized'' versions of Proposition~\ref{p7.1n}.

\begin{proposition}\label{Icov-sigma} Let $G$ be a group and $\mathcal I=\{B\subset G:\hat\sigma_R(B)=0\}$. Any subset $A\subset G$ with $\sigma_R(A)>0$ has $\cov(\Delta_\I(A))\le\Ipack(A)\le1/{\sigma_R(A)}$.
\end{proposition}

\begin{proof}  Proposition~\ref{Ipack-cov} implies that $\cov(\Delta_{\I}(A))\le \Ipack(A)$.

Assuming that $\Ipack(A)>1/\sigma_R(A)$, we can find a finite set $F\subset G$ of cardinality $|F|>1/\sigma_R(A)$ such that $xA\cap yA\in\I$ for all distinct points $x,y\in F$.
Then the set $E=\bigcup\{xA\cap yA:x,y\in E,\;x\ne y\}$ belongs to the ideal $\I$ and so does the set $F^{-1}E$. Now consider the set $A'=A\setminus F^{-1}E$ and observe that $$\sigma_R(A)-\sigma_R(A')\le\sigma_R(A'\cup F^{-1}E)-\sigma_R(A')\le\hat\sigma_R(F^{-1}E)=0.$$
So, $\sigma_R(A')=\sigma_R(A)$ and  $\pack(A')\le1/\sigma^R(A')\le1/\sigma_R(A')=1/\sigma_R(A)$ according to Proposition~\ref{p7.1n}. On the other hand, for any distinct points $x,y\in F$ the sets $xA'$ and $yA'$ are disjoint. Assuming conversely that $xA'\cap yA'$ contains some points $z$, we would conclude that $z\in xA'\cap yA'\subset xA\cap yA\subset E$. Then $z=xx^{-1}z\in xF^{-1}E$ which is not possible as $z\in xA'=x(A\setminus F^{-1}E)$. This contradiction shows that the indexed family $(xA')_{x\in F}$ is disjoint and hence $\pack(A')\ge|F|>1/\sigma_R(A)\ge\pack(A')$, which is a desired contradiction.
\end{proof}

By analogy we can prove:

\begin{proposition}\label{Icov-sigma2} Let $G$ be a group and $\mathcal I=\{B\subset G:\hat\sigma^R(B)=0\}$. Any subset $A\subset G$ with $\sigma^R(A)>0$ has $\cov(\Delta_\I(A))\le\Ipack(A)\le1/{\sigma^R(A)}$.
\end{proposition}

Next, we prove a quantitative  version of Theorem~\ref{t4.2}.

\begin{theorem}\label{semikourov} Let $G$ be a group and $\I=\{B\subset G:\hat\varsigma_R(B)=0\}$. For any subset $A\subset G$ with $\hat\sigma_R(A)>0$ there is a finite subset $E\subset G$ such that the set $\Delta_\I(A)^{\circlearrowright E}=\bigcup_{x\in E}x^{-1}\cdot \Delta_\I(A)\cdot x$ has covering number $$\cov(\Delta_\I(A)^{\circlearrowright E})\le \frac1{\varsigma_R(A)}\le\frac1{\hat \sigma_R(A)}.$$
\end{theorem}

\begin{proof} Assume conversely that $\cov(\Delta_\I(A)^{\circlearrowright E})>1/\varsigma_R(A)$ for any non-empty finite subset $E\subset G$. Then we can choose a positive $\e$ such that
$$\frac1{\varsigma_R(A)-\e}<\min\big\{\cov(\Delta_\I(A)^{\circlearrowright E}):E\in [G]^{<\w}\big\}.$$
 By the definition of the density $\varsigma_R$,
there are measures $\mu_1\in P(G)$ and $\mu_2\in P_\w(G)$ such that $$\inf_{x\in G}\mu_2*\delta_x*\mu_1(A)>\varsigma_R(A)-\e.$$ Write the finitely supported measure $\mu_2$ as a convex combination $\mu_2=\sum_{i\in n}\alpha_i\delta_{a_i}$ of Dirac measures and put $E=\{a_i\}_{i\in n}$.

Using Zorn's Lemma, choose a maximal subset $M\subset G$ such that for every $a\in E$  and any distinct points $x,y\in M$ we get $xa^{-1}A\cap ya^{-1}A\in \I$. Then for every $g\in G$, by the maximality of $M$, there is a point $x\in M$ such that $ga^{-1}A\cap xa^{-1}A\notin\I$ for some $a\in E$, which implies that $g\in xa^{-1}\Delta_\I(A)a\subset M\cdot\Delta_\I(A)^{\circlearrowright E}$ and hence $1/(\varsigma_R(A)-\e)<\cov(\Delta_\I(A)^{\circlearrowright E})\le|M|$.

Choose any finite subset $F\subset M$ of cardinality $|F|\ge\cov(\Delta_\I(A)^{\circlearrowright E})$ and consider the set
$$S=\bigcup\{xa^{-1}A\cap ya^{-1}A:\mbox{$a\in E$ and $x,y$ are distinct points of $F$}\},$$
which belongs to the ideal $\I=\{B\subset G:\hat\varsigma_R(B)=0\}$ by the choice of $M$.
Then the set $B=EF^{-1}S$  belongs to the ideal $\I$ too.

Repeating the argument from the proof of Proposition~\ref{Icov-sigma}, we can show that
 for every $a\in E$ the indexed family $(xa^{-1}(A\setminus B))_{x\in F}$ is disjoint, which implies that $\sum_{x\in F}\mu_1(xa^{-1}(A\setminus B))\le 1$ and hence
$$
\sum_{x\in F}\mu_2*\delta_{x^{-1}}*\mu_1(A\setminus B)=\sum_{x\in F}\sum_{i\in n}\alpha_i\delta_{a_i}*\delta_{x^{-1}}*\mu_1(A\setminus B)=\sum_{i\in n}\alpha_i\sum_{x\in F}\mu_1(xa^{-1}_i(A\setminus B))\le \sum_{i\in n}\alpha_i\cdot 1=1.$$
It follows from $B\in\I$ that $\hat\varsigma_R(B)=0$ and hence $\varsigma_R(A)\le\varsigma_R(A\setminus B)+\hat\varsigma_R(B)=\varsigma_R(A\setminus B)$.

Consequently,
$$
|F|\cdot(\varsigma_R(A)-\e)=|F|\cdot (\varsigma_R(A\setminus B)-\e)<|F|\cdot \inf_{x\in G}\mu_2*\delta_x*\mu_1(A\setminus B)\le\sum_{x\in F}\mu_2*\delta_{x^{-1}}*\mu_1(A\setminus B)\le 1$$
and we obtain a desired contradiction: $|F|\le 1/(\varsigma_R(A)-\e)<\cov(\Delta_\I(A)^{\circlearrowright E})\le|F|$.
\end{proof}

\begin{remark} Since $\hat\varsigma_R\le \sigma$ (according to Theorem~\ref{subadit-sol}), the ideal $\I=\{B\subset G:\hat\varsigma_R(B)=0\}$ appearing in Theorem~\ref{semikourov} is contains the ideal of Solecki null sets.
\end{remark}

Now we apply Theorem~\ref{Icov-sigma} to give a partial answer to the following problem of I.V.Protasov from the Kourovka Problem Notebook \cite[Problem 13.44]{Kourov} and its ``idealized'' version from \cite{ProtD}.

\begin{problem}[Protasov]\label{prob-Kourov} Is it true that for every finite cover $G=A_1\cup\dots\cup A_n$ of an (infinite) group $G$ there is an index $i\le n$ such that $\cov(A_iA_i^{-1})\le n$ \textup{(}and  $\cov(\Delta_\I(A_i))\le n$ for the ideal $\I$ of finite subsets of $G$\textup{)}?
\end{problem}

The answer to this problem is positive for covers of groups by subgroup cosets as follows from Lemma 1 of \cite{Neumann} or can be alternatively derived from Corollary~\ref{c6.5n}.

We prove that the answer to Problem~\ref{prob-Kourov} is affirmative if the group $G$ is Solecki amenable or the partition consists of inner invariant sets. Let us recall that a subset $A$ of a group $G$ is called {\em inner invariant} if $xAx^{-1}=A$ for all $x\in G$.
The following theorem is a joint result of T.Banakh, I.Protasov and S.Slobodianiuk (cf. \cite{BPS}).

\begin{theorem}[Banakh, Protasov, Slobodiadiuk]\label{BPS} Let $G=A_1\cup \dots\cup A_n$ be a finite partition of a group and let $\I=\{A\subset G:\hat\sigma_R(A)=0\}$. If the group $G$ is Solecki amenable or the cells $A_i$ of the partition are inner invariant, then for some index $i\le n$ the $\I$-difference set $\Delta_\I(A_i)$ has covering number $\cov(A_iA_i^{-1})\le\cov(\Delta_\I(A_i))\le n$.
\end{theorem}

\begin{proof} We claim that $\sigma_R(A_i)\ge 1/n$ for some $i\le n$. If the group $G$ is Solecki amenable, then this follows from the subadditivity of the right Solecki density $\sigma_R$. If each cell $A_i$ of the partition is inner invariant, then  $\sigma_R(A_i)=\sigma(A_i)$ for all $i\le n$ and the existence of an index $i\le n$ with $\sigma_R(A_i)=\sigma(A_i)\ge 1/n$ follows from the subadditivity of the Solecki submeasure $\sigma$. By Proposition~\ref{Icov-sigma}, $\cov(A_iA_i^{-1})\le \cov(\Delta_\I(A_i))\le1/\sigma^R(A_i)\le1/\sigma_R(A_i)\le n.$
\end{proof}

Theorem~\ref{BPS} can be also deduced from the following corollary of Theorem~\ref{semikourov}.

\begin{corollary} Let $G$ be a group and $\I=\{B\subset G:\hat\varsigma_R(B)=0\}$. For any finite partition $G=A_1\cup\dots\cup A_n$ of a group there is a cell $A_i$ of the partition and a finite subset $E\subset G$ such that the set $\Delta_\I(A)^{\circlearrowright E}=\bigcup_{x\in E}x^{-1}\cdot \Delta_\I(A_i)\cdot x$ has covering number  $\cov(\Delta_\I(A)^{\circlearrowright E}) \le n$.
\end{corollary}

In Theorem~12.7 of \cite{PB} it was proved that for every partition $G=A_1\cup\dots\cup A_n$ there is cell $A_i$ of the partition such that $\cov(A_iA_i^{-1})\le 2^{2^{n-1}-1}$. An ``idealized'' version of this result was proved in \cite{Erde}. In \cite{BRS} these results were improved to the following form giving a partial answer to Protasov's Problem~\ref{prob-Kourov}.

\begin{theorem}[Banakh, Ravsky, Slobodianiuk] For any finite partition $G=A_1\cup\dots\cup A_n$ of a group $G$ and any left-invariant ideal $\I$ on $G$ there is an index $i\le n$ such that $$\cov(\Delta_\I(A_i))\le \max_{1\le k\le n}\sum_{i=0}^{n-k}k^i.$$
\end{theorem}

Let $\I$ be an ideal on a group $G$. A subset $A\subset G$ is called {\em $\I$-thin} if for any distinct elements $x,y\in G$ the intersection $xA\cap yA$ belongs to the ideal $\I$.

\begin{proposition}\label{thin-null} Let $G$ be a group and $\I$ be a left-invariant ideal on $G$ such that $\I\subset\{A\subset G:\varsigma(A)=0\}$. Each $\I$-thin subset $A$ of an infinite group $G$ has density $\varsigma_R(A)=0$ and hence is right-Solecki null.
\end{proposition}

\begin{proof} To prove that $\varsigma_R(A)=0$, it suffices to show that $\varsigma_R(A)<\e$ for every $\e>0$. Given any measures $\mu_1\in P(G)$ and $\mu_2\in P_\w(G)$ we need to find a measure $\mu_3\in P_\w(G)$ such that $\mu_2*\mu_3*\mu_1(A)<\e$. This is trivial if $A=\emptyset$. So we assume that $A\ne\emptyset$, in which case the group $G$ is infinite. So, we can choose a finite subset $F\subset G$ of cardinality $|F|>2/\e$ and consider the uniformly distributed measure $\mu_u=\frac1{|F|}\sum_{x\in F}\delta_x$.
Write the finitely supported measure $\mu_2$ as a convex combination $\sum_{i\in n}\alpha_i\delta_{a_i}$ of Dirac measures. Let $S=\{a_i\}_{i\in n}$. Since the set $A$ is $\I$-thin the set
$$U=\bigcup\{x^{-1}a_i^{-1}A\cap y^{-1}a_i^{-1}A:i\in n,\;\;x,y\in F,\;\;x\ne y\}$$belong to the ideal $\I$ and so does  the set $SFU$. Then $\varsigma_R(SFU)=0$ and by definition of the density $\varsigma_R$, we can find a finitely supported probability measure $\mu=\sum_{j\in m}\beta_j\delta_{b_j}\in P_\w(G)$ such that
$(\mu_2*\mu_u)*\mu*\mu_1(SFU)<\e/2$. We claim that for the measure $\mu_3=\mu_u*\mu$ we get the desired inequality $\mu_2*\mu_3*\mu_1(A)<\e$. For this consider the set $A'=A\setminus SFU$ and observe that for every $i\in n$ the indexed family $(x^{-1}a_i^{-1}A')_{x\in F}$ is disjoint. Indeed, assuming conversely that for some distinct points $x,y\in F$ the intersection $x^{-1}a_i^{-1}A'\cap y^{-1}a_i^{-1}A'$ contains some point $g$, we would conclude that $g\in U$ and $a_ixg\in a_ixU\cap A'\subset SFU\cap A'=\emptyset$. Therefore, the family $(x^{-1}a_i^{-1}A')_{x\in F}$ is disjoint and thus the family
$(b_j^{-1}x^{-1}a_i^{-1}A')_{x\in F}$ is disjoint for every $j\in m$. This implies that
$$\sum_{x\in F}\mu_1(b_j^{-1}x^{-1}a^{-1}_iA')\le 1$$and then
$$
\begin{aligned}
\mu_2*\mu_3*\mu_1(A)&\le \mu_2*\mu_3*\mu_1(A')+\mu_2*\mu_3*\mu_1(SFU)=\mu_2*\mu_u*\mu*\mu_1(A')+\mu_2*\mu_u*\mu*\mu_1(SFU)<\\
&<\sum_{i\in n}\sum_{j\in m}\sum_{x\in F}\alpha_i\frac1{|F|}\beta_j\cdot\delta_{a_i}*\delta_x*\delta_{b_j}*\mu_1(A')+\frac\e2=\\
&=\sum_{i\in n}\sum_{j\in m}\alpha_i\beta_j\frac1{|F|}\sum_{x\in F}\mu_1(b_j^{-1}x^{-1}a_i^{-1}A')+\frac\e2\le\\
&\le\sum_{i\in n}\sum_{j\in m}\alpha_i\beta_j\frac1{|F|}+\frac\e2<\frac1{|F|}+\frac\e2=\frac\e2+\frac\e2=\e.
\end{aligned}
$$
\end{proof}

\section{The difference sets of Solecki positive sets in Polish groups}

Let us recall \cite{Ke} that a subset $A$ of a topological space $X$ is called {\em analytic} if $A$ is a continuous image of a Polish space. Proposition~\ref{p7.1n} and Theorem~\ref{t4.2}  have a nice topological corollary, which can be considered as a variation of the classical theorem of Steinhaus and Weil \cite[20.17]{HR}.

\begin{corollary}\label{c7.5n} If a subset $A$ of a Polish group $G$ is right-Solecki positive (or has density $\varsigma_R(A)>0$), then the set $AA^{-1}$ is not meager in $G$. If the set $A$ is analytic, then $AA^{-1}AA^{-1}$ is a neighborhood of the unit $1_G$ in $G$.
\end{corollary}

\begin{proof} By Proposition~\ref{p7.1n} or Theorem~\ref{t4.2}, there is  a finite subset $F\subset G$ such that $G=FAA^{-1}F$. By the Baire Theorem, the set $AA^{-1}$ is not meager in $G$. If the set $A$ is analytic, then so is the set $AA^{-1}$. By \cite[29.5]{Ke}, the set $B=AA^{-1}$ has the Baire Property in $G$ and by the Picard-Pettis Theorem \cite[9.9]{Ke}, $BB^{-1}=AA^{-1}AA^{-1}$ is a neighborhood of the unit in $G$.
\end{proof}

It is natural to ask if right-Solecki positive sets in Corollary~\ref{c7.5n} can be replaced by Solecki positive sets. The following example shows that this cannot be done.

\begin{example} There exists a Polish group which contains a closed nowhere dense Solecki one subgroup.
\end{example}

\begin{proof} Let $X$ be a countable infinite set and $Y\subsetneqq X$ be a proper infinite subset of $X$. Endow the countable group $FS_Y$ with the discrete topology. By Example~\ref{permutation}, the subgroup $FS_Y=\{f\in FS_X:\supp(f)\subset Y\}$ is Solecki one in $FS_X$. This fact can be used to prove that the countable power $FS_Y^\w$ of $FS_Y$ is Solecki one in $FS_X^\w$. Since $FS_Y\ne FS_X$, the subgroup $FS_Y^\w$ is closed and nowhere dense in $FS_X$.
\end{proof}

However we do not know the answer to the following problem.

\begin{problem} Let $A$ be an analytic Solecki positive set in a compact Polish group $G$. Is $AA^{-1}AA^{-1}$ a neighborhood of the unit in $G$?
\end{problem}

The answer to this problem is affirmative under the condition that $A$ is closed in $G$.

\begin{proposition} For any Solecki positive closed subset $A$ in a compact topological group $G$ the set $AA^{-1}$ is a neighborhood of the unit in $G$.
\end{proposition}

\begin{proof} By Lemma~\ref{l6.5}, the set $A$ has Haar measure $\lambda(A)=\sigma(A)>0$. Then $AA^{-1}$ is a neighborhood of the unit in $G$ according to a  classical result of Steinhaus and Weil (see \cite[20.17]{HR} or \cite[\S3]{Jarai}).
\end{proof}

It is clear that a meager subgroup $A$ of a Polish group $G$ has infinite index in $G$, which implies that $\varsigma_R(A)=0$ according to Proposition~\ref{zero-index}.

\begin{problem} Let $H$ be a meager (analytic) subgroup of a compact topological group $G$. Is $H$  Solecki null in $G$?
\end{problem}

\section{The $\e$-difference sets of right-Solecki positive sets in amenable groups}

In this section, given a subset $A$ of an amenable group $G$ and $\e>0$ we study the largeness properties of the $\e$-difference set
$$\Delta_\e(A)=\{x\in G:\sigma^R(A\cap xA)\ge\e\}.$$
Our aim is to generalize to arbitrary amenable groups a theorem of Veech \cite{Veech}, generalized later to countable amenable groups by Beiglb\"ock, Bergelson and Fish \cite{BBF}. They proved that for any subset $A$ of positive Banach density $d^*(A)$ in a countable amenable group $G$ there is $\e>0$ and a subset $N\subset G$ of upper Banach density $d^*(N)=0$ such that the set $N\cup\Delta_\e(A)$ is a neighborhood of the unit in the Bohr topology of $G$.

Let us recall that the {\em Bohr topology} on a group $G$ is the smallest topology on $G$ such that the canonical homomorphism $\eta:G\to bG$ into the Bohr compactification $bG$ of $G$ is continuous. Since continuous homomorphisms into orthogonal groups $O(n)$, $n\in\IN$, separate points of compact Hausdorff topological groups, the Bohr topology on $G$ can be equivalently defined as the smallest topology in which all homomorphisms from $G$ to the compact Hausdorff group $K=\prod_{n=1}^\infty O(n)$ are continuous. Subsets $U\subset G$ belonging to the Bohr topology will be called {\em Bohr open}.

\begin{theorem}\label{t:Bohr} If a subset $A$ of an amenable group $G$ is right-Solecki positive, then for some positive $\e$ the $\e$-difference set $\Delta_\e(A)$ contains the intersection $U\cap T$ for some Bohr open neighborhood $U\subset G$ of the unit $1_G$ and some subset $T\subset G$ with $\sigma^R(G\setminus T)=0$.
\end{theorem}

\begin{proof} For countable amenable groups this theorem follows from Corollary 5.3 \cite{BBF} and the equality $d^*=\sigma^R$ proved in Theorem~\ref{Bandens}. The general case will be derived by a suitable compactness argument. So, we assume that $G$ is an uncountable amenable group and $A$ is a right-Solecki positive subset in $G$.

Let $\HH$ be the family of all countable subgroups of the group $G$ partially ordered by the inclusion relation. A subset $\F\subset \HH$ will be called
\begin{itemize}
\item {\em closed} if for each increasing sequence of countable subgroups $\{H_n\}_{n\in\w}\subset\F$ the union $\bigcup_{n\in\w}H_n$ belongs to $\F$;
\item {\em dominating} if each countable subgroup $H\in\HH$ is contained in some subgroup $H'\in\F$;
\item {\em stationary} if $\F\cap\C\ne\emptyset$ for every closed dominating subset $\C\subset\HH$.
\end{itemize}
It is well-known (see \cite[4.3]{Jech}) that the intersection $\bigcap_{n\in\w}\C_n$ of any countable family of closed dominating sets $\C_n\subset \HH$, $n\in\w$, is closed and dominating in $\HH$.

For a subgroup $H\subset G$ let
$$\sigma^R_H(A)=\inf_{F\in[H]^{<\w}}\max_{y\in H}\frac{|F\cap Ay|}{|F|}$$be the right Solecki density of the set $A\cap H$ in the group $H$.

For every $\e>0$ let $\Delta_\e(A;H)=\{x\in H:\sigma^R_H(A\cap xA)\ge\e\}$ be the counterpart of the $\e$-difference set $\Delta_\e(A)$ in the subgroup $H$.

\begin{claim}\label{cl1:t:Bohr} The subfamily $$\mathcal A=\{H\in\HH:\sigma^R_H(A)\ge \sigma^R(A)\}$$ is closed and dominating in $\HH$.
\end{claim}

%\begin{claim} For every  $\e>0$ the set $\mathcal D_\e=\{H\in\HH:\Delta_\delta(A;H)\subset\Delta_\e(A)\}$ is closed and dominating in $\HH$.
%\end{claim}

\begin{proof}
To show that $\A$ is closed in $\HH$, we need to prove that the union $H=\bigcup_{n\in\w}H_n$
of any increasing sequence of subgroups $\{H_n\}_{n\in\w}\subset \A$ belongs to $\A$, which means that $\sigma^R_H(A)\ge\sigma^R(A)$. Assuming conversely that $\sigma^R_H(A)<\sigma^R(A)$, we can find a finite subset $F\subset H$ such that $\sup_{y\in H}|Fy\cap A|/|F|<\sigma^R(A)$. Find $n\in\w$ with $F\subset H_n\in\A$ and obtain a desired contradiction:
$$\sup_{y\in H}\frac{|Fy\cap A|}{|F|}<\sigma^R(A)\le\sigma^R_{H_n}(A)\le\sup_{y\in H_n}\frac{|Fy\cap A|}{|F|}\le \sup_{y\in H}\frac{|Fy\cap A|}{|F|}.$$

To show that $\A$ is dominating in $\HH$, fix any countable subgroup $H_0\subset G$. Taking into account that
$$\sigma^R(A)=\inf_{F\in[F]^{<\w}}\sup_{y\in G}\frac{|Fy\cap A|}{|F|}=\inf_{F\in[F]^{<\w}}\max_{y\in G}\frac{|Fy\cap A|}{|F|},$$for every finite set $F\subset G$ choose a point $y_F\in G$ such that
$|Fy_F\cap A|/|F|\ge\sigma^R(A)$. For every $n\in\w$ let $H_{n+1}$ be the countable subgroup of $G$ generated by the countable set $H_n\cup\{y_F:F\in[H_n]^{<\w}\}$. To see that the subgroup $H=\bigcup_{n\in\w}H_n$ belongs to the family $\A$, observe that
$$\begin{aligned}
\sigma^R_H(A)=\inf_{F\in[H]^{<\w}}\sup_{y\in H}\frac{|Fy\cap A|}{|F|}\ge\inf_{n\in\w}\inf_{F\in[H_n]^{<\w}}\sup_{y\in H_{n+1}}\frac{|Fy\cap A|}{|F|}\ge\inf_{n\in\w}\inf_{F\in[H_n]^{<\w}}\frac{|Fy_F\cap A|}{|F|}\ge\sigma^R(A).
\end{aligned}
$$\end{proof}

Let $K=\prod_{n=1}^\infty O(n)$ be the Tychonoff product of orthogonal groups and $\{U_n\}_{n\in\w}$ be a countable base of open neighborhoods at the unit $1_K$ of the group $K$ such that $U_{n+1}\subset U_n$ for all $n\in\w$. For a subgroup $H\in\HH$ by $\Hom(H,K)$ we denote the set of all homomorphisms from $H$ to $K$. Since homomorphisms into orthogonal groups separate points of compact Hausdorff topological groups, the Bohr topology on $H$ coincides with the smallest topology in which all homomorphisms $h\in \Hom(H,K)$ are continuous.

\begin{claim}\label{cl2:t:Bohr} For some number $n\in\IN$ the set
$$\A_n=\{H\in\A:\exists h\in\Hom(H,K)\;\;\sigma^R_H(h^{-1}(U_n)\setminus \Delta_{1/n}(A;H))=0\}$$is stationary in $\HH$.
\end{claim}

\begin{proof} Assuming that for every $n\in\IN$ the set $\A_n$ is not stationary in $\HH$, we can find a closed dominating subset $\C_n\subset \HH$ which is disjoint with $\A_n$. It is standard to show that the intersection $\C_\infty=\mathcal A\cap \bigcap_{n=1}^\infty \C_n$ is closed and dominating in $\HH$ and hence contains some element $H\in\C_\infty$. It follows from $H\in\C_\infty\subset \A$ that $\sigma^R_H(A)\ge\sigma^R(A)>0$.
By Theorem~\ref{Bandens}, the set $A_H=A\cap H$ has positive upper Banach density $d^*(A_H)=\sigma^R_H(A_H)$ in $H$. Then by (the proof of) Corollary~5.3 of \cite{BBF}, there exists $\e>0$ and a neighborhood $U\subset H$ of the unit $1_H$ in the Bohr topology of $H$ such that $d^*(U\setminus \Delta_\e(A;H))=0$. By Theorem~\ref{Bandens}, $\sigma^R_H(U\setminus \Delta_\e(A;H))=0$. For the Bohr neighborhood $U$ we can find a number $n>1/\e$ and a homomorphism $h\in\Hom(H,K)$ such that $h^{-1}(U_n)\subset U$. Then $H\in\A_n$ and hence $H\in\A_n\cap\C_\infty\subset \HH_n\cap\C_n=\emptyset$, which is a desired contradiction.
\end{proof}

Claim~\ref{cl2:t:Bohr} allows us to fix a number $n\in\w$ such that the family $\A_n$ is stationary in $\HH$.
By the definition of $\A_n$, for every subgroup $H\in\A_n$ there exists a homomorphism $h_H\in\Hom(H,K)$ such that the set $D_H=h_H^{-1}(U_n)\setminus\Delta_{1/n}(A;H)$ has right Solecki density $\sigma^R_H(D_H)=0$. Then for each $m\in\IN$ we can find a finite subset $F_{H,m}\subset H$ such that $\sup_{y\in H}{|F_{H,m}y\cap D_H|}/{|F_{H,m}|}<1/m$. Let $\mathcal S_0=\A_n$ and for every $m\in\IN$ let $f_m:\mathcal S_0\to [G]^{<\w}$ be the function assigning to each subgroup $H\in\mathcal S_0$ the finite subset $f_m(H)=F_{H,m}\subset H$. By Jech's generalization \cite{Jech72}, \cite[4.4]{Jech}  of Fodor's Lemma, the stationary set $\mathcal S_0$ contains a stationary subset $\mathcal S_1\subset \mathcal S_0$ such that the restriction $f_1|\mathcal S_1$ is a constant function. Proceeding by induction, we can construct a decreasing sequence $(\mathcal S_m)_{m\in\w}$ of stationary sets in $\HH$ such that for every $m\in\IN$ the restriction $f_m|\mathcal S_m$ is constant.

For every subgroup $H\in\mathcal S_0$ extend the homomorphism $h_H:H\to K$ to any function $\bar h_H:G\to K$. The function $\bar h_H$ is an element of the compact Hausdorff space $K^H$. For every $m\in\w$ and a finite subset $F\subset G$ consider the closure $\bar{\mathcal K}_{H,m}$ of the set $\mathcal K_{F,m}=\{\bar h_S:F\subset S\in\mathcal S_m\}$ in the compact Hausdorff space $K^G$. The stationarity of $\mathcal S_m$ guarantees that the set $\mathcal K_{F,m}$ is not empty. Observe that for any pairs $(F,m),(E,k)\in[G]^{<\w}\times\w$ the intersection $\K_{F,m}\cap\K_{E,k}$ contains the set $\K_{F\cup E,\max\{m,k\}}$. This implies that the family $\{\bar{\mathcal \K}_{F,m}:(F,m)\in[G]^{<\w}\times\w\}$ is centered and hence the intersection $\bigcap\{\bar{\mathcal K}_{F,m}:(F,m)\in[G]^{<\w}\times\w\}$ contains some function $h\in K^G$.

It is standard to check that the function $h:G\to K$ is a group homomorphism.
To finish the proof of the theorem, it remains to prove that $\sigma^R(h^{-1}(U_n)\setminus\Delta_{1/n}(A))=0$. Assume conversely that the set $D=h^{-1}(U_n)\setminus\Delta_{1/n}(A)$ has right Solecki density $\sigma^R(D)>0$.
Find $m\ge n$ such that $\frac1m<\sigma^R(D)$. By the choice of the stationary set $\mathcal S_m$, the function $f_m|\mathcal S_m$ is constant and hence $f_m(\mathcal S_m)=\{F\}$ for some finite set $F\subset G$. For the set $F$ choose a point $y\in G$ such that $|Fy\cap D|/|F|\ge\sigma^R(D)$.
For every point $x\in Fy\setminus \Delta_{1/n}(A)$ we get $\sigma^R(A\cap xA)<\frac1n$ and hence there exists a non-empty finite set $F_x\subset G$ such that $\sup_{z\in G}|F_xz\cap(A\cap xA)|/|F_x|<\frac1n$. Consider the finite set $E=Fy\cup\{F_x:x\in Fy\setminus\Delta_{1/n}(A)\}$.
It follows that $$O_h=\{f\in K^G:f(Fy\cap h^{-1}(U_n))\subset U_n\}$$ is an open neighborhood of the function $h$ in $K^X$. Since $h\in\bar\K_{E,m}$, there is a subgroup $H\in \mathcal S_m$ such that $E\subset H$ and $\bar h_H\in O_h$. By the choice of the set $F_{H,m}=f_m(H)=F$, $|Fy\cap D_H|/|F|<\frac1m$.

We claim that $Fy\cap D\subset Fy\cap D_H$, where $D_H=h_H^{-1}(U_n)\setminus \Delta_{1/n}(A;H)$. Take any point $x\in Fy\cap D$ and observe that $x\in Fy\cap D=Fy\cap h^{-1}(U_n)\setminus\Delta_{1/n}(A)\subset Fy\cap h^{-1}(U_n)\subset Fy\cap h_H^{-1}(U_n)$ as $h_H\in O_h$.
Since $x\notin\Delta_{1/n}(A)$, the set $F_x\subset E$ is contained in the subgroup $H$ which implies that $\sigma^R_H(A\cap xA)\le\sup_{z\in H}\frac{|F_xz\cap(A\cap xA)|}{|F_x|}<\frac1n$ and hence $x\in Fy\cap h_H^{-1}(U_n)\setminus\Delta_{1/n}(A;H)=Fy\cap D_H$. Finally, we obtain the desired contradiction as: $$\sigma^R(D)\le\frac{|Fy\cap D|}{|F|}\le \frac{|Fy\cap D_H|}{|F|}<\frac1m<\sigma^R(D).$$
\end{proof}

Theorem~\ref{t:Bohr} is related to the following classical problem from Combinatorial Number Theory and Harmonic Analysis (see \cite[Question 2]{Pestov} and references therein):

\begin{problem} Let $A$ be a large set in the group of integers $\IZ$. Is $AA^{-1}$ a Bohr open neighborhood of zero in $\IZ$?
\end{problem}

\begin{remark} In \cite{Protas} Protasov proved that each countable totally bounded topological group $G$ contains a dense thin subset $N$. By Proposition~\ref{thin-null} this set $N$ is right-Solecki null in $G$.
So, for a Bohr open subset $U$ of a group $G$ and a subset $T\subset G$ with $\sigma^R(G\setminus T)=0$ the intersection $U\cap T$ (from Theorem~\ref{t:Bohr}) can have empty interior in the Bohr topology on $G$.
\end{remark}

The following two corollaries of Theorem~\ref{t:Bohr} generalize the results of Bogoliuboff, F\o lner \cite{Folner}, Cotlar, Ricabarra \cite{CR}, Ellis, Keynes \cite{EK}.

\begin{corollary}\label{c1:Bohr} For any right-Solecki positive sets $A,B$ in an amenable group $G$ the set $B^{-1}AA^{-1}$ has non-empty interior in the Bohr topology on $G$.
\end{corollary}

\begin{proof} By Theorem~\ref{t:Bohr}, there are a Bohr open neighborhood $U\subset G$ of the unit and a right-Solecki null set $N\subset G$ such that $U\setminus N\subset AA^{-1}$. Since the multiplication and the inversion are continuous in the Bohr topology on $G$, there is a Bohr open neighborhood $V\subset G^{\#}$ of the unit such that $VV^{-1}\subset U$. By the total boundedness of the Bohr topology, there is a finite subset $F\subset G$ such that $G=VF$. Since $B=\bigcup_{x\in F}Vx\cap B$, the subadditivity of the right Solecki density $\sigma^R$ (which follows from Corollary~\ref{solamenable}) yields a point $x\in F$ such that $B_x=Vx\cap B$ is right-Solecki positive. We claim that $x^{-1}V\subset B_x^{-1}(U\setminus N)$. Given any point $v\in V$, consider the set $B_xx^{-1}v\subset Vxx^{-1}v\subset VV\subset U$. Being right-Solecki positive, the set $B_xx^{-1}v$ is not contained in the right-Solecki null set $N$ and hence meets the complement $U\setminus N$. Then $x^{-1}v\in B_x^{-1}(U\setminus N)\subset B^{-1}AA^{-1}$ and hence the set $B^{-1}AA^{-1}$ contains the non-empty Bohr open set $x^{-1}V$.
\end{proof}

\begin{corollary}\label{c2:Bohr} For any right-Solecki positive sets $A,B$ in an amenable group $G$ the set $AA^{-1}BB^{-1}$ is a neighborhood of the unit $1_G$ in the Bohr topology of $G$.
\end{corollary}

\begin{proof} By Theorem~\ref{t:Bohr}, there are a right-Solecki null set $N_A,N_B\subset G$ and a Bohr open neighborhood $U\subset G$ of the unit such that $U\setminus N_A\subset AA^{-1}$ and $U\setminus N_B\subset BB^{-1}$. Using the continuity of the multiplication and inversion with respect to the Bohr topology on $G$, find a Bohr open neighborhood $V\subset G$ of the unit $1_G$ such that $VV^{-1}\subset U$. We claim that $V\subset AA^{-1}BB^{-1}$. The subadditivity of the right Solecki density on amenable groups and the total boundedness of the topological group $G$ implies also that the neighborhood $V$ is right-Solecki positive. The subadditivity of the right Solecki density $\sigma^R$ implies that $\sigma^R(V\setminus N_B)=\sigma^R(V)>0$. Then for every $v\in V$ the set $v(V\setminus N_B)\subset U$, being right-Solecki positive, meets the set $U\setminus N_A$, which implies $v\in (U\setminus N_A)(V\setminus N_B)^{-1}\subset AA^{-1}BB^{-1}$.
\end{proof}

\begin{problem}\label{prob12.8} Is Theorem~\ref{t:Bohr} true for non-amenable groups?
\end{problem}

The following weaker version of Problem~\ref{prob12.8} also seems to be open:

\begin{problem}\label{prob:nE} Let $A$ be an inner invariant Solecki positive subset of a group $G$. Is $\sigma(U\setminus AA^{-1})=0$ for some Bohr open neighborhood $U$ of the unit $1_G$? Is $AA^{-1}AA^{-1}$ a neighborhood of the unit in the Bohr topology on $G$?
\end{problem}

The following proposition can be considered as a partial answer to this problem.

\begin{proposition} If a subset $A$ of a group $G$ has right-Solecki density $\sigma^R(A)\ge\frac1n$ for some $n\in\IN$, then the set $U=(AA^{-1})^{4^{n-1}}$ is a subgroup of finite index $\le n$ in $G$ and hence $U$ is a Bohr open neighborhood of the unit $1_G$.
\end{proposition}

\begin{proof} By Proposition~\ref{p7.2n}, $\cov(AA^{-1})\le1/\sigma^R(A)\le n$. By Lemma 12.3 of \cite{PB}, $H=(AA^{-1})^{4^{n-1}}$ is a subgroup of finite index $\le n$ in $G$.
By \cite[1.6.9]{Rob}, the subgroup $H$ contains a normal subgroup of finite index in $G$ and hence is a Bohr neighborhood of the unit.
\end{proof}

For the (non-amenable) group $G=S_X$ of all permutations of an infinite set, we can apply results of Bergman \cite{Bergman} and obtain another partial answer to Problem~\ref{prob:nE}.

\begin{proposition}\label{p:SX} If $A$ is an inner invariant Solecki positive set in
the group $G=S_X$ of all permutations of an infinite set $X$, then $(AA^{-1})^{18}=G$.
\end{proposition}

\begin{proof}  Following \cite{Bergman}, we say that a subset $U\subset S_X$ has a {\em full moiety}
if there is an infinite set $Y\subset X$ with infinite  complement $X\setminus Y$ (called a {\em full moiety for $U$}) such that for each
permutation $f\in S_Y$ extends to a permutation $\bar f\in U$. In this case the set $U^{-1}U$ also has the full moiety $Y$.

Since $A$ is inner invariant, $\sigma^R(A)=\sigma(A)>0$ and hence $\hat\sigma_R(A)=\sigma(A)$. By Proposition~\ref{p7.1n}, $\cov(AA^{-1})<1/\sigma^R(A)<\infty$ and hence there is a finite subset $F\subset G$ such that $G=FAA^{-1}$. By Lemma 4 of \cite{Bergman}, for some $g\in F$ the set $xAA^{-1}$ has a full moiety and then so does the set $U=(xAA^{-1})^{-1}(xAA^{-1})=(AA^{-1})^2$. By Lemma 3 of \cite{Bergman}, there is an element $g\in G$ of order 2 such that $G=((Ug)^7U^2g)\cup ((gU)^7gU^2)$. Since the set $U=(AA^{-1})^2$ is inner invariant and the element $g$ has order 2, we finally conclude that
$G=(U^9g^8)\cup(g^8U^9)=U^9=(AA^{-1})^{18}.$
\end{proof}

It is interesting to compare Proposition~\ref{p:SX} with:

\begin{proposition}\label{p:ASX} If $A$ is a right-Solecki positive set in
the group $G=A_X$ of all even finitely supported permutations of an infinite set $X$, then $AA^{-1}A=G$.
\end{proposition}

\begin{proof} By Corollary~\ref{c1:Bohr}, the set $A^{-1}AA^{-1}$ has non-empty interior in the Bohr topology on $G$. Since the Bohr compactification of the group $G=A_X$ is trivial, the unique non-empty Bohr open subset of $G$ is $G$. Consequently, $G=A^{-1}AA^{-1}$ and $G=G^{-1}=AA^{-1}A$.
\end{proof}

Comparing Propositions~\ref{p:SX} and \ref{p:ASX}, it is natural to ask:

\begin{problem} Is $G=AA^{-1}A$ for each (inner invariant) right-Solecki positive set $A$ in the group $G=S_X$ of permutations of an infinite set?
\end{problem}

\section{The sumsets of right-Solecki positive sets in amenable groups}

In \cite{Jin} Jin proved that for any subsets $A,B\subset\IZ$ of positive upper Banach density there is a finite set $F\subset \IZ$ such that the sumset $F+A+B=\{f+a+b:f\in F,\; a\in A,\; b\in B\}$ is thick (equivalently, is right-Solecki one). The initial proof of Jin's theorem used arguments of non-standard analysis.
In \cite{Jin2} Jin found a ``standard'' proof of this theorem and in \cite{BBF} Jin's theorem was generalized to all countable amenable groups. In \cite{NL} Di Nasso and Lupini using arguments of non-standard analysis generalized Jin's theorem to all amenable groups.

\begin{theorem}[Jin-Beiglb\"ock-Bergelson-Fish-Di Nasso-Lupini]\label{Jin} For any subsets $A,B$ of positive upper Banach density $d^*(A)=\sigma^R(A)$, $d^*(B)=\sigma^R(B)$ in an amenable group $G$ there is a finite set $F\subset G$ such that the sumset $FAB$ is thick and hence has right-Solecki density $\sigma^R(FAB)=1$.
\end{theorem}

In this section we shall present an elementary proof of this results.
Our proof of Theorem~\ref{Jin} (like that from \cite{BBF}) is based on the following ergodicity property of the right Solecki density $\sigma_R$ in arbitrary (not necessarily amenable) groups.

\begin{theorem}\label{ergodic} For any subset $A$ of a group $G$ we get
$$\sup_{F\in[G]^{<\w}}\sigma_R(FA)\in\{0,1\}.$$
\end{theorem}

\begin{proof} To see that $\sup_{F\in[G]^{<\w}}\sigma_R(FA)\in\{0,1\}$, it suffices to show that for any right-Solecki positive subset $A\subset G$ and every $\e>0$ there is a finite set $F\subset G$ such that $\sigma_R(FA)>1-\e$. Find a positive number $\delta$ such
$\frac{\sigma_R(A)-\delta}{\sigma_R(A)+\delta}>1-\e$.
By Theorem~\ref{t3.3nn}, $\sigma_R(A)=I(\{xA\}_{x\in G})$. Then the definition of the intersection number yields points $x_1,\dots,x_n\in G$ such that $$\sup_{y\in G}\frac1n\sum_{i=1}^n\chi_{x_iA}(y)<I(\{xA\}_{x\in G}+\delta=\sigma_R(A)+\delta$$and hence
\begin{equation}\label{eqq}
\frac1n\sum_{i=1}^n\chi_{x_iA}\le (\sigma_R(A)+\delta)\cdot\chi_{FA}
\end{equation}
where $F=\{x_1,\dots,x_n\}$. Taking into account the equality
 $\sigma_R(A)=\sup_{\mu\in P(G)}\inf_{x\in G}\mu(xA)$ established in Theorem~\ref{t3.3nn}, find a measure $\mu$ on $G$ such that $\inf_{x\in G}\mu(xA)>\sigma_R(A)-\delta$. Integrating the inequality (\ref{eqq}) by the measure $\mu$ we get
$$(\sigma_R(A)+\delta)\cdot \mu(FA)\ge \frac1n\sum_{i=1}^n\mu(x_iA)>\sigma_R(A)-\delta,$$which implies
the desired lower bound $$\mu(FA)>\frac{\sigma_R(A)-\delta}{\sigma_R(A)+\delta}>1-\e.$$
\end{proof}

We shall also need the following version of Lemma 3.1 \cite{BBF}.

\begin{lemma}\label{ergo-sum} Let $A,B$ be two subsets of an amenable group $G$. If $\sigma^R(A)+\sigma^R(B)>1$, then $\sigma^R(AB)=1$.
\end{lemma}

\begin{proof} Choose a positive real number $\e>0$ such that $\sigma^R(A)+\sigma^R(B)>1+\e$. The equality $\alpha^R(AB)=1$ will follow as soon as we check that for every finite subset $F\subset G$ there is a point $z\in G$ such that $Fz\subset AB$. We lose no generality assuming that $F$ contains the unit of the group $G$.

The amenability of $G$ yields a finite subset $E\subset G$ such that $|F^{-1}E\setminus E|<\e|E|$.
Since $\sigma^R(A)\le\max_{y\in G}\frac{|Ey\cap A|}{|E|}$, there is a point $y\in G$ such that $\frac{|Ey\cap A|}{|E|}\ge\sigma^R(A)$. Let $K=Ey$ and observe that $|F^{-1}K\setminus K|<\e|K|$ and $|K\cap A|\ge\sigma^R(A)|K|$. Then for every $x\in F$ we obtain that
$$
\begin{aligned}
\sigma^R(A)\cdot|K|&\le |K\cap A|\le|(xK\cup (K\setminus xK))\cap A|\le |xK\cap A|+|K\setminus xK|=\\
&=|K\cap x^{-1}A|+|x^{-1}K\setminus K|\le |K\cap x^{-1}A|+|F^{-1}K\setminus K|<|K\cap x^{-1}A|+\e|K|,
\end{aligned}
$$
and hence  $|K\cap x^{-1}A|>(\sigma^R(A)-\e)\cdot|K|$.

Since $\sigma^R(B)\le\max_{z\in G}\frac{|K^{-1}\cap Bz^{-1}|}{|K^{-1}|}$, there is a point $z\in G$ such that $\frac{|K^{-1}\cap Bz^{-1}|}{|K|}\ge \sigma^R(B)$.
Observe that for every point $x\in F$
$$|K\cap x^{-1}A|+|K\cap zB^{-1}|=|K\cap x^{-1}A|+|K^{-1}\cap Bz^{-1}|>(\sigma^R(A)-\e)\cdot|K|+\sigma^R(B)\cdot|K|>|K|,$$which implies that the set $K\cap x^{-1}A$ and $K\cap zB^{-1}$ have a common point and hence $xz\in AB$ and $Fz\subset AB$.
\end{proof}
\smallskip

Now we are able to present
\smallskip

\noindent{\em Proof of Theorem~\ref{Jin}}: Let $A,B$ be two sets of positive upper Banach density $d^*(A),d^*(B)$ in an amenable group $G$. By Theorem~\ref{Bandens} these sets have positive right Solecki densities $\sigma_R(A)=\sigma^R(A)=d^*(A)>0$ and $\sigma_R(B)=\sigma^R(B)=d^*(B)>0$. By Ergodic Theorem~\ref{ergodic}, there is a finite subset $F\subset G$ such that $\sigma_R(FA)>1-\sigma^R(B)$. By Theorem~\ref{t-FC}, $\sigma^R(FA)=\sigma_R(FA)>1-\sigma^R(B)$ and hence $\sigma^R(FA)+\sigma^R(B)>1$. Then $\sigma^R(FAB)=1$ by Lemma~\ref{ergo-sum} and hence $FAB$ is  thick by Proposition~\ref{p:rizne}(2).\qed
\medskip

In fact, using methods of non-standard analysis, Di Nasso and Lupini \cite{NL} proved the following quantitative version of Theorem~\ref{Jin}.

\begin{theorem}[Di Nasso, Lupini]\label{t:Jin-NL} For any right-Solecki positive sets $A,B$ in an amenable group $G$ there is a finite set $F\subset G$ of cardinality $|F|\le1/(\sigma^R(A)\cdot\sigma^R(B))$ such that the set $FAB$ is thick.
\end{theorem}

We know no standard proof of this result and also do not know if this theorem is valid for non-amenable groups. In \cite{BBF} Beiglb\"ock, Bergelson and Fisher obtained a striking generalization of Jin's theorem proving that for any subsets $A,B$ of positive upper Banach density in a countable amenable group there is a non-empty Bohr open set $U\subset G$ which is finitely embeddable in $AB$.

We shall say that a subset $A$ of a group $G$ is {\em finitely embeddable} in a subset $B\subset G$ if for every finite set $F\subset A$ there is a point $x\in G$ such that $Fx\subset B$. Observe that a subset $A\subset G$ is thick if and only if $G$ is finitely embeddable in $A$. The following simple proposition can be easily derived from the definition.

\begin{proposition}\label{repres} If a subset $A$ of a group $G$ is finitely embeddable in a subset $B\subset G$, then $\sigma^R(A)\le\sigma^B(B)$ and $AA^{-1}\subset BB^{-1}$.
\end{proposition}

A subset $A$ of a group $G$ is called {\em piecewise Bohr} if $A$ contains the intersection $U\cap T$ of a non-empty Bohr open subset $U\subset G$ and a  thick set $T\subset G$.

\begin{proposition}\label{piece-Bohr} A subset $A$ of a group $G$ is piecewise Bohr in $G$ if and only if some non-empty Bohr open set $U\subset G$ is finitely embeddable in $A$.
\end{proposition}

\begin{proof} To prove the ``only if'' part, assume that $A$ is piecewise Bohr in $G$. Find a non-empty Bohr open set $V\subset G$ and a  thick set $T\subset G$ such that $V\cap T\subset A$. Fix a point $x\in V$ and choose a Bohr open neighborhood $W$ of the unit $1_G$ such that $WxW\subset V$. Find a finite subset $Z\subset G$ such that $G=ZW$. Since $T$ is thick, there is a function $t:[G]^{<\w}\to G$ such that $F\cdot t(F)\subset T$ for all $F\in[G]^{<\w}$.
In the following claim,  $[G]^{<\w}$ considered as a partially ordered set endowed with the inclusion relation.

\begin{claim}\label{cl:t:sumset2} For some point $z\in Z$ the family $\F_z=\{F\in [G]^{<\w}:t(F)\in zW\}$ is dominating in $[G]^{<\w}$.
\end{claim}

\begin{proof}  Assuming the opposite, for every  $z\in Z$ find a finite subset $F_z\in[G]^{<\w}$ which is contained in no set $F\in\F_z$. Now consider the finite set $F=\bigcup_{z\in Z}F_z$. Since $t(F)\in G=ZW$, there is a point $z\in Z$ such that $t(F)\in zW$ and hence $F_z\subset F\in\F_z$, which contradicts the choice of the set $F_z$.
\end{proof}

Using Claim~\ref{cl:t:sumset2}, we can fix a point $z\in Z$ such that the family $\F_z$ is dominating in $[G]^{<\w}$. We claim that the Bohr open set $U=Wxz^{-1}$ is finitely embeddable in $A$. Given any finite subset $E\subset U$, find a set $F\in\F_z$ containing $E$.
Then $E\cdot t(F)\subset F\cdot t(F)\subset T$. On the other hand, $E\cdot t(F)\subset (Wxz^{-1})zW=WxW\subset V$. So, $E\cdot t(F)\subset T\cap V\subset A$, which means that $U$ is finitely embeddable in $A$. To completes the proof of the ``only if'' part of the proposition.
\smallskip

To prove the ``if'' part, assume that some non-empty Bohr open set $U\subset G$ is finitely embeddable in $A$. Replacing $U$ by a suitable right shift of $U$, we can assume that $U$ is a Bohr neighborhood of the unit $1_H$. Since $U$ is finitely embeddable in $A$, for every finite set $F\subset G$ there is a point $y_F\in G$ such that $(F\cap U)y_F\subset A$. Since the multiplication and the inversion are continuous with respect to the Bohr topology on $G$, there is an open neighborhood $W\subset G$ such that $WW^{-1}\subset U$. By the total boundedness of the Bohr topology, there exists a finite subset $Z\subset G$ such that $G=WZ$. Repeating the argument from the proof of  Claim~\ref{cl:t:sumset2}, we can fix a point $z\in Z$ such that the family $\F_z=\{F\in[G]^{<\w}:y_F\in Wz\}$ is dominating in $[G]^{<\w}$. Then for every $F\in\F_z$ we get $y_F\in Wz$ and hence $$zy_F^{-1}\in W^{-1}.$$

Since $\F_z$ is dominating in $[G]^{<\w}$, the set $T=\bigcup_{F\in\F_z}Fy_F$ is thick in $G$. We claim that for the non-empty Bohr open set $V=Wz\subset G$ the intersection $T\cap V$ lies in the set $A$. Given any point $x\in T\cap V$, find a finite set $F\in\F_z$ such that $x\in Fy_F$. Then $$x\in Fy_F\cap V=Fy_F\cap Wz=(F\cap Wzy_F^{-1})y_F\subset (F\cap WW^{-1})y_F\subset (F\cap U)y_F\subset A.$$So, $T\cap V\subset A$, which means that the set $A$ is piecewise Bohr in $G$.
\end{proof}

The following theorem generalizes to arbitrary amenable group the result of Beiglb\"ock, Bergelson and Fisher \cite{BBF} mentioned above.

\begin{theorem}\label{sumset} For any right-Solecki positive set $A,B$ in an amenable group $G$ the sumset $AB$ is piecewise Bohr. Consequently, some Bohr open neighborhood $U\subset G$ of the unit $1_G$ is finitely embeddable in the sumset $AB$.
\end{theorem}

\begin{proof} For countable amenable groups the first part of this theorem was proved in Theorem 3  \cite{BBF} while the second part follows from the first part and Proposition~\ref{piece-Bohr}. So, assume that $G$ is an uncountable group and $A,B\subset G$ be two sets of positive upper Banach density. By Theorem~\ref{Bandens}, $\sigma^R(A)=d^*(A)>0$ and $\sigma^R(B)=d^*(B)>0$.

In the subsequent proof we shall use some notations and results from the proof of Theorem~\ref{t:Bohr}.

In particular, by $K=\prod_{n=1}^\infty O(n)$ we denote the Tychonoff product of orthogonal groups,
 by $(U_n)_{n\in\w}$ a neighborhood base at $1_K$ consisting of open neighborhoods subset in $K$ such that $U_{n+1}\subset U_n$ for all $n\in\w$.
 By $\HH$ we denote the family of all countable subgroups partially ordered by the inclusion relation.

By analogy with Claim~\ref{cl1:t:Bohr} we can prove that the sets
$$\A=\{H\in\HH:\sigma^R_H(A)\ge\sigma^R(A)\}\mbox{ \ and \ }\mathcal B=\{H\in\HH:\sigma^R_H(B)\ge\sigma^R(B)\}$$ are closed and dominating in $\HH$.
For every subgroup $H\in\A\cap\mathcal B$ the sets $A_H=A\cap H$ and $B_H=B\cap H$ have positive right Solecki density in $H$. Consequently, by the ``countable'' version of Theorem~\ref{sumset}, some Bohr open neighborhood $U_H\subset H$ of $1_H$ is finitely embeddable in the sumset $A_H\cdot B_H$. Since the Bohr topology on $H$ is generated by preimages of open sets under homomorphisms from $H$ to the compact Hausdorff group $K$, we can find a number $n(H)\in\w$ for which there is a homomorphism $h_H:H\to K$ such that $U_H\supset h_H^{-1}(U_{n(H)})$. It is standard to check that for some $n\in\w$ the set $$\C=\{H\in\A\cap\mathcal B:n(H)=n\}$$is stationary in $\HH$.

Then for every subgroup $H\in\C$ we can choose a homomorphism $h_H:H\to K$ such that $h_H^{-1}(U_n)\subset U_H$. Let $\bar h_H:G\to K$ be any extension of the function $h_H$. By the compactness of the space $K^G$, the net $(\bar h_H)_{H\in\C}$ has an accumulation point $h\in K^G$. This is a function $h:G\to K$ such that for each neighborhood $O_h\subset K_G$ and each countable subgroup $H_0\in\HH$ there is a subgroup $H\in\C$ such that $H_0\subset H$ and $\bar h_H\in O_h$. It is standard to check that $h:G\to K$ is a group homomorphism.

To finish the proof it remains to check that the Bohr open neighborhood $U=h^{-1}(U_n)\subset G$ of the
unit $1_G$ is finitely embeddable in the sumset $AB$. Fix any finite subset $F\subset h^{-1}(U_n)$ and consider the open neighborhood $O_h=\{f\in K^G: f(F)\subset U_n\}$ of the function $h$ in the  compact Hausdorff space $K^G$. Since $h$ is an accumulation point of the net $(\bar h_H)_{H\in\C}$, there is a countable subgroup $H\in\C$ such that $F\subset H$ and $\bar h_H\in O_h$. Then $F\subset h_H^{-1}(U_n)\subset U_H$ and by the finite embeddability of the Bohr open set $U_H$ in $A_HB_H$ there is a point $y\in H$ such that $Fy\subset A_HB_H\subset AB$, which means that $U$ is finitely embeddable in the sumset $AB$. By Proposition~\ref{piece-Bohr}, the set $AB$ is piecewise Bohr.
\end{proof}

Theorem~\ref{sumset} and Proposition~\ref{repres} imply:

\begin{corollary}\label{c:sumset} For any right-Solecki positive sets $A,B$ in an amenable group $G$ the set $ABB^{-1}A^{-1}$ is a neighborhood of the unit $1_G$ in the Bohr topology of $G$.
\end{corollary}

\begin{problem} Is Theorem~\ref{sumset} true for any (not necessarily countable) amenable group $G$?
\end{problem}

A weaker form of this problem also seems to be open:

\begin{problem} Let $A,B$ be inner invariant Solecki positive sets in a group $G$. Is the set $AB$ piecewise Bohr? Is $ABB^{-1}A^{-1}$ a neighborhood of the unit in the Bohr topology on $G$?
\end{problem}

\section{Characterizing amenable groups with trivial Bohr compactification}

In this section we shall apply Theorems~\ref{t:Bohr} and \ref{sumset} to characterize amenable groups with trivial Bohr compactification. Observe that a group $G$ has trivial Bohr compactification if and only if any homomorphism $h:G\to K$ to a compact Hausdorff (or metrizable) topological group $K$ is constant. A simple example of an amenable group with trivial Bohr compactification is the group $A_X$ of all even finitely supported permutations of any infinite set $X$.

\begin{theorem}\label{trivBohr1} Let $G$ be a group.
\begin{enumerate}
\item If $G$ is amenable and has trivial Bohr compactification, then for any right-Solecki positive sets $A,B\subset G$ we get
    $$ABB^{-1}A^{-1}=B^{-1}AA^{-1}=AA^{-1}A=G,\;\sigma^R(G\setminus AA^{-1})=0,\;\;\sigma^R(AB)=1.$$
\item If the Bohr compactification of $G$ is not trivial, then $G$ contains an inner invariant Bohr open neighborhood $V=V^{-1}$ of the unit such that $$\sigma_R(V)>0,\;\;\sigma(VV^{-1}VV^{-1})\le \frac12,\;\; \sigma_R(G\setminus VV^{-1})\ge \frac12.$$
\end{enumerate}
\end{theorem}

\begin{proof} 1. The first statement follows immediately from Theorem~\ref{t:Bohr}, Corollary~\ref{c1:Bohr}, Theorem~\ref{sumset} and Corollary~\ref{c:sumset}.
\smallskip

2. Assume that the group $G$ has non-trivial Bohr compactification $bG$. The compact Hausdorff group $bG$, being non-trivial, contains an open neighborhood $U\subset bG$ of the unit of Haar measure $\lambda(U)\le\frac12$. By the continuity of the group operations on $bG$, we can choose an inner invariant closed neighborhood $W\subset bG$ of the unit such that $W=W^{-1}$ and $WW^{-1}WW^{-1}\subset U$. We claim that the preimage $V=\eta^{-1}(W)$ of $W$ under the canonical homomorphism $\eta:G\to bG$ has the required properties. It is clear that  $V$ is an inner invariant Bohr open neighborhood of the unit. The subadditivity of the Solecki submeasure $\sigma$ implies that $\sigma(V)>0$. Since $V$ is inner invariant, $\sigma_R(V)=\sigma(V)>0$. By Theorem~\ref{bohr}, $$\sigma(VV^{-1})\le\sigma(VV^{-1}VV^{-1})\le\bar\lambda(VV^{-1}VV^{-1})\le\lambda(WW^{-1}WW^{-1})\le\lambda(U)\le \frac12.$$
To see that $\sigma_R(G\setminus VV^{-1})\ge \frac12$, observe that by the inner invariance of the set $G\setminus VV^{-1}$ we get $\sigma_R(G\setminus VV^{-1})=\sigma(G\setminus VV^{-1})$ and by the subadditivity of the Solecki submeasure $\sigma$,
$\sigma(G\setminus VV^{-1})\ge 1-\sigma(VV^{-1})\ge\frac12.$
\end{proof}

Theorem~\ref{trivBohr1} and the subadditivity of the right Solecki density $\sigma^R$ on amenable groups imply the following Ramsey characterization of amenable groups with trivial Borh compactification.

\begin{corollary}\label{trivBohr} An amenable group $G$ has trivial Bohr compactification if and only if for each finite partition $G=A_1\cup\dots\cup A_n$ there is an index $i\le n$ such that $A_iA_i^{-1}A_i=G$.
\end{corollary}

Another characterization of amenable groups with trivial Bohr compactification was obtained by Bergelson and Furstenberg in \cite{BF}. A subset $B$ of a group $G$ is called an {\em IP-set} if there is a sequence $(x_i)_{i\in\w}$ of elements of $G$ such that for any finite number sequence $i_1<i_2<\dots<i_n$ the product $x_{i_1}\cdot x_{i_2}\cdots x_{i_n}$ belongs to $B$.

\begin{theorem}[Bergelson, Furstenberg] An amenable group $G$ has trivial Bohr compactification if and only if each right-Solecki positive subset  $A$ of $G$ contains an IP-set.
\end{theorem}

It is interesting to compare Corollary~\ref{trivBohr} with the characterization of odd groups proved in Theorem 3.2 of \cite{BGN}. A group $G$ is called {\em odd} if each element $x\in G$ has odd finite order.

\begin{theorem}[Banakh-Gavrylkiv-Nykyforchyn] A group $G$ is odd if and only if
 for any partition $G=A\cup B$ into two sets either $AA^{-1}=G$ or $BB^{-1}=G$.
\end{theorem}

\section{Concluding Remarks and an Open Problem}

The Solecki densities $\sigma_L$, $\sigma_R$ and the Solecki submeasures are initial representatives of the hierarchy of extremal densities defined on each group $G$ as follows.

First we remark that the densities $\sigma_R$, $\sigma_R$ and $\sigma$ can be equivalently defined by the formulas:
$$
\begin{aligned}
&\sigma_R(A)=\inf_{\mu_1\in P_\w(G)}\sup_{\mu_2\in P_\w(G)}\mu_1*\mu_2(A),\\
&\sigma_L(A)=\inf_{\mu_1\in P_\w(G)}\sup_{\mu_2\in P_\w(G)}\mu_2*\mu_1(A),\\
&\sigma(A)=\inf_{\mu_1\in P_\w(G)}\sup_{\mu_2\in P_\w(G)}\sup_{\mu_3\in P_\w(G)}\mu_2*\mu_1*\mu_3(A),
\end{aligned}
$$which can be shortly written as $\sigma_R=\inf\sup_{12}$, $\sigma_L=\inf\sup_{21}$, $\sigma=\inf\sup\sup_{213}$ or even shorter as $\sigma_R=\mathsf{is}_{12}$, $\sigma_L=\mathsf{is}_{21}$, $\sigma=\mathsf{iss}_{213}$.
The density $$\varsigma_R(A)=\sup_{\mu_1\in P(G)}\sup_{\mu_2\in P_\w(G)}\inf_{\mu_3\in P_\w(G)}\mu_2*\mu_3*\mu_1(A)=\sup_{\mu_1\in P_\w(G)}\sup_{\mu_2\in P(G)}\inf_{\mu_3\in P_\w(G)}\mu_1*\mu_3*\mu_2(A),$$considered in Section~\ref{s4}  can be shortly written as $\varsigma_R=\mathsf{Ssi}_{231}=\mathsf{sSi}_{132}$.
By analogy we can define densities $\mathsf{Si}_{12}$ and $\mathsf{Si}_{21}$ letting
$$\mathsf{Si}_{12}(A)=\sup_{\mu_1\in P(G)}\inf_{\mu_2\in P_\w(G)}\mu_1*\mu_2(A)\mbox{ \ and \ }
\mathsf{Si}_{21}(A)=\sup_{\mu_1\in P(G)}\inf_{\mu_2\in P_\w(G)}\mu_2*\mu_1(A)\mbox{ \ for $A\subset G$}.
$$
Theorems~\ref{t3.3nn} and \ref{subadit-sol} imply that
$$\mathsf{is}_{12}=\mathsf{Si}_{21}\le\mathsf{Ssi}_{231}\le\mathsf{iss}_{213}\mbox{ \ \ and \ \ }
\mathsf{is}_{21}=\mathsf{Si}_{12}\le\mathsf{Ssi}_{132}\le\mathsf{iss}_{312}
.$$

These observations suggest the following definition. Given a positive integer number $n\in\IN$, a  function $\mathsf e:\{1,\dots,n\}\to\{\inf,\sup,\Inf,\Sup\}$ with $\big|\big\{i\in\{1,\dots,n\}:\mathsf{e}(i)\in\{\Inf,\Sup\}\big\}\big|\le 1$ and a permutation $s:\{1,\dots,n\}\to\{1,\dots,n\}$ define
a density $\mathsf e_s:\mathcal P(G)\to[0,1]$ by the formula
$$\mathsf e_s(A)=\underset{\mu_1\in P_1(G)}{\mathsf e(1)}\cdots
\underset{\mu_n\in P_n(G)}{\mathsf e(n)}\mu_{s(1)}*\cdots*\mu_{s(n)}(A)\mbox{ \ for \ }A\subset G$$where $$P_i(G)=\begin{cases}P_\w(G)&\mbox{if $\mathsf{e}(i)\in\{\inf,\sup\}$}\\
P(G)&\mbox{if $\mathsf{e}(i)\in\{\Inf,\Sup\}$}
\end{cases}\mbox{ \ \ for $1\le i\le n$}.
$$
The density $\mathsf e_s$ will be called the {\em extremal density} generated by the function $\mathsf e$ and the substitution $s$. To shorten the notations, we shall write $\mathsf{i}$, $\mathsf{s}$, $\mathsf{I}$, $\mathsf{S}$ instead of $\inf$, $\sup$, $\Inf$, $\Sup$, respectively, and identify the functions $\mathsf e$ and $s$ with the sequences $(\mathsf e(1),\dots,\mathsf e(n))$ and $(s(1),\dots,s(m))$ or even words $\mathsf e(1)\cdots \mathsf e(n)$ and $s(1)\cdots s(m)$. In these notations, we get  $\sigma_R=\mathsf{is}_{12}$, $\sigma_L=\mathsf{is}_{21}$,  $\sigma=\mathsf{iss}_{213}=\mathsf{iss}_{231}$, and $\varsigma_R=\mathsf{Ssi}_{231}$. Therefore, this paper was devoted to study and applications of the extremal densities $\mathsf{is}_{12}$, $\mathsf{is}_{21}$, and $\mathsf{iss}_{213}$.
The (subadditive) extremal density $\mathsf{sis}_{213}$ was used in the paper \cite{BPS} as an instrument for solving an invariant version of Protasov's Problem~\ref{prob-Kourov}. It can be shown that $\sigma_R=\mathsf{is}_{12}\le\hat\sigma_R\le\mathsf{sis}_{213}\le\mathsf{iss}_{213}=\sigma$.

Observe that the simplest extremal densities
$\mathsf{i}_1$ and $\mathsf{s}_1$ can be calculated by the formulas
$$\mathsf{i}_1(A)=\begin{cases}0&\mbox{if $A\ne G$}\\
1&\mbox{if $A=G$}\end{cases}
\mbox{ \ \ and \ \ } \mathsf{s}_1(A)=\begin{cases}0&\mbox{if $A=\emptyset$}\\
1&\mbox{if $A\ne\emptyset$}\end{cases}$$
implying that $\mathsf{i}_1$ and $\mathsf{s}_1$ are the smallest and largest densities on $G$,  respectively. Therefore, the Solecki densities $\sigma_R=\mathsf{is}_{12}$ and $\sigma_L=\mathsf{is}_{21}$ are the simplest nontrivial extremal densities in this hierarchy.
This suggests the following problem, or rather, a program of research.

\begin{problem} Study the properties of the extremal densities $\mathsf e_s$ on groups. Detect extremal densities which are subadditive. Study the interplay between various extremal densities on a group.
Find further applications of extremal densities in combinatorics of groups and $G$-spaces.
\end{problem}

\section{Acknowledgements}

The author would like to express his sincere thanks to Valerio Capraro, Ostap Chervak, Rostyslav Grigorchuk, Igor Protasov, Sergiy Slobodianiuk, Slawomir Solecki and Piotr Zakrzewski for valuable discussions related to the topic of this paper.
%\newpage


\begin{thebibliography}{}

\bibitem{Baer} R.~Baer, {\em Finiteness properties of groups}, Duke Math. J. {\bf 15} (1948), 1021--1032.

\bibitem{BGN} T.~Banakh, V.~Gavrylkiv, O.~Nykyforchyn, {\em Algebra in superextensions of groups, I: zeros and commutativity}, Algebra Discrete Math. (2008), no.3, 1--29.

\bibitem{BG} T.~Banakh, I.~Guran, {\em Perfectly supportable semigroups are $\sigma$-discrete in each Hausdorff shift-invariant topology}, Topological Algebra and Applications, {\bf 1} (2013) 1--8.

\bibitem{BGP} T.~Banakh, I.~Guran, I.~Protasov, {\em Algebraically determined topologies on permutation groups}, Topology Appl. {\bf 159}:9 (2012) 2258--2268.

\bibitem{BL} T.~Banakh, N.~Lyaskovska, {\em Completeness of translation-invariant ideals in groups}, Ukr. Mat. Zh. {\bf 62}:8 (2010), 1022--1031; transl. in Ukrainian Math. J. {\bf 62}:8 (2011) 1187--1198.

\bibitem{BL11} T.~Banakh, N.~Lyaskovska, {\em Constructing universally small subsets of a given packing index in Polish groups}, Colloq. Math. {\bf 125} (2011) 213--220.

\bibitem{BLR} T.~Banakh, N.~Lyaskovska, D.~Repovs, {\em Packing index of subsets in Polish groups}, Notre Dame J. Formal Logic., {\bf 50}:4 (2009) 453--468.

\bibitem{BPS} T.~Banakh, I.~Protasov, S.~Slobodianiuk, {\em On invariant partitions of groups}, preprint.

\bibitem{BRS} T.~Banakh, O.~Ravsky,  S.~Slobodianiuk,{\em On partitions of $G$-spaces and $G$-lattices}, preprint (http://arxiv.org/abs/1303.1427).

\bibitem{BJ} T.~Bartoszy\' nski, H.~Judah, {\em Set Theory: on the Structure of the Real Line}, Wellesley, MA, 1995.

\bibitem{BS} T.~Bartoszy\'nski, S.~Shelah, {\em Closed measure zero sets}, Ann. Pure Appl. Logic, {\bf 58}:2 (1992), 93--110.

\bibitem{BBF} M.~Beiglb\"ock, V.~Bergelson, A.~Fish, {\em Sumset phenomenon in countable amenable groups}, Adv. Math. {\bf 223}:2 (2010), 416--432.

%\bibitem{BM} A.~Bella, Malykhin,

\bibitem{BF} V.~Bergelson, H.~Furstenberg, {\em WM groups and Ramsey theory}. Topology Appl. {\bf 156}: 16 (2009), 2572--2580.

\bibitem{BJM} J.~Berglund, H.~Junghenn, P.~Milnes, {\em Analysis on Semigroups. Function Spaces, Compactifications, Representations}, A Wiley-Intersci. Publ. John Wiley \&\ Sons, Inc., New York, 1989.

\bibitem{Bergman} G.~Bergman, {\em Generating infinite symmetric groups}, Bull. London Math. Soc. {\bf 38} (2006) 429--440.

\bibitem{Blass} A.~Blass, {\em Combinatorial Cardinal
Characteristics of the Continuum}, in: Handbook of Set Theory (Eds.: M. Foreman, A. Kanamori),
Springer Science + Business Media B.V., 2010.

\bibitem{CR} M.~Cotlar, R.~Ricabarra, {\em On the existence of characters in topological groups},
Amer. J. Math. {\bf 76} (1954), 375--388.

\bibitem{NL} M.~Di Nasso, M.~Lupini, {\em Product sets and Delta-sets in amenable groups}, preprint
(http://arxiv.org/abs/1211.4208).

\bibitem{DS} S.~Dierolf, U.~Schwanengel, {\em Un exemple d'un groupe topologique $Q$-minimal mais non pr\'ecompact},
Bull. Sci. Math. (2) {\bf 101}:3 (1977), 265--269.

\bibitem{EK} R.~Ellis, H.~Keynes, {\em
Bohr compactifications and a result of F\o lner},
Israel J. Math. {\bf 12} (1972), 314--330.

\bibitem{Emerson} W.~Emerson, {\em Characterizations of amenable groups}, Trans. Amer. Math. Soc. {\bf 241} (1978), 183--194.

\bibitem{Erde} J.~Erde, {\em A Note on Combinatorial Derivation}, preprint (http://arxiv.org/abs/1210.7622).

\bibitem{En} R.~Engelking, {\em General Topology}, Sigma Series in Pure Mathematics, 6. Heldermann Verlag, Berlin, 1989.

\bibitem{Fed} V.V.~Fedorchuk, {\em Functors of probability measures in topological categories}, J. Math. Sci. (New York) {\bf 91}:4 (1998), 3157--3204.

\bibitem{Folner} E.~F\o lner, {\em
Generalization of a theorem of Bogoliouboff to topological abelian groups. With an appendix on Banach mean values in non-abelian groups}, Math. Scand. {\bf 2} (1954), 5--18.

\bibitem{FK} H.~Furstenberg, Y.~Katznelson, {\em A density version of the Hales-Jewett theorem}, J. Anal. Math. {\bf 57} (1991), 64--119.

\bibitem{Gau} E.~Gaughan, {\em Group structures of infinite symmetric groups}, Proc. Nat. Acad. Sci. U.S.A. {\bf 58} (1967), 907--910.

\bibitem{GRS} R.~Graham, B.~Rothschild, J.~Spencer, {\em Ramsey Theory}, A Wiley-Interscience Publ. John Wiley \&\  Sons, Inc., New York, 1990.

\bibitem{GT} B.~Green, T.~Tao, {\em The primes contain arbitrarily long arithmetic progressions},
{\bf 167} (2008), 481--547. I

%\bibitem{GGRC} I.~Guran, O.~Gutik, O.~Ravsky, I.~Chuchman, {\em On symmetric topologiacl semigroups and groups}, Visnyk Lviv Univ. Ser. Mech. Math. {\bf 74} (2011), 61--73 (in Ukrainian).

\bibitem{HW} G.H. Hardy, E.M.~Wright, A.~Wiles, {\em An Introduction to the Theory of Numbers}, Oxford University Press, 2008.

\bibitem{HR} E.~Hewitt, K.A.~Ross, {\em Abstract Harmonic Analysis}, Vol. I, Springer Verlag, 1963.

\bibitem{HS} N.~Hindman, D.~Strauss, {\em Density in arbitrary semigroups}, Semigroup Forum {\bf 73}:2  (2006), 273--300.

\bibitem{HM} K.~Hofmann, S.~Morris, {\em The structure of Compact Groups. A primer for the student -— a handbook for the expert}, Walter de Gruyter \&\ Co., Berlin, 2006.

\bibitem{Jarai} A.~J\'arai, {\em Regularity Properties of Functional Equations in Several Variables}, Springer, 2005.

\bibitem{Jech72} T.~Jech, {\em Some combinatorial problems concerning uncountable cardinals}, Ann. Math. Logic 5 (1972/73), 165--198.

\bibitem{Jech} T.~Jech, {\em Stationary sets}. in: Handbook of set theory. Vol.1, 93--128, Springer, Dordrecht, 2010.

\bibitem{Jin} R.~Jin, {\em The sumset phenomenon}, Proc. Amer. Math. Soc. {\bf 130}:3 (2002), 855--861.

\bibitem{Jin2} R.~Jin, {\em Standardizing nonstandard methods for upper Banach density problems}. Unusual applications of number theory, 109--124, DIMACS Ser. Discrete Math. Theoret. Comput. Sci., 64, Amer. Math. Soc., Providence, RI, 2004.

\bibitem{Ke} A.~Kechris, {\em Classical Descriptive Set Theory}, Springer-Verlag, New York, 1995.

\bibitem{Kelly} J.L.~Kelley, {\em Measures on Boolean algebras}, Pacific J. Math. {\bf 9} (1959),  1165--1177.


\bibitem{LP} Ie.~Lutsenko, I.V.~Protasov, {\em Sparse, thin and other subsets of groups}, Internat. J. Algebra Comput. {\bf 19}:4 (2009), 491--510.

\bibitem{Lyas} N.~Lyaskovska, {\em Constructing subsets of a given packing index in abelian groups}, Acta Univ. Carolin. Math. Phys. {\bf 48}:2 (2007), 69--80.

\bibitem{Kourov} V.D.~Mazurov, E.I.~Khukhro, (eds.) {\em Unsolved problems in group theory: the Kourovka notebook}, Thirteenth augmented edition. Russian Academy of Sciences Siberian Division, Institute of Mathematics, Novosibirsk, 1995. 120 pp.

\bibitem{Namioka} I.~Namioka, {\em F\o lner's conditions for amenable semi-groups},
Math. Scand. {\bf 15} (1964), 18--28.

%\bibitem{Neumann51} B.H.~Neumann, {\em Groups with finite classes of conjugate elements}, Proc. London Math. Soc. {\bf 1} (1951), 178--187.

\bibitem{Neumann} B.H.~Neumann, {\em Groups covered by permutable subsets}, J. London Math. Soc. {\bf 29} (1954), 236--248.

\bibitem{vN} J.~von Neumann, {\em Invariant Measures}, Amer. Math. Soc., Providence, RI, 1999.

\bibitem{Pat} A.~Paterson, {\em Amenability}, Math. Surveys and Monographs. {\bf 29}, Amer. Math. Soc. Providece, RI, 1988.

\bibitem{Pestov} V.~Pestov, {\em Forty-plus annotatted questions about large topological groups}, in: Open Problems in Topology, II (E.Pearl ed.), 439--450, Elsevier, 2007.

\bibitem{Polya} G.~Polya, {\em  Untersuchungen \"uber L\"ucken und Singularit\"ten von Potenzreihen}, Math. Z. {\bf 29}:1 (1929),  549--640.

\bibitem{Prot} I.V.~Protasov, {\em Selective survey on subset combinatorics of groups}, Ukr. Mat. Visn. {\bf 7}:2 (2010), 220--257; transl. in J. Math. Sci. (N. Y.) {\bf 174}:4 (2011), 486--514.

\bibitem{Protas} I.V.~Protasov, {\em Thin subsets of topological groups}, preprint.

\bibitem{ProtD} I.V.~Protasov, {\em Combinatorial derivation}, preprint (http://arxiv.org/abs/1210.0696).

\bibitem{PB} I.V.~Protasov, T.~Banakh, {\em Ball Structures and Colorings of Graphs and Groups}, VNTL Publ., Lviv, 2003.

\bibitem{Rob} D.~Robinson, {\em A course in the theory of groups}, Springer-Verlag, New York, 1996.

\bibitem{Sol} S.~Solecki, {\em Size of subsets of groups and Haar null sets}, Geom. Funct. Anal. {\bf 15}:1 (2005), 246--273.

\bibitem{Sol2001} S.~Solecki, {\em On Haar null sets}, Fund. Math. {\bf 149} (1996), 205--210.

\bibitem{Szemeredi} E.~Szemer\'edi, {\em On sets of integers containing no $k$ elements in arithmetic progression}, Acta Arith. {\bf 27} (1975), 199--245.

\bibitem{Var} V.S.~Varadarajan, {\em Measures on topological spaces}, Mat. Sb. (N.S.) {\bf 55} (1961) 35--100.

\bibitem{Veech} W.~Veech, {\em Topological dynamics}, Bull. Amer. Math. Soc. {\bf 83}:5 (1977)  775--830.

\bibitem{Zak} P.~Zakrzewski, {\em On the complexity of the ideal of absolute null sets}, Ukrainian Math. J. {\bf 64}:2 (2012), 306--308.

\end{thebibliography}
\end{document}